\documentclass[11pt]{amsart}
\usepackage[margin=1in]{geometry}

\usepackage{amssymb}
\usepackage{amsthm}
\usepackage{amsmath}
\usepackage{mathrsfs}
\usepackage{asymptote}
\usepackage{amsbsy}
\usepackage[all]{xy}
\usepackage{bm}
\usepackage{hyperref}
\usepackage{tikz}
\usepackage{array}
\usepackage{float}
\usepackage{enumerate}
\usepackage{xcolor}
\usepackage{hhline}
\usepackage{mathrsfs}
\usepackage{pbox}
\allowdisplaybreaks
\usepackage[noadjust]{cite}
\usepackage[nobysame]{amsrefs}

\definecolor{lavender}{rgb}{0.5,0,1.0}

\usepackage[noabbrev,capitalise,nameinlink]{cleveref}
\crefname{conjecture}{Conjecture}{Conjectures}

\hypersetup{colorlinks=true, citecolor=lavender, linkcolor=lavender,urlcolor=lavender}

\newcommand{\ST}{\operatorname{ST}}

\newcommand{\cyc}{\operatorname{cyc}}

\DeclareMathOperator{\Pop}{\mathsf{Pop}}

\newenvironment{enumerate*}
  {\begin{enumerate}[(I)]
    \setlength{\itemsep}{10pt}
    \setlength{\parskip}{0pt}}
  {\end{enumerate}}

\newtheorem{theorem}{Theorem}[section]
\newtheorem{proposition}[theorem]{Proposition}
\newtheorem{corollary}[theorem]{Corollary}
\newtheorem{conjecture}[theorem]{Conjecture}

\newtheorem{lemma}[theorem]{Lemma}

\theoremstyle{definition}
\newtheorem{definition}[theorem]{Definition}
\newtheorem{remark}[theorem]{Remark}
\newtheorem{example}[theorem]{Example}

\DeclareMathOperator{\ddeg}{ddeg}
\DeclareMathOperator{\NC}{NC}
\DeclareMathOperator{\ShNC}{ShNC}
\DeclareMathOperator{\HNC}{HNC}
\DeclareMathOperator{\Krew}{Krew}
\DeclareMathOperator{\Tam}{Tam}
\DeclareMathOperator{\row}{Row}

\DeclareMathOperator{\Rot}{Rot}
\DeclareMathOperator{\SSS}{SS}
\DeclareMathOperator{\Cat}{Cat}
\DeclareMathOperator{\Orb}{Orb}

\DeclareMathOperator{\N}{N}
\DeclareMathOperator{\E}{E}
\DeclareMathOperator{\rk}{rk}

\DeclareMathOperator{\JJ}{\mathscr{J}}
\DeclareMathOperator{\MM}{\mathscr{M}}
 
\newcommand{\Inv}{\mathrm{Inv}}

\newcommand{\dfn}[1]{\textcolor{blue}{\emph{#1}}}

\begin{document}

\title[]{Rowmotion on $m$-Tamari and BiCambrian Lattices}
\subjclass[2010]{}

\author[]{Colin Defant}
\address[]{Department of Mathematics, Harvard University, Cambridge, MA 02139, USA}
\email{colindefant@gmail.com}
\author[]{James Lin}
\address[]{Department of Mathematics and Computer Science, Massachusetts Institute of Technology, Cambridge, MA 02139, USA}
\email{jameslin@mit.edu}

\begin{abstract}
Thomas and Williams conjectured that rowmotion acting on the rational $(a,b)$-Tamari lattice has order $a+b-1$. We construct an equivariant bijection that proves this conjecture when $a\equiv 1\pmod b$; in fact, we determine the entire orbit structure of rowmotion in this case, showing that it exhibits the cyclic sieving phenomenon. We additionally show that the down-degree statistic is homomesic for this action. In a different vein, we consider the action of rowmotion on Barnard and Reading's biCambrian lattices. Settling a different conjecture of Thomas and Williams, we prove that if $c$ is a bipartite Coxeter element of a coincidental-type Coxeter group $W$, then the orbit structure of rowmotion on the $c$-biCambrian lattice is the same as the orbit structure of rowmotion on the lattice of order ideals of the doubled root poset of type $W$. 
\end{abstract}

\maketitle

\bigskip

\section{Introduction}\label{Sec:Intro} 

The growing field of \emph{dynamical algebraic combinatorics} seeks to understand dynamical systems arising from operators defined on objects in algebraic combinatorics. One of the most vigorously-studied operators in this area is \dfn{rowmotion}, which is an operator $\row$ that acts on the set $\mathcal J(P)$ of order ideals of a finite poset $P$. More precisely, $\row\colon \mathcal J(P)\to\mathcal J(P)$ is defined by\footnote{Several authors define rowmotion to be the inverse of the operator that we have defined. Our definition agrees with the conventions used in \cite{Barnard, Semidistrim, ThomasWilliams}.} \[\row(I)=\{x\in P:x\not\geq y\text{ for all }y\in\max(I)\},\] where $\max(I)$ is the set of maximal elements of $I$. When ordered by inclusion, the order ideals of $P$ form a lattice; in fact, Birkhoff's Fundamental Theorem of Finite Distributive Lattices \cite{Birkhoff} states that a finite lattice is distributive if and only if it is isomorphic to the lattice of order ideals of a finite poset. Thus, one can view rowmotion as an operator defined on a distributive lattice. 

Another very influential operator is \emph{Kreweras complementation}, which Kreweras introduced in 1972 in its most basic form in \cite{Kreweras} as a natural anti-automorphism of the lattice of noncrossing set partitions of the set $[n]=\{1,\ldots,n\}$. More generally, if one chooses a Coxeter element $c$ of a finite irreducible Coxeter group $W$, then one can consider the \dfn{noncrossing partition lattice} $\NC(W,c)$, which is the interval $[e,c]$ in the absolute order on $W$. In this setting, \dfn{Kreweras complementation} is the operator that sends $w$ to $w^{-1}c$. Associated to the pair $(W,c)$ is a \emph{Cambrian lattice}, which is a lattice on the set of $c$-sortable elements of $W$ \cite{ReadingCambrian}. Using Reading's bijection between $\NC(W,c)$ and the set of $c$-sortable elements of $W$ \cite{ReadingClusters}, one can view Kreweras complementation as an operator on a Cambrian lattice. 

In 2006, Thomas \cite{Thomas} introduced \emph{trim lattices} as generalizations of distributive lattices that are not necessarily graded. Another natural family of lattices that generalize distributive lattices is the family of \emph{semidistributive lattices}. All distributive lattices and finite Cambrian lattices are both trim and semidistributive. Barnard \cite{Barnard} gave a natural definition of rowmotion on semidistributive lattices, and Thomas and Williams \cite{ThomasWilliams} gave a natural definition of rowmotion on trim lattices; in recent work, the first author and Williams provided a simultaneous generalization of both of these definitions to the even broader class of \emph{semidistrim lattices} \cite{Semidistrim}. These definitions generalize that of classical rowmotion on distributive lattices and that of Kreweras complementation on finite Cambrian lattices.

Prototypical examples of Cambrian lattices are provided by the \emph{Tamari lattices}, which Tamari originally introduced in 1962 \cite{Tamari}. These lattices, which are often defined on Dyck paths, have grown in prominence to become fundamental objects in algebraic combinatorics \cite{Muller}. In 2012, Bergeron and Pr\'eville-Ratelle \cite{Bergeron} defined generalizations of Tamari lattices called \emph{$m$-Tamari lattices} in order to state conjectural combinatorial interpretations of the dimensions of certain spaces arising from the study of trivariate diagonal harmonics. We denote the $n$-th $m$-Tamari lattice by $\Tam_n(m)$. These lattices have now received a great deal of attention (see \cite{BousquetRep, BousquetIntervals, Chatel} and the references therein). 

The $m$-Tamari lattices are special examples of \emph{rational Tamari lattices}. Each rational Tamari lattice is trim and semidistributive, so it comes equipped with a natural rowmotion operator. Thomas and Williams conjectured a formula for the order of rowmotion on rational Tamari lattices; one of our main results settles this conjecture for $m$-Tamari lattices by showing that the order of $\row\colon\Tam_n(m)\to\Tam_n(m)$ is $(m+1)n$. In fact, we will determine the entire orbit structure of rowmotion on $\Tam_n(m)$, showing that it exhibits the cyclic sieving phenomenon with respect to the polynomial $\Cat_n^{(m)}(q)=\frac{1}{[(m+1)n+1]_q}{(m+1)n+1\brack n}_q$. Moreover, we will conjecture that rowmotion on an arbitrary rational Tamari lattice exhibits a similar cyclic sieving phenomenon. 

A popular and ubiquitous notion in dynamical algebraic combinatorics is that of \emph{homomesy}, which occurs when a statistic on a set of objects has the same average along every orbit of an operator. The \dfn{down-degree statistic} on a poset $P$ is the function $\ddeg\colon P\to\mathbb R$ given by $\ddeg(x)=\lvert\{y\in P:y\lessdot x\}\rvert$. The articles \cite{AST, ProppRoby, RushWang} prove that $\ddeg$ is homomesic for rowmotion on certain distributive lattices, and the article \cite{HopkinsCDE} establishes the homomesy of $\ddeg$ for rowmotion on certain semidistributive lattices. As a byproduct of our methods, we will show that this story extends into the realm of $m$-Tamari lattices. Namely, we will prove that the down-degree statistic on $\Tam_n(m)$ is homomesic for rowmotion with average $m(n-1)/(m+1)$. We will also conjecture that $\ddeg$ on the rational Tamari lattice $\Tam(a,b)$ is homomesic for rowmotion with average $(a-1)(b-1)/(a+b-1)$. 

In a different direction, we will consider rowmotion operators arising from Coxeter-biCatalan combinatorics, which Barnard and Reading introduced in \cite{BarnardReading}. Each finite irreducible Coxeter group of coincidental type has an associated \emph{doubled root poset} that is minuscule, and the orbit structure of rowmotion acting on the set of order ideals of that doubled root poset is already understood. In the same article, Barnard and Reading introduced the \emph{biCambrian lattice} associated to a pair $(W,c)$, where $W$ is a finite irreducible Coxeter group and $c$ is a Coxeter element of $W$. Each biCambrian lattice is semidistributive and, consequently, comes equipped with a rowmotion operator. Thomas and Williams \cite{ThomasWilliams} conjectured that if $W$ is a finite Coxeter group of coincidental type and $c$ is a bipartite Coxeter element of $W$, then the orbit structure of rowmotion on the biCambrian lattice associated to $(W,c)$ is the same as the orbit structure of rowmotion acting on the lattice of order ideals of the doubled root poset of $W$. Our second main result is a proof of this conjecture. 

In \cref{Sec:Background}, we provide background information concerning rowmotion, cyclic sieving, and homomesy. We will only need to define rowmotion for semidistributive lattices (and not trim lattices) because all of the specific lattices that we consider belong to this family. \cref{Sec:nu-Tamari} defines $\nu$-Tamari lattices, of which rational Tamari lattices are specific examples. In \cref{sec:describing}, we give an explicit description of rowmotion on $\nu$-Tamari lattices. In \cref{Sec:mTamari}, we specialize our attention to $m$-Tamari lattices, producing an equivariant bijection that transfers from rowmotion on $\Tam_n(m)$ to cyclic rotation on Armstrong's set of \emph{$(m+1)$-shuffle noncrossing partitions} of $[(m+1)n]$. This bijection allows us to prove the aforementioned cyclic sieving phenomenon and homomesy for rowmotion on $\Tam_n(m)$. In \cref{Sec:BiCambrianDoubled}, we shift gears and gather notions from Barnard and Reading's Coxeter-biCatalan combinatorics. Namely, we discuss lattice congruences, doubled root posets, and biCambrian lattices. \cref{Sec:BiCambrianA} is devoted to understanding rowmotion on biCambrian lattices of type $A$ with respect to bipartite Coxeter elements; we produce an equivariant bijection demonstrating that the orbit structure of rowmotion on such a biCambrian lattice is the same as the orbit structure of rowmotion on the lattice of order ideals of the associated doubled root poset of type $A$ (in this case, the doubled root poset is a product of two chains of the same length). In \cref{Sec:BiCambrianB}, we prove the analogous result in type $B$ (in this case, the doubled root poset is a shifted staircase). The remaining cases of the conjecture of Thomas and Williams are for types $H_3$ and $I_2(m)$; \cref{Sec:BiCambrianOther} handles these remaining (not so difficult) cases. Finally, \cref{Sec:Conclusion} mentions some ideas for future research. This includes generalizing our results in \cref{Sec:mTamari} on the cyclic sieving phenomenon and homomesy for rowmotion on $\Tam_n(m)$ to all rational Tamari lattices $\Tam(a,b)$. We also include two conjectures about the behavior of the down-degree statistic on the orbits of rowmotion on the even larger class of $\nu$-Tamari lattices. The first conjecture describes a phenomenon that we call \emph{asymptotic homomesy}; the second concerns the \emph{homometry phenomenon}, which was introduced recently by Elizalde, Plante, Roby, and Sagan \cite{ElizaldeFences}.

\section{Background}\label{Sec:Background}

\subsection{Basic Notation and Terminology}
Given positive integers $a$ and $b$, we write $[a,b]$ for the set $\{a,\ldots,b\}$ of integers weakly between $a$ and $b$, where we make the convention that this set is empty if $a>b$. We also write $[b]$ instead of $[1,b]$. Given a finite set $X$ and an invertible map $f\colon X\to X$, we define the \dfn{order} of $f$ to be the smallest positive integer $\omega$ such that $f^\omega(x)=x$ for all $x\in X$. When we refer to the \dfn{orbit structure} of $f$, we mean the multiset of the sizes of its orbits. 

\subsection{Posets and Lattices}\label{subsec:posets}
We assume basic familiarity with the theory of posets, as discussed in \cite[Chapter~3]{Stanley2}. All posets throughout this article are assumed to be finite and connected (meaning their Hasse diagrams are finite connected graphs). Suppose $P$ is a poset. For $x,y\in P$, we say $y$ \dfn{covers} $x$ and write $x\lessdot y$ to mean that $x<y$ and there does not exist $z\in P$ with $x<z<y$. We say $P$ is \dfn{ranked} if there exists a function $\text{rk}\colon P\to\mathbb Z_{\geq 0}$ such that $\rk(y)=\rk(x)+1$ whenever $x\lessdot y$ and such that there exists $x_0\in P$ with $\rk(x_0)=0$. Note that the rank function $\rk$ is unique if it exists (because $P$ is connected). 

A \dfn{lattice} is a poset $L$ such that any two elements $x,y\in L$ have a unique greatest lower bound, called their \dfn{meet} and denoted $x\wedge y$, and a unique least upper bound, called their \dfn{join} and denoted $x\vee y$. Given a subset $X$ of a finite lattice, we write $\bigwedge X$ for the meet of $X$ (i.e., the greatest lower bound of all elements of $X$). 

An element $j$ in a lattice $L$ is called \dfn{join-irreducible} if it covers exactly $1$ element of $L$. Similarly, an element $m\in L$ is called \dfn{meet-irreducible} if it is covered by exactly $1$ element of $L$. 

\subsection{Rowmotion on Semidistributive Lattices}\label{subsec:distributive}

Let $P$ be a finite poset, and let $\mathcal J(P)$ be the set of order ideals of $P$. As discussed in the introduction, rowmotion is the operator $\row\colon\mathcal J(P)\to\mathcal J(P)$ defined by declaring $\row(I)$ to be the complement in $P$ of the order filter generated by the maximal elements of $I$. The origins of rowmotion date back to the works of Duchet \cite{Duchet}, Brouwer \cite{Brouwer}, and Brouwer--Schriver \cite{BrouwerSchriver}. It was studied further in the 1990's by Deza and Fukada \cite{DezaFukada}, by Cameron and Fon-der-Flass \cite{Cameron}, and by Fon-der-Flass \cite{FonderFlass}. Panyushev \cite{Panyushev} revived interest in rowmotion via several conjectures about root posets that were later settled by Armstrong, Stump, and Thomas \cite{AST}. We refer to \cite{StrikerWilliams,ThomasWilliams} for more detailed historical accounts of rowmotion. 

A lattice $L$ is called \dfn{semidistributive} if the implications \[x\vee y=x\vee z\implies x\vee(y\wedge z)=x\vee y\quad\text{and}\quad x\wedge y=x\wedge z\implies x\wedge(y\vee z)=x\wedge y\] hold for all $x,y,z\in L$. An equivalent definition says that $L$ is semidistributive if and only if for all $x,y\in L$ with $x\leq y$, the set $\{z\in L:z\wedge y=x\}$ has a unique maximal element and the set $\{z\in L:z\vee x=y\}$ has a unique minimal element. Every distributive lattice (i.e., every lattice of order ideals of a finite poset) is semidistributive. We refer the reader to \cite{Barnard, ThomasWilliams, Semidistrim} for the facts about semidistributive lattices that we discuss next. 

Suppose $L$ is a semidistributive lattice. Let $\JJ$ and $\MM$ be the set of join-irreducible elements of $L$ and the set of meet-irreducible elements of $L$, respectively. For each $j\in\JJ$, let $j_*$ be the unique element covered by $j$; for each $m\in\MM$, let $m^*$ be the unique element that covers $m$. There is a bijection $\kappa\colon\JJ\to\MM$ given by $\kappa(j)=\max\{z\in L:z\wedge j=j_*\}$ for all $j\in\JJ$; its inverse is given by $\kappa^{-1}(m)=\min\{z\in L:z\vee m=m^*\}$ for all $m\in\MM$. 

\begin{lemma}\label{lem:propertiesofkappa}
Let $L$ be a semidistributive lattice, and let $\kappa\colon\JJ\to\MM$ be the bijection defined above. We have $\kappa(j)\wedge j=j_*$ and $\kappa(j)\vee j=(\kappa(j))^*$ for all $j\in \JJ$. Furthermore, for each cover relation $x\lessdot y$ in $L$, there is a unique $j_{xy}\in\JJ$ such that $j_{xy}\leq y$ and $\kappa(j_{xy})\geq x$. 
\end{lemma}

The first statement in \cref{lem:propertiesofkappa} actually characterizes the bijection $\kappa$. For each $z\in L$, \cref{lem:propertiesofkappa} allows us to define the \dfn{downward label set} and the \dfn{upward label set} of $z$ by \[\mathcal D(z)=\{j_{xz}:x\lessdot z\}\quad\text{and}\quad\mathcal U(z)=\{j_{zy}:z\lessdot y\}.\] The set $\mathcal D(z)$ is also called the \dfn{canonical join representation} of $z$, while $\kappa(\mathcal U(z))$ is called the \dfn{canonical meet representation} of $z$. 

It turns out that for each $x\in L$, there is a unique element of $L$ whose upward label set is equal to the downward label set of $x$. We define the \dfn{rowmotion} operator on $L$ to be the bijective operator $\row\colon L\to L$ given by \[\mathcal U(\row(x))=\mathcal D(x).\] When $L$ is the lattice of order ideals of a finite poset, this definition agrees with the classical definition of rowmotion given above. This more general definition is originally due to Barnard \cite{Barnard}, who denoted rowmotion by $\overline\kappa$ because it extends the map $\kappa$. Indeed, one can show that $\row(j)=\kappa(j)$ for every $j\in\JJ$. 

Given a lattice $L$ with minimum element $\widehat 0$, the \dfn{pop-stack sorting operator} on $L$ is the operator $\Pop\colon L\to L$ defined by  $\Pop(x)=\bigwedge\{y\in L:y\lessdot x\}$ for $x\neq\widehat 0$ and by $\Pop(\widehat 0)=\widehat 0$. This operator was introduced in \cite{DefantMeeting}, and it was found to have strong connections with rowmotion in \cite{Semidistrim} in the case when $L$ is semidistrim. In particular, if $L$ is semidistributive, then the pop-stack sorting operator affords a useful alternative description of rowmotion. 

\begin{proposition}[{\cite[Section~9]{Semidistrim}}]\label{prop:pop1}
If $L$ is a semidistributive lattice and $x\in L$, then $\row(x)$ is the unique maximal element of the set $\{z\in L:z\wedge x=\Pop(x)\}$. 
\end{proposition}

\subsection{The Cyclic Sieving Phenomenon}\label{subsec:CSP}

The cyclic sieving phenomenon, which Reiner, Stanton, and White introduced in \cite{CSPDefinition}, occurs when the orbit structure of a finite dynamical system is determined by the evaluations of a specific relatively nice generating function at roots of unity. To make this more precise, let $X$ be a finite set, and let $f\colon X\to X$ be an invertible map of order $\omega$ (i.e., $\omega$ is the smallest positive integer such that $f^\omega(x)=x$ for all $x\in X$). Let $F(q)\in\mathbb C[q]$ be a polynomial in the variable $q$. We say the triple $(X,f,F(q))$ \dfn{exhibits the cyclic sieving phenomenon} if for all integers $d$, the number of elements of $X$ fixed by $f^d$ is $F(e^{2\pi id/\omega})$. Notice that this forces $F(1)$ to be the cardinality of $X$. The cyclic sieving phenomenon has been observed in many settings, including the relevant setting of noncrossing partitions in \cite{BodnarRhoades}.

\subsection{Homomesy and Homometry}\label{Sec:Homomesy}

Suppose $X$ is a finite set and $f\colon X\to X$ is an invertible function. A \dfn{statistic} on $X$ is simply a real-valued function on $X$. We say a statistic $\text{stat}\colon X\to\mathbb R$ is \dfn{homomesic} for $f$ if there exists a constant $c$ such that \[\frac{1}{|\mathcal O|}\sum_{x\in\mathcal O}\text{stat}(x)=c\] for every orbit $\mathcal O$ of $f$. Propp and Roby coined the term \emph{homomesy} in \cite{ProppRoby}. This phenomenon has now become one of the central focuses of dynamical algebraic combinatorics because of its surprising ubiquity and its ability to encapsulate interesting structural aspects of combinatorial dynamical systems. We refer the reader to \cite{Roby} for a very accessible survey of homomesy. 

The very recent article \cite{ElizaldeFences} introduces a new generalization of homomesy called \emph{homometry}. We say a statistic $\text{stat}\colon X\to \mathbb R$ is \dfn{homometric} for $f$ if \[\sum_{x\in\mathcal O}\text{stat}(x)=\sum_{x'\in\mathcal O'}\text{stat}(x')\] whenever $\mathcal O$ and $\mathcal O'$ are orbits of $f$ of the same cardinality. Note that homomesy implies homometry, but not the other way around. In \cref{Sec:Conclusion}, we will conjecture that the down-degree statistic is homometric for rowmotion on an arbitrary $\nu$-Tamari lattice.

\section{$\nu$-Tamari Lattices}\label{Sec:nu-Tamari}

\subsection{Lattice Paths}\label{subsec:LatticePaths} 
For us, a \dfn{lattice path} is a finite path in the plane that starts at the origin and uses unit north steps and unit east steps. We denote north and east steps by $\text{N}$ and $\text{E}$, respectively, thereby identifying a lattice path with a finite word over the alphabet $\{\text{N},\text{E}\}$. We use exponents to denote concatenation of a word with itself; for instance, $(\text{NE}^3)^2=\text{NEEENEEE}$. Given a lattice path $\nu$, let $\Tam(\nu)$ be the set of lattice paths lying weakly above $\nu$ that have the same endpoints as $\nu$. For example, if $\nu=\text{ENNEEEENNE}$ is the lattice path shown in black in \cref{FigMeet2}, then the path $\mu=\text{NENENEEENE}$ shown in red is in $\Tam(\nu)$. 

\begin{figure}[ht]
  \begin{center}{\includegraphics[height=4.2cm]{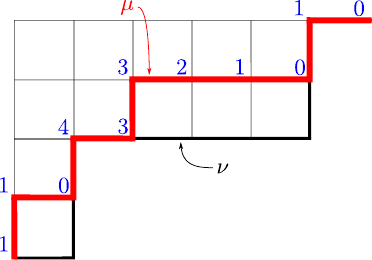}}
  \end{center}
  \caption{The lattice path $\mu=\text{NENENEEENE}$ in the $\nu$-Tamari lattice $\Tam(\nu)$, where $\nu=\text{ENNEEEENNE}$. Each lattice point $p$ is labeled with its horizontal distance. }\label{FigMeet2}
\end{figure}

Suppose $\mu\in\Tam(\nu)$. Given a lattice point $p$ on $\mu$, define the \dfn{horizontal distance} of $p$ to be the maximum number of east steps that can be taken, starting at $p$, before crossing $\nu$; we illustrate this definition in \cref{FigMeet2}. Now suppose $p$ is a lattice point on $\mu$ that is preceded by an east step and followed by a north step in $\mu$. Let $p'$ be the first lattice point on $\mu$ that appears after $p$ and has the same horizontal distance as $p$. If we let $D_{[p,p']}$ be the subpath of $\mu$ that starts at $p$ and ends at $p'$, then we can write $\mu=X\text{E}D_{[p,p']}Y$ for some lattice paths $X$ and $Y$. Let $\mu'=XD_{[p,p']}\text{E}Y$. Then $\mu'\in \Tam(\nu)$. In this case, we write $\mu\lessdot \mu'$. See \cref{FigMeet3}. 

\begin{figure}[ht]
  \begin{center}{\includegraphics[height=3.923cm]{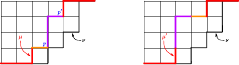}}
  \end{center}
  \caption{The lattice path $\mu$ on the left is covered by the lattice path $\mu'$ on the right in $\Tam(\nu)$.}\label{FigMeet3}
\end{figure}

The relations $\mu\lessdot\mu'$ described in the previous paragraph form the cover relations of a partial order $\leq$ on $\Tam(\nu)$. Pr\'eville-Ratelle and Viennot \cite{PrevilleViennot} introduced this partial order and proved that it is a lattice. In fact, they proved that $\Tam(\nu)$ is an interval of a Tamari lattice, and it is well known that Tamari lattices are semidistributive. Because intervals of semidistributive lattices are semidistributive, it follows that $\Tam(\nu)$ is semidistributive. This means that we have a rowmotion operator $\row\colon\Tam(\nu)\to\Tam(\nu)$. 

When $\nu=(\text{NE}^m)^n$, the lattice $\Tam(\nu)$ is precisely the $n$-th \dfn{$m$-Tamari lattice}, which we denote by $\Tam_n(m)$. When $\nu=(\text{NE})^n$, $\Tam(\nu)$ is the classical $n$-th \dfn{Tamari lattice}. Given relatively prime positive integers $a$ and $b$, let $\Tam(a,b)$ be the set of lattice paths from $(0,0)$ to $(a,b)$ that lie strictly above the line $y=(b/a)x$ except at their endpoints. We have $\Tam(a,b)=\Tam(\nu_{a,b})$ for an appropriately-chosen lattice path $\nu_{a,b}$, so we can view $\Tam(a,b)$ as a lattice. When $a=mn+1$ and $b=n$, we have $\nu_{a,b}=(\text{N}\text{E}^m)^n\text{E}$, and it is straightforward to check that $\Tam(a,b)$ is isomorphic to $\Tam_n(m)$. Thomas and Williams stated the following conjecture. 

\begin{conjecture}[\cite{ThomasWilliams}]\label{conj:ThomasWilliams}
If $a$ and $b$ are relatively prime positive integers, then the rowmotion operator $\row\colon\Tam(a,b)\to\Tam(a,b)$ has order $a+b-1$. 
\end{conjecture}

In \cref{cor:mainnuTamari}, we will determine the exact orbit structure of rowmotion on $\Tam_n(m)$, thereby settling and extending \cref{conj:ThomasWilliams} in the case where $a=mn+1$ and $b=n$. We will also formulate a stronger version of \cref{conj:ThomasWilliams} for arbitrary coprime $a$ and $b$ in \cref{THING2}.

\subsection{$\nu$-Bracket Vectors}

There is a very useful alternative description of $\nu$-Tamari lattices due to Ceballos, Padrol, and Sarmiento \cite{Ceballos2}. In what follows, we make the convention that vectors are $0$-indexed. For example, we would say that the number $8$ appears in positions $0$ and $4$ in the vector $(8,3,6,5,8,6)$. Our convention is that $\mathsf b_i$ always denotes the entry in position $i$ of a vector $\mathsf b$. 

\begin{definition}\label{DefMeet2}
Fix a lattice path $\nu$ that starts at $(0,0)$ and ends at $(\ell-n,n)$. Let us write ${{\bf b}(\nu)=(b_0(\nu),\ldots,b_{\ell}(\nu))}$ for the vector obtained by reading the heights (i.e., $y$-coordinates) of the lattice points on $\nu$ in the order they appear in $\nu$. For $0\leq k\leq n$, let $f_k$ be the maximum index such that $b_{f_k}(\nu)=k$. We call $f_0,\ldots,f_n$ the \dfn{fixed positions} of $\nu$. A \dfn{$\nu$-bracket vector} is an integer vector $\mathsf{b}=(\mathsf b_0,\ldots,\mathsf b_{\ell})$ such that:
\begin{enumerate}[(I)]
\item $\mathsf b_{f_k}=k$ for all $k\in[0,n]$; 
\item\label{Item2} $b_i(\nu)\leq \mathsf b_i\leq n$ for all $i\in[0,\ell]$;
\item\label{Item3} if $\mathsf b_i=k$, then $\mathsf b_j\leq k$ for all $j\in[i+1, f_k]$. 
\end{enumerate} 
\end{definition}

\begin{remark}\label{rem:121}
    It is often useful to note that condition~\eqref{Item3} in \cref{DefMeet2} can be replaced with the condition that $\mathsf b$ avoids the pattern $121$; this condition means that there do not exist indices $i_1<i_2<i_3$ such that $\mathsf b_{i_1}=\mathsf b_{i_3}<\mathsf b_{i_2}$. 
\end{remark}

\begin{remark}\label{rem:decreasing}
If $\mathsf b$ is a $\nu$-bracket vector, then ${\mathsf b_{f_k+1}\geq\mathsf b_{f_k+2}\geq\cdots\geq\mathsf b_{f_{k+1}}}$ for each $k\in[-1,n-1]$, where we make the convention $f_{-1}=-1$. 
\end{remark}

Define the \dfn{componentwise} partial order $\leq$ on the set of $\nu$-bracket vectors by the condition that $\mathsf b\leq\mathsf b'$ if $\mathsf b_i\leq\mathsf b_i'$ for all $i\in[0,\ell]$. It is straightforward to check that the componentwise minimum $\min(\mathsf{b},\mathsf{b}')=(\min\{\mathsf b_0,\mathsf b_0'\},\ldots,\min\{\mathsf b_{\ell},\mathsf b_{\ell}'\})$ of two $\nu$-bracket vectors $\mathsf b$ and $\mathsf b'$ is also a $\nu$-bracket vector. In fact, the set of $\nu$-bracket vectors forms a lattice in which the meet operation is given by the componentwise minimum: $\mathsf b\wedge\mathsf b'=\min(\mathsf b,\mathsf b')$. Ceballos, Padrol, and Sarmiento \cite{Ceballos2} showed that this lattice is isomorphic to $\Tam(\nu)$ (using an explicit bijection from the set of $\nu$-bracket vectors to $\Tam(\nu)$ ). Therefore, by abuse of notation, we will henceforth write $\Tam(\nu)$ to refer to the lattice of $\nu$-bracket vectors with the componentwise order. Similarly, $\Tam_n(m)$ will denote the lattice of $\nu$-bracket vectors when $\nu=(\text{N}\text{E}^m)^n$. This convention should not lead to any confusion since we will no longer be dealing with lattice paths. The Hasse diagram of $\Tam_3(2)$ appears in \cref{FigRowTam2}.

\section{Describing Rowmotion on $\nu$-Tamari Lattices}\label{sec:describing}
Let us fix a lattice path $\nu$ that starts at $(0,0)$ and ends at $(\ell-n,n)$. Let $f_0,\ldots,f_n$ be the fixed positions of $\nu$. As before, $\Tam(\nu)$ denotes the $\nu$-Tamari lattice, which we view as the lattice of $\nu$-bracket vectors under the componentwise order. Recall that there is a rowmotion operator $\row\colon\Tam(\nu)\to\Tam(\nu)$ because $\Tam(\nu)$ is semidistributive.

Recall that $\mathsf b_i$ always denotes the entry in position $i$ of $\mathsf b$. Let $\mathsf b\leftarrow_it$ be the vector obtained from $\mathsf b$ by replacing the entry in position $i$ with $t$. For instance, \[(0,1,1,2,2,3,4,5)\leftarrow_35=(0,1,1,5,2,3,4,5).\] Recall that for $k\in[0,n]$, every $\nu$-bracket vector has the entry $k$ in position $f_k$. For $\mathsf{b}\in\Tam(\nu)$ and $0\leq i\leq \ell$, we are going to define integers $\zeta_i(\mathsf b)$ and $\eta_i(\mathsf b)$; the primary motivation for these definitions is that they will allow us to describe rowmotion on $\Tam(\nu)$ via a step-by-step procedure. The $i$-th step in this procedure will either do nothing or replace the entry in a $\nu$-bracket vector $\mathsf b$ with $\zeta_i(\mathsf b)$ or $\eta_i(\mathsf b)$. 

For $\mathsf{b}\in\Tam(\nu)$ and $0\leq i\leq \ell$, let
\[\zeta_i(\mathsf{b})=\min\{\mathsf b_j:j\in[0,i-1]\text{ and }\mathsf b_j\geq \mathsf b_i\},\] with the convention that $\zeta_i(\mathsf b)=n$ if $\{\mathsf b_j:j\in[0,i-1]\text{ and }\mathsf b_j\geq \mathsf b_i\}=\emptyset$. Let \[\Delta(\mathsf b)=\{i\in[0,\ell-1]:\mathsf b_i>\mathsf b_{i+1}\}.\] Suppose $i\in\Delta(\mathsf b)$, and let $k=\mathsf b_{i+1}$. We have $b_i(\nu)=b_{i+1}(\nu)\leq k\leq \mathsf b_i-1$, and it follows from condition~\eqref{Item3} in \cref{DefMeet2} that $\mathsf b_j\leq k$ for all $j\in[i+1,f_{k}]$. Therefore, we can define \[\eta_i(\mathsf b)=\max\{h\in[b_i(\nu),\mathsf b_i-1]:\mathsf b_j\leq h\text{ for all }j\in[i+1,f_h]\}.\] On the other hand, if $i'\in[0,\ell]\setminus\Delta(\mathsf{b})$, then we let $\eta_{i'}(\mathsf{b})=\mathsf{b}_{i'}$.

In the following proposition, we write $\Pop$ for the pop-stack sorting operator on $\Tam(\nu)$ (see \cref{subsec:distributive}).

\begin{proposition}[{\cite[Proposition~4.4]{DefantMeeting}}]\label{prop:pop2}
For each $\mathsf b\in \Tam(\nu)$, we have \[\Pop(\mathsf{b})=(\eta_0(\mathsf{b}),\ldots,\eta_\ell(\mathsf{b})).\]    
\end{proposition}

The following theorem provides a description of rowmotion on $\Tam(\nu)$; in order to state it, we define maps $\Theta_r\colon\Tam(\nu)\to\Tam(\nu)$ according to the following rules: 
\begin{enumerate}
\item If $r\in\{f_0,\ldots,f_n\}$, then $\Theta_r(\mathsf b)=\mathsf b$ for all $\mathsf b\in\Tam(\nu)$.
\item If $r\in\Delta(\mathsf b)$, then $\Theta_r(\mathsf b)=\mathsf b\leftarrow_r\eta_r(\mathsf b)$.
\item If $r\not\in\{f_0,\ldots,f_n\}$ and $r\not\in\Delta(\mathsf{b})$, then $\Theta_r(\mathsf b)=\mathsf b\leftarrow_r\zeta_r(\mathsf b)$.
\end{enumerate}
It is straightforward to check that $\Theta_r$ does in fact send $\nu$-bracket vectors to $\nu$-bracket vectors. Indeed, $(2)$ says that if $\mathsf b_r$ is followed by a descent, then it is decreased by at least $1$, and the amount by which it is decreased is as small as possible for $\Theta_r(\mathsf b)$ to be a $\nu$-bracket vector. On the other hand, $(3)$ says that if $\mathsf b_r$ is not followed by a descent (and $r$ is not a fixed position), then it is increased as much as possible for $\Theta_r(\mathsf b)$ to be a $\nu$-bracket vector. 

\begin{theorem}\label{thm:rowmotionbracket}
For $\mathsf b\in\Tam(\nu)$, we have \[\row(\mathsf b)=(\Theta_\ell\circ\cdots\circ\Theta_1\circ\Theta_0)(\mathsf b).\]
\end{theorem}

The computation of rowmotion in \cref{thm:rowmotionbracket} is similar to the computation in \emph{slow motion} that Thomas and Williams introduced for trim lattices in \cite{ThomasWilliams}. However, our step-by-step procedure is slightly different because we have added additional steps that always ``do nothing'' (applying $\Theta_r$ when $r$ is a fixed position) and we have combined some of the steps from the description in \cite{ThomasWilliams} (thereby ``speeding up'' parts of the ``slow motion''). 

\begin{example}\label{ex:rowmotionexample}
Let $\nu=(\text{N}\text{E}^2)^3$ so that $\Tam(\nu)=\Tam_3(2)$. Let us compute the image of the $\nu$-bracket vector $\mathsf b=(0,2,2,1,2,2,2,3,3,3)$ under rowmotion using \cref{thm:rowmotionbracket}. To ease notation, let \[\mathsf b^{(\alpha)}=(\Theta_\alpha\circ\cdots\circ\Theta_0)(\mathsf b).\] First, $\mathsf b^{(0)}=\Theta_0(\mathsf b)=\mathsf b$ since $0=f_0$. Next, note that $1\not\in\Delta(\mathsf b^{(0)})$ and $\zeta_1(\mathsf b^{(0)})=n=3$. Thus, \[\mathsf b^{(1)}=\Theta_1(\mathsf b^{(0)})=\mathsf b^{(0)}\leftarrow_1\zeta_1(\mathsf b^{(0)})=\mathsf b^{(0)}\leftarrow_13=(0,3,2,1,2,2,2,3,3,3).\] Since $2\in\Delta(\mathsf b^{(1)})$ and $\eta_2(\mathsf b^{(1)})=1$, we have \[\mathsf b^{(2)}=\Theta_2(\mathsf b^{(1)})=\mathsf b^{(1)}\leftarrow_2\eta_2(\mathsf b^{(1)})=\mathsf b^{(1)}\leftarrow_21=(0,3,1,1,2,2,2,3,3,3).\] Now $3=f_1$, so $\mathsf b^{(3)}=\Theta_3(\mathsf b^{(2)})=\mathsf b^{(2)}$. Since $4\not\in\Delta(\mathsf b^{(3)})$ and $\zeta_4(\mathsf b^{(3)})=3$, we have \[\mathsf b^{(4)}=\Theta_4(\mathsf b^{(3)})=\mathsf b^{(3)}\leftarrow_4\zeta_4(\mathsf b^{(3)})=\mathsf b^{(3)}\leftarrow_43=(0,3,1,1,3,2,2,3,3,3).\] Since $5\not\in\Delta(\mathsf b^{(4)})$ and $\zeta_5(\mathsf b^{(4)})=3$, we have \[\mathsf b^{(5)}=\Theta_5(\mathsf b^{(4)})=\mathsf b^{(4)}\leftarrow_5\zeta_5(\mathsf b^{(4)})=\mathsf b^{(4)}\leftarrow_53=(0,3,1,1,3,3,2,3,3,3).\] Then $6=f_2$, so $\mathsf b^{(6)}=\Theta_6(\mathsf b^{(5)})=\mathsf b^{(5)}$. Since $7\not\in\Delta(\mathsf b^{(6)})$, we have \[\mathsf b^{(7)}=\mathsf b^{(6)}\leftarrow_7\zeta_7(\mathsf b^{(6)})=\mathsf b^{(6)}\leftarrow_7 3=\mathsf b^{(6)}=\mathsf b^{(5)}.\] Since $8\not\in\Delta(\mathsf b^{(7)})$, we have \[\mathsf b^{(8)}=\mathsf b^{(7)}\leftarrow_8\zeta_8(\mathsf b^{(7)})=\mathsf b^{(7)}\leftarrow_8 3=\mathsf b^{(7)}=\mathsf b^{(5)}.\] Finally, $9=f_3$, so \[\row(\mathsf b)=\mathsf b^{(9)}=\Theta_9(\mathsf b^{(8)})=\mathsf b^{(8)}=(0,3,1,1,3,3,2,3,3,3).\]
This fact is illustrated by the pink arrow from $(0,2,2,1,2,2,2,3,3,3)$ to $(0,3,1,1,3,3,2,3,3,3)$ in \cref{FigRowTam2}.
\end{example}

\begin{figure}[ht]
  \begin{center}{\includegraphics[width=\linewidth]{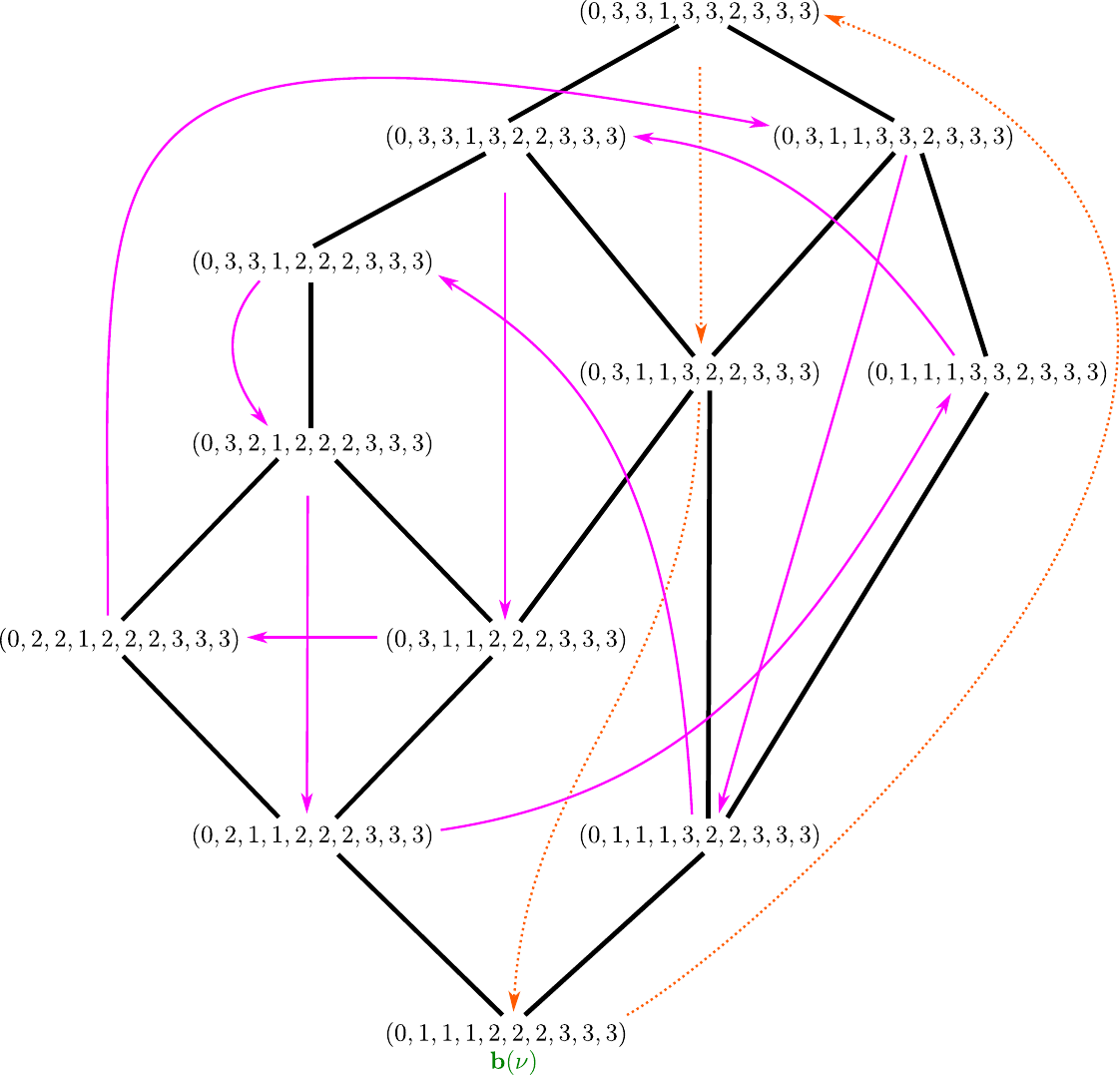}}
  \end{center}
  \caption{Rowmotion on $\Tam_3(2)$. The solid pink and dotted orange arrows indicate the action of rowmotion. Notice that there is one orbit of size $9$ (solid pink) and one orbit of size $3$ (dotted orange).}\label{FigRowTam2}
\end{figure}

\begin{proof}[Proof of \cref{thm:rowmotionbracket}]
Fix $\mathsf{b}\in\Tam(\nu)$, and let $\row(\mathsf{b})=\widehat{\mathsf b}=(\widehat{\mathsf b}_0,\ldots,\widehat{\mathsf b}_\ell)$. For each $\alpha\in[0,\ell]$, let $\mathsf b^{(\alpha)}=(\Theta_{\alpha}\circ\cdots\circ\Theta_0)(\mathsf b)$; let us also write ${\mathsf b}^{(-1)}=\mathsf b$. We claim that ${\mathsf b}^{(\alpha)}=(\widehat{\mathsf b}_0,\ldots,\widehat{\mathsf b}_{\alpha},\mathsf b_{\alpha+1},\ldots,\mathsf b_\ell)$ for all $\alpha\in[-1,\ell]$; once we prove this claim, setting $\alpha=\ell$ will yield the desired result. 

We will prove the desired claim by induction on $\alpha$; it is certainly true when $\alpha=-1$, so we may assume that $\alpha\in[0,\ell]$ and that ${\mathsf b}^{(\alpha-1)}=(\widehat{\mathsf b}_0,\ldots,\widehat{\mathsf b}_{\alpha-1},\mathsf b_{\alpha},\ldots,\mathsf b_\ell)$. Since $\Theta_\alpha$ can only change the entry in position $\alpha$ of a $\nu$-bracket vector, we just need to show that the entry in position $\alpha$ of the vector $\Theta_\alpha(\mathsf{b}^{(\alpha-1)})={\mathsf b}^{(\alpha)}$ is $\widehat b_{\alpha}$. If $\alpha=f_k$ for some $k\in[0,n]$, then every $\nu$-bracket vector has the entry $k$ in position $\alpha$, so ${\mathsf b}^{(\alpha)}_\alpha=\widehat b_{\alpha}=k$, as desired. 

Now suppose $\alpha\in\Delta(\mathsf b)$. In this case, the definition of $\Theta_\alpha$ ensures that ${\mathsf b}^{(\alpha)}_\alpha=\eta_\alpha(\mathsf b^{(\alpha-1)})$. Since $\mathsf b^{(\alpha-1)}$ agrees with $\mathsf b$ in positions $\alpha,\alpha+1,\ldots,\ell$, we have $\eta_\alpha(\mathsf b^{(\alpha-1)})=\eta_\alpha(\mathsf b)$.
Note that $\eta_\alpha(\mathsf b)<\mathsf b_\alpha$. \cref{prop:pop2} tells us that $\eta_\alpha(\mathsf b)$ is the entry in position $\alpha$ of $\Pop(\mathsf b)$, and \cref{prop:pop1} tells us that $\Pop(\mathsf b)=\row(\mathsf b)\wedge\mathsf b=\widehat{\mathsf b}\wedge\mathsf b$. Since the meet of two $\nu$-bracket vectors is their componentwise minimum, this shows that $\widehat{\mathsf b}_\alpha=\eta_\alpha(\mathsf b)={\mathsf b}^{(\alpha)}_\alpha$.   

Finally, assume that $\alpha\not\in\{f_0,\ldots,f_n\}$ and $\alpha\not\in\Delta(\mathsf b)$. In this case, the definition of $\Theta_\alpha$ ensures that ${\mathsf b}^{(\alpha)}_\alpha=\zeta_\alpha(\mathsf b^{(\alpha-1)})$ is the minimum element of $\{\widehat {\mathsf b}_j:j\in[0,\alpha-1]\text{ and }\widehat {\mathsf{b}}_j\geq\mathsf b_\alpha\}$ (or is $n$ if this set is empty). Let $k=\zeta_\alpha(\mathsf b^{(\alpha-1)})$; we want to prove that $\widehat {\mathsf b}_\alpha=k$. 

Suppose by way of contradiction that $\widehat {\mathsf b}_\alpha>k$. This implies that $k\neq n$, so there exists $j\in[0,\alpha-1]$ such that $\widehat{\mathsf b}_j=k\geq\mathsf b_\alpha$. Then $f_k\geq f_{\mathsf b_\alpha}>\alpha$ (the last inequality is strict because $\alpha\not\in\{f_0,\ldots,f_n\}$). Since $\widehat{\mathsf b}_\alpha>k$, the entries $\widehat{\mathsf b}_j,\widehat{\mathsf b}_\alpha,\widehat{\mathsf b}_{f_k}$ form a $121$-pattern in $\widehat{\mathsf b}$, contradicting \cref{rem:121}. This shows that $\widehat{\mathsf b}_\alpha\leq k$.  

Now suppose by way of contradiction that $\widehat {\mathsf b}_\alpha<k$. Let $\mathsf b'=\widehat{\mathsf b}\leftarrow_\alpha k$. We claim that $\mathsf b'$ is a $\nu$-bracket vector. It is straightforward to check that $\mathsf b'$ satisfies the first two conditions in \cref{DefMeet2}, so (by \cref{rem:121}) we just need to check that $\mathsf b'$ does not contain a $121$-pattern. Assume instead that $\mathsf b'$ contains a $121$-pattern in positions $i_1<i_2<i_3$. Let $k'=\mathsf b'_{i_3}$, and note that $i_3\leq f_{k'}$ and $\mathsf b'_{f_{k'}}=k'$. Therefore, $\mathsf b'$ contains a $121$-pattern in positions $i_1<i_2<f_{k'}$. Since $\mathsf b'$ agrees with the $\nu$-bracket vector $\widehat{\mathsf b}$ except in position $\alpha$, one of $i_1,i_2,f_{k'}$ must be $\alpha$ (otherwise, $\widehat{\mathsf b}$ would contain a $121$-pattern). We have assumed that $\alpha\not\in\{f_0,\ldots,f_n\}$, so either $\alpha=i_1$ or $\alpha=i_2$. If $\alpha=i_1$, then $k=\mathsf b'_\alpha<\mathsf b'_{i_2}\leq n$, so (by the definition of $k$) there exists $j\in[0,\alpha-1]$ such that $\widehat{\mathsf b}_j=k=\widehat{\mathsf b}_{\alpha}$. In this case, $\widehat{\mathsf b}$ has a $121$-pattern in positions $j<i_2<f_{k'}$, which is a contradiction. On the other hand, if $\alpha=i_2$, then the $\nu$-bracket vector \[\mathsf b^{(\alpha)}=\Theta_\alpha(\mathsf b^{(\alpha-1)})=\mathsf b^{(\alpha-1)}\leftarrow_\alpha k\] agrees with $\mathsf b'$ in positions $i_1$, $i_2$, and $f_{k'}$. However, this means that $\mathsf b^{(\alpha)}$ contains a $121$-pattern in positions $i_1<i_2<f_{k'}$, which is impossible. This proves our claim that $\mathsf b'$ is a $\nu$-bracket vector. Now, $\min\{\mathsf b'_\alpha,\mathsf b_\alpha\}=\min\{k,\mathsf b_\alpha\}=\mathsf b_\alpha=\eta_\alpha(\mathsf b)$, where the last equality comes from the fact that $\alpha\not\in\Delta(\mathsf b)$. For $i\in[0,\ell]\setminus\{\alpha\}$, we can use \cref{prop:pop1,prop:pop2} to see that \[\min\{\mathsf b'_i,\mathsf b_i\}=\min\{\widehat{\mathsf b}_i,\mathsf b_i\}=(\Pop(\mathsf b))_i=\eta_i(\mathsf b).\] This proves that $\mathsf b'\wedge\mathsf b=(\eta_0(\mathsf b),\ldots,\eta_\ell(\mathsf b))=\Pop(\mathsf b)$ (using \cref{prop:pop2} for the last equality), which contradicts \cref{prop:pop1} because $\mathsf b'>\widehat{\mathsf b}$.  
\end{proof}

The following corollary is immediate from the proof of \cref{thm:rowmotionbracket}. 

\begin{corollary}\label{cor:delta}
If $\mathsf b\in\Tam(\nu)$ and $i\in\Delta(\mathsf b)$, then $(\row(\mathsf b))_i=\eta_i(\mathsf b)<\mathsf b_i$. 
\end{corollary}

Before we proceed, let us record a simple lemma that will be useful later. 

\begin{lemma}\label{lem:etazthing}
Suppose $\mathsf b\in\Tam(\nu)$ and $i\in\Delta(\mathsf b)$, and let $z=\eta_i(\mathsf b)$. We have $\mathsf b_{f_{z}+1}=\mathsf b_i$. 
\end{lemma}

\begin{proof}
Let $k=f_z+1$ and $h=\mathsf b_k$. Note that $h\geq b_k(\nu)=z+1$. The definition of $\eta_i(\mathsf b)$ tells us that $\mathsf b_j\leq z$ for all $j\in[i+1,f_z]$. We have $\mathsf b_{f_z+1}=h$, and condition~\eqref{Item3} in \cref{DefMeet2} implies that $\mathsf b_j\leq h$ for all $j\in[f_z+2,f_h]$. Thus, $\mathsf b_j\leq h$ for all $j\in[i+1,f_{h}]$. It follows from the maximality in the definition of $\eta_i(\mathsf b)$ that $h$ cannot satisfy $z<h<\mathsf b_i$. We know that $z<h$, so $h\geq\mathsf b_i$. If we had $h>\mathsf b_i$, then the entries of $\mathsf b$ in positions $i,k,f_{b_i}$ would form a $121$ pattern, contradicting \cref{rem:121}. 
\end{proof}

\section{$m$-Tamari Lattices}\label{Sec:mTamari}

In this section, we fix positive integers $m$ and $n$ and focus our attention on the $m$-Tamari lattice $\Tam_n(m)$, which is the lattice $\Tam(\nu)$ for $\nu=(\text{N}\text{E}^m)^n$. It is a classical result \cite{Dvoretzky} that the cardinality of $\Tam_n(m)$ is the \dfn{Fuss--Catalan number} \[\Cat^{(m)}_n=\frac{1}{(m+1)n+1}\binom{(m+1)n+1}{n}.\] Let ${[k]_q!=[k]_q[k-1]_q\cdots[1]_q}$, where $[k]_q=\frac{1-q^k}{1-q}=1+q+\cdots+q^{k-1}$ is the $q$-analogue of the integer $k$. The \dfn{$q$-binomial coefficient} ${k\brack r}_q$ is defined to be $\dfrac{[k]_q!}{[r]_q![k-r]_q!}$. There is a natural $q$-analogue of the Fuss--Catalan number $\Cat_n^{(m)}$ given by \[\Cat^{(m)}_n(q)=\frac{1}{[(m+1)n+1]_q}{(m+1)n+1\brack n}_q.\] We are going to completely describe the orbit structure of $\row\colon\Tam_n(m)\to\Tam_n(m)$ by showing that the triple $(\Tam_n(m),\row,\Cat_n^{(m)}(q))$ exhibits the cyclic sieving phenomenon (see \cref{subsec:CSP}). 

Consider a partition $\rho$ of a totally ordered set $X$. We say two distinct blocks $B_1$ and $B_2$ in $\rho$ form a \dfn{crossing} if there exist $i_1,j_1\in B_1$ and $i_2,j_2\in B_2$ such that either $i_1<i_2<j_1<j_2$ or $i_1>i_2>j_1>j_2$. We say $\rho$ is \dfn{noncrossing} if no two distinct blocks in $\rho$ form a crossing. Let $\NC(k)$ be the set of noncrossing partitions of $[k]$. 

We are going to make use of the \dfn{Kreweras complementation} map $\Krew\colon\NC(k)\to\NC(k)$. Given a noncrossing partition $\rho$ of the set $[k]$, let $K'(\rho)$ be the coarsest partition of the set $\{1',\ldots,k'\}$ such that $\rho\cup K'(\rho)$ is a noncrossing partition of the totally ordered set $\{1<1'<2<2'<\cdots<k<k'\}$. Then let $\Krew(\rho)$ be the noncrossing partition of $[k]$ obtained from $K'(\rho)$ by removing the primes from the elements of $\{1',\ldots,k'\}$. For example, if \[\rho=\{\{1,2,3\},\{4,8,12\},\{5,6,7\},\{9,10,11\}\}\in\NC(12)\] is as shown in pink in \cref{FigRowTam6}, then $K'(\rho)$ is the partition shown in orange in \cref{FigRowTam6}, so \[\Krew(\rho)=\{\{1\},\{2\},\{3,12\},\{4,7\},\{5\},\{6\},\{8,11\},\{9\},\{10\}\}\in\NC(12).\] Kreweras complementation is a bijection. 

\begin{figure}[ht]
  \begin{center}{\includegraphics[height=6.5cm]{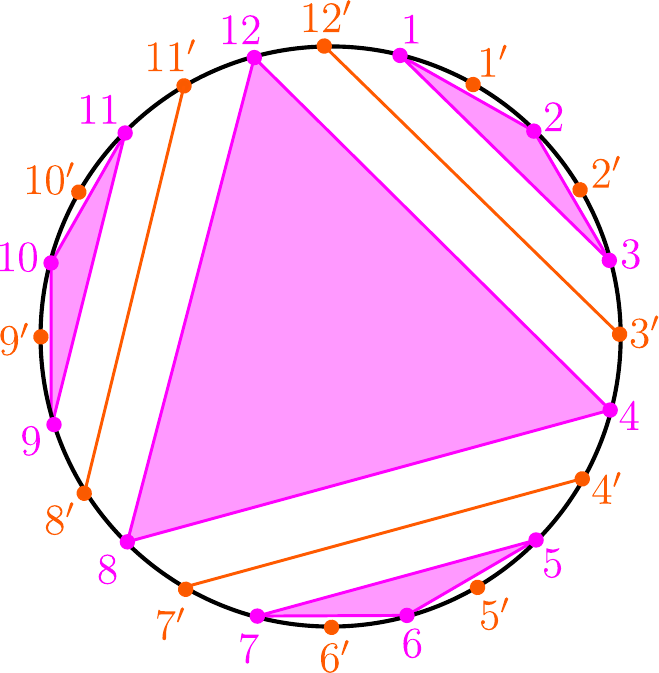}}
  \end{center}
  \caption{A noncrossing partition $\rho\in\NC(12)$ (pink) together with $K'(\rho)$ (orange).}\label{FigRowTam6}
\end{figure}

Following \cite{BodnarRhoades}  we say a partition $\rho\in\NC((m+1)n)$ is \dfn{$(m+1)$-homogeneous} if each of its blocks has size $m+1$ (such partitions are sometimes called \dfn{$(m+1)$-equal}). Let $\HNC^{(m)}(n)$ denote the set of $(m+1)$-homogeneous noncrossing partitions of $[(m+1)n]$. Following \cite{Armstrong}, we say $\rho\in\NC((m+1)n)$ is \dfn{$(m+1)$-shuffle} if it has the property that $i\equiv j\pmod{m+1}$ whenever $i$ and $j$ are in the same block of $\rho$. Let $\ShNC^{(m)}(n)$ denote the set of $(m+1)$-shuffle noncrossing partitions of $[(m+1)n]$ that have exactly $mn+1$ blocks. One can readily check (see also \cite[Section~4.3]{Armstrong}) that \[\Krew(\HNC^{(m)}(n))=\ShNC^{(m)}(n)\quad\text{and}\quad\Krew(\ShNC^{(m)}(n))=\HNC^{(m)}(n).\] For example, the partition $\rho$ shown in pink in \cref{FigRowTam6} is $3$-homogeneous, while its Kreweras complement $\Krew(\rho)$ is $3$-shuffle. 

There is a natural \dfn{rotation} operator $\Rot\colon\NC((m+1)n)\to\NC((m+1)n)$ that acts by cyclically rotating the blocks of a partition one space to the left. More precisely, for $\rho\in\NC((m+1)n)$, two integers $i$ and $j$ lie in the same block of $\rho$ if and only if $i-1$ and $j-1$ lie in the same block of $\Rot(\rho)$, where the integers are taken modulo $(m+1)n$. It is well known that $\Rot=\Krew^2$. Notice that $\Rot$ preserves each of the sets $\HNC^{(m)}(n)$ and $\ShNC^{(m)}(n)$. 

Bodnar and Rhoades \cite{BodnarRhoades} proved that the triple $(\HNC^{(m)}(n),\Rot,\Cat_n^{(m)}(q))$ exhibits the cyclic sieving phenomenon. Since $\Krew$ is a bijection from $\HNC^{(m)}(n)$ to $\ShNC^{(m)}(n)$ satisfying \[\Rot\circ\Krew=\Krew\circ\Rot,\] it follows that \begin{equation}\label{Eq9}
(\ShNC^{(m)}(n),\Rot,\Cat_n^{(m)}(q))\text{ exhibits the cyclic sieving phenomenon.}
\end{equation}
Therefore, in order to prove the main result of this section, it suffices to construct a bijection $\Psi\colon\Tam_n(m)\to\ShNC^{(m)}(n)$ satisfying $\Psi\circ\row=\Rot\circ\Psi$.

Now let $\mathsf b\in\Tam_n(m)$ be a $\nu$-bracket vector, where
$\nu=(\text{N}\text{E}^m)^n$. The fixed positions $f_0,\ldots,f_n$ of $\nu$ are given by $f_k=(m+1)k$, so $\mathsf b_{(m+1)k}=k$ for all $k\in[0,n]$. Furthermore, $b_{(m+1)s+t}(\nu)=s+1$ for all $s\in[0,n-1]$ and $t\in[m+1]$. Recall from \cref{sec:describing} that $\Delta(\mathsf b)$ is the set of indices $i\in[0,(m+1)n-1]$ such that $\mathsf b_i>\mathsf b_{i+1}$; it follows from \cref{DefMeet2} that none of the fixed positions $(m+1)k$ belong to $\Delta(\mathsf b)$. Let us say an index $i\in[(m+1)n]$ is \dfn{$k$-initial} in $\mathsf b$ if $\mathsf b_i=k$ and $\mathsf b_{j}\neq k$ for all $j\in[0,i-1]$. 

If $\rho$ is a partition of the set $[(m+1)n]$ and $i\in[(m+1)n]$ is an integer that is not maximal in its block in $\rho$, we define the \dfn{successor} of $i$ in $\rho$ to be the smallest integer that lies in the same block as $i$ and is greater than $i$. We are going to define the set partition $\Psi(\mathsf b)$ by specifying successors of various integers. To begin, we define a function $\varphi_{\mathsf b}\colon[(m+1)n]\to[(m+1)n]\cup\{\emptyset\}$. (To motivate this definition, the reader may wish to look ahead at \cref{lem:Psi1} to see how $\varphi_{\mathsf{b}}$ will be used in defining $\Psi(\mathsf b)$.) Suppose $i\in[(m+1)n]$; there are three cases to consider. 

\medskip

\noindent {\bf Case 1.} If $i<(m+1)n$ is a multiple of $m+1$ and $i+1$ is $k$-initial in $\mathsf b$ for some $k$, then we define $\varphi_{\mathsf b}(i)=f_k=(m+1)k$. 

\medskip

\noindent {\bf Case 2.} Suppose $i\in\Delta(\mathsf b)$ (in particular, $i\not\equiv 0\pmod{m+1}$). In this case, we define $\varphi_{\mathsf b}(i)$ to be the smallest integer $j\in[(m+1)n]$ satisfying $j>i$, $j\equiv i\pmod{m+1}$, and $\mathsf b_i=\mathsf b_{j}$. To see that such an integer exists, let us write $i=(m+1)s+t$, where $0\leq s\leq n-1$ and $1\leq t\leq m$. If we let $k=\mathsf b_{(m+1)s+t}$, then (by the definition of $\Delta(\mathsf b)$) we have $k>\mathsf b_{(m+1)s+t+1}\geq b_{(m+1)s+t+1}(\nu)=s+1$. Then $(m+1)s+t<(m+1)(k-1)+t<(m+1)k=f_k$, so it follows from condition~\eqref{Item3} in \cref{DefMeet2} that $\mathsf b_{(m+1)(k-1)+t}\leq k$. We also have $\mathsf b_{(m+1)(k-1)+t}\geq b_{(m+1)(k-1)+t}(\nu)=k$, so $\mathsf b_{(m+1)(k-1)+t}=k$. This shows that $(m+1)(k-1)+t$ satisfies the conditions defining $j$ (though it might not be the smallest integer with these properties).

\medskip

\noindent {\bf Case 3.} If we are not in Case~1 or Case~2, then we simply define $\varphi_{\mathsf b}(i)=\emptyset$. 

\medskip

\begin{example}\label{Exam1}
Suppose $m=2$, $n=4$, and $\mathsf b=(0,4,1,1,3,3,2,3,3,3,4,4,4)$. Then 
$\varphi_{\mathsf b}(1)=10$, $\varphi_{\mathsf b}(3)=9$, $\varphi_{\mathsf b}(5)=8$, and $\varphi_{\mathsf b}(i)=\emptyset$ for all $i\in[12]\setminus\{1,3,5\}$. 
\end{example}

\begin{lemma}\label{lem:injective}
The map $\varphi_{\mathsf b}\colon[(m+1)n]\to[(m+1)n]\cup\{\emptyset\}$ is injective on $[(m+1)n]\setminus\varphi_{\mathsf b}^{-1}(\emptyset)$. 
\end{lemma}

\begin{proof}
Suppose that $i,i'\in[(m+1)n]\setminus\varphi_{\mathsf b}^{-1}(\emptyset)$ are such that $\varphi_{\mathsf b}(i)=\varphi_{\mathsf b}(i')$. Without loss of generality, assume $i\leq i'$. It follows from the definition of $\varphi_{\mathsf b}$ that $i\equiv \varphi_{\mathsf b}(i)\equiv i'\pmod{m+1}$. If $i\equiv i'\equiv 0\pmod{m+1}$ (i.e., $i$ and $i'$ both fall into Case~1 in the definition of $\varphi_{\mathsf b}$), then $\varphi_{\mathsf b}(i)=\varphi_{\mathsf b}(i')=(m+1)k$ for some $k$. In this case, $i+1$ and $i'+1$ are both $k$-initial in $\mathsf b$, so we must have $i=i'$. If $i\equiv i'\not\equiv 0\pmod{m+1}$ (i.e., $i$ and $i'$ both fall into Case~2 in the definition of $\varphi_{\mathsf b}$), then $\mathsf b_{i}=\mathsf b_{i'}=\mathsf b_{\varphi_{\mathsf b}(i)}$ and $i\leq i'<\varphi_{\mathsf b}(i)$, so the minimality in the definition of $\varphi_{\mathsf b}(i)$ forces us to have $i=i'$. 
\end{proof}

\begin{lemma}\label{lem:Psi1}
Let $\mathsf b\in\Tam_n(m)$. There is a unique set partition $\Psi(\mathsf b)$ of $[(m+1)n]$ such that the following conditions hold for all $i\in[(m+1)n]$:
\begin{itemize}
    \item We have $\varphi_{\mathsf b}(i)=\emptyset$ if and only if $i$ is the largest element in its block in $\Psi(\mathsf b)$.
    \item If $\varphi_{\mathsf b}(i)\neq\emptyset$, then $\varphi_{\mathsf b}(i)$ is the successor of $i$ in $\Psi(\mathsf b)$.
\end{itemize}
\end{lemma}

\begin{proof}
Let $\Psi(\mathsf b)$ be the finest partition of $[(m+1)n]$ such that $i$ and $\varphi_{\mathsf b}(i)$ are in the same block for every $i\in[(m+1)n]\setminus\varphi_{\mathsf b}^{-1}(\emptyset)$. Choose a block $B$ of $\Psi(\mathsf b)$, and let $p=\min B$. Let $h$ be the smallest positive integer such that $\varphi_{\mathsf b}^h(p)=\emptyset$. The definition of $\varphi_{\mathsf b}$ guarantees that $\varphi_{\mathsf b}(j)>j$ for all $j\in[(m+1)n]\setminus\varphi_{\mathsf b}^{-1}(\emptyset)$. It follows that $B=\{p<\varphi_{\mathsf b}(p)<\varphi_{\mathsf b}^2(p)<\cdots<\varphi_{\mathsf b}^{h-1}(p)\}$. This proves that the largest element of $B$ is in $\varphi_{\mathsf b}^{-1}(\emptyset)$ and that $\varphi_{\mathsf b}(i)$ is the successor of $i$ for every non-maximal element $i$ of $B$. We have proven the existence of the partition $\Psi(\mathsf b)$ with the desired properties; the uniqueness is immediate.
\end{proof}

The preceding lemma provides us with a map $\Psi$ that sends $\nu$-bracket vectors in $\Tam_n(m)$ to set partitions of $[(m+1)n]$. 

\begin{example}
Let $m=2$, $n=4$, and $\mathsf b=(0,4,1,1,3,3,2,3,3,3,4,4,4)$; we computed the function $\varphi_{\mathsf b}$ in \cref{Exam1}. The set partition described in \cref{lem:Psi1} is \[\Psi(\mathsf b)=\{\{1,10\},\{2\},\{3,9\},\{4\},\{5,8\},\{6\},\{7\},\{11\},\{12\}\}. \qedhere\]
\end{example}

\begin{lemma}\label{lem:mn+1blocks}
If $\mathsf b\in\Tam_n(m)$, then $\Psi(\mathsf b)$ has exactly $mn+1$ blocks.
\end{lemma}

\begin{proof}
According to the description of $\Psi(\mathsf b)$ in \cref{lem:Psi1}, our goal is to show that \[\lvert\varphi_{\mathsf b}^{-1}(\emptyset)\rvert=mn+1.\] For each $k\in[n]$, let $g_k$ be the smallest integer such that $\mathsf b_{g_k}=k$ (i.e., $g_k$ is the unique integer that is $k$-initial in $\mathsf b$). Let $a\in[n]$ be the integer such that $g_{a}=1$. We will prove that \[\varphi_{\mathsf b}^{-1}(\emptyset)=[(m+1)n]\setminus\{g_k-1:k\in[n]\setminus \{a\}\}.\] To this end, we first observe that $(m+1)n\in\varphi_{\mathsf b}^{-1}(\emptyset)$.

Fix $k\in[n]$, and let $r$ be the number of occurrences of $k$ in $\mathsf b$. Let $i_1<\cdots<i_r$ be the indices such that $\mathsf b_{i_1}=\cdots=\mathsf b_{i_r}=k$; thus, $i_1=g_k$. Consider $t\in[2,r]$. If $i_t\equiv 1\pmod{m+1}$, then $\varphi_{\mathsf b}(i_t-1)=\emptyset$ because $i_t$ is not $k$-initial in $\mathsf b$. If $i_t\not\equiv 1\pmod{m+1}$, then the observation that $\mathsf b_{i_{t-1}}=\mathsf b_{i_t}$ implies (via \cref{rem:121}) that $\mathsf b_{i_t-1}\leq\mathsf b_{i_t}$, so $i_t-1\not\in\Delta(\mathsf b)$. In either case, $i_t-1\in\varphi_{\mathsf b}^{-1}(\emptyset)$. We now want to show that $i_1-1\not\in\varphi_{\mathsf b}^{-1}(\emptyset)$. This is obvious if $k=a$ (i.e., if $i_1-1=0$), so we may assume $k\neq a$. If $i_1\equiv 1\pmod{m+1}$, then $\varphi_{\mathsf b}(i_1-1)=f_k$ because $i_1$ is $k$-initial in $\mathsf b$, so $i_1-1\not\in\varphi_{\mathsf b}^{-1}(\emptyset)$. If $i_1\not\equiv 1\pmod{m+1}$, then we can use \cref{rem:decreasing} to see that $\mathsf b_{i_1-1}\geq\mathsf b_{i_1}$. In this case, we must actually have $\mathsf b_{i_1-1}>\mathsf b_{i_1}$ since $i_1$ is $k$-initial in $\mathsf b$, so $i_1-1\in\Delta(\mathsf b)$. This shows that $i_1-1\not\in\varphi_{\mathsf b}^{-1}(\emptyset)$ in this case as well.   
\end{proof}

\begin{proposition}\label{prop:ShNC}
If $\mathsf b\in\Tam_n(m)$, then $\Psi(\mathsf b)\in\ShNC^{(m)}(n)$.
\end{proposition}

\begin{proof}
The definition of $\varphi_{\mathsf b}$ ensures that $i\equiv\varphi_{\mathsf b}(i)\pmod{m+1}$ for all $i\in[(m+1)n]\setminus\varphi_{\mathsf b}^{-1}(\emptyset)$. Hence, it follows from the description of $\Psi(\mathsf b)$ given in \cref{lem:Psi1} that any two numbers in the same block of $\Psi(\mathsf b)$ must be congruent modulo $m+1$. We also know by \cref{lem:mn+1blocks} that $\Psi(\mathsf b)$ has exactly $mn+1$ blocks. We are left to prove that $\Psi(\mathsf b)$ is noncrossing. 

Suppose by way of contradiction that $\Psi(\mathsf b)$ is not noncrossing. This implies that there are indices $i_1<i_2<j_1<j_2$ such that $j_1$ is the successor of $i_1$ and $j_2$ is the successor of $i_2$. Then $j_1=\varphi_{\mathsf b}(i_1)$ and $j_2=\varphi_{\mathsf b}(i_2)$ by the description of $\Psi(\mathsf b)$ in \cref{lem:Psi1}. 

Let us first consider the case in which neither $i_1$ nor $i_2$ is divisible by $m+1$. This implies that $i_1,i_2\in\Delta(\mathsf b)$ and that $\mathsf b_{i_1}=\mathsf b_{j_1}$ and $\mathsf b_{i_2}=\mathsf b_{j_2}$. If we had $\mathsf b_{i_2}>\mathsf b_{i_1}$, then $\mathsf b_{i_1},\mathsf b_{i_2},\mathsf b_{j_1}$ would form a $121$-pattern in $\mathsf b$, contradicting \cref{rem:121}. If we had $\mathsf b_{i_2}<\mathsf b_{i_1}$, then $\mathsf b_{i_2},\mathsf b_{j_1},\mathsf b_{j_2}$ would form a $121$-pattern in $\mathsf b$, again contradicting \cref{rem:121}. Therefore, it must actually be the case that $\mathsf b_{i_1}=\mathsf b_{i_2}=\mathsf b_{j_1}=\mathsf b_{j_2}$. Let us write $i_1=(m+1)s_1+t_1$, $i_2=(m+1)s_2+t_2$, $j_1=(m+1)s_1'+t_1$, and $j_2=(m+1)s_2'+t_2$, where $0\leq s_1\leq s_2\leq s_1'\leq s_2'\leq n-1$ and $t_1,t_2\in[m]$. Because $i_1\in\Delta(\mathsf b)$, we have $\mathsf b_{(m+1)s_1+t_1}>\mathsf b_{(m+1)s_1+t_1+1}\geq\mathsf b_{(m+1)s_1+r}$ for all $r\in[t_1+1,m]$ by \cref{rem:decreasing}. This implies that $s_1<s_2$. Similarly, the fact that $i_2\in\Delta(\mathsf b)$ implies that $s_2<s_1'$. If $t_1\leq t_2$, then we can use \cref{rem:121} to see that $\mathsf b_{(m+1)s_2+t_1}\leq \mathsf b_{(m+1)s_2+t_2}$ (since $\mathsf b_{(m+1)s_1+t_1},\mathsf b_{(m+1)s_2+t_1},\mathsf b_{(m+1)s_2+t_2}$ cannot form a $121$-pattern), and we can use \cref{rem:decreasing} to see that $\mathsf b_{(m+1)s_2+t_1}\geq \mathsf b_{(m+1)s_2+t_2}$. However, this means that $\mathsf b_{(m+1)s_2+t_1}=\mathsf b_{i_1}$, which contradicts the minimality in the definition of $\varphi_{\mathsf b}(i_1)$ because $(m+1)s_2+t_1<j_1=\varphi_{\mathsf b}(i_1)$. On the other hand, if $t_2<t_1$, then we can use \cref{rem:121} to see that $\mathsf b_{(m+1)s_1'+t_2}\leq \mathsf b_{(m+1)s_1'+t_1}$ (since $\mathsf b_{(m+1)s_2+t_2},\mathsf b_{(m+1)s_1'+t_2},\mathsf b_{(m+1)s_1'+t_1}$ cannot form a $121$-pattern), and we can use \cref{rem:decreasing} to see that $\mathsf b_{(m+1)s_1'+t_2}\geq \mathsf b_{(m+1)s_1'+t_1}$. However, this means that $\mathsf b_{(m+1)s_1'+t_2}=\mathsf b_{i_2}$, which contradicts the minimality in the definition of $\varphi_{\mathsf b}(i_2)$ because $(m+1)s_1'+t_2<j_2=\varphi_{\mathsf b}(i_2)$. This yields the desired contradiction when neither $i_1$ nor $i_2$ is divisible by $m+1$. 

We now consider the case where $i_1$ is divisible by $m+1$. We have $j_1=\varphi_{\mathsf b}(i_1)=f_k$, where $k=\mathsf b_{i_1+1}$. We know by \cref{DefMeet2} that $\mathsf b_{j_2}\geq b_{j_2}(\nu)>k$ because $j_2>f_k$. Condition~\eqref{Item3} in \cref{DefMeet2} tells us that $\mathsf b_{i_2}\leq k$ since $i_2\in[i_1+1,f_k]$. This shows that $\mathsf b_{i_2}\neq\mathsf b_{j_2}=\mathsf b_{\varphi_{\mathsf b}(i_2)}$, so it follows from the definition of $\varphi_{\mathsf b}$ that $i_2$ is divisible by $m+1$. Thus, $i_1\equiv i_2\equiv 0\pmod{m+1}$. We have $i_1+1<i_2+1<j_1$ and $\mathsf b_{i_1+1}=\mathsf b_{j_1}=k$, so \cref{rem:121} tells us that $\mathsf b_{i_2+1}\leq k$. We have $i_2+1<j_1<j_2$ and $\mathsf b_{i_2+1}=\mathsf b_{j_2}$, so \cref{rem:121} tells us that $k=\mathsf b_{j_1}\leq\mathsf b_{i_2+1}$. Therefore, $k=\mathsf b_{i_2+1}=\mathsf b_{j_2}$. However, this is impossible because $j_2>j_1=f_k$. 

Finally, consider the case where $i_2$ is divisible by $m+1$ and $i_1$ is not. Let $k'=\mathsf b_{i_2+1}$. Then $i_2+1$ is $k'$-initial in $\mathsf b$, and $k'=\mathsf b_{i_2+1}=\mathsf b_{\varphi_{\mathsf b}(i_2)}=\mathsf b_{j_2}$. Since $i_2+1\leq j_1<j_2$, it follows from \cref{rem:121} that $\mathsf b_{j_1}\leq\mathsf b_{i_2+1}=k'$. Because $i_1$ is not divisible by $m+1$, we must have $\mathsf b_{i_1}=\mathsf b_{\varphi_{\mathsf b}(i_1)}=\mathsf b_{j_1}$. Since $i_1<i_2+1\leq j_1$, \cref{rem:121} tells us that $k'=\mathsf b_{i_2+1}\leq\mathsf b_{j_1}$. This shows that $\mathsf b_{i_1}=\mathsf b_{j_1}=k'$, which is a contradiction because $i_2+1$ is $k'$-initial in $\mathsf b$.
\end{proof}

Our next goal is to show that the map $\Psi\colon\Tam_n(m)\to\ShNC^{(m)}(n)$ is a bijection. To do so, we require the following simple lemma about noncrossing partitions.

\begin{lemma}\label{lem:uniqueNC}
If $\rho=\{B_1,\ldots,B_r\}$ and $\rho'=\{B_1',\ldots,B_r'\}$ are noncrossing partitions of a totally ordered set $X$ such that \[\{\min B_1,\ldots,\min B_r\}=\{\min B_1',\ldots,\min B_r'\}\] 
 and \[\{\max B_1,\ldots,\max B_r\}=\{\max B_1',\ldots,\max B_r'\},\] then $\rho=\rho'$. 
\end{lemma}

\begin{proof}
The statement is trivial if $r=1$, so we may assume $r\geq 2$ and proceed by induction on $r$. We may assume that $X=[k]$ for some positive integer $k$. There exist $u\in\{\min B_1,\ldots,\min B_r\}$ and $v\in\{\max B_1,\ldots,\max B_r\}$ such that the set $[u+1,v-1]$ is disjoint from both $\{\min B_1,\ldots,\min B_r\}$ and $\{\max B_1,\ldots,\max B_r\}$. We must have $[u,v]=B_i=B_j'$ for some $i,j\in[r]$. If we remove $B_i$ from $\rho$ and $B_j'$ from $\rho'$, we obtain noncrossing partitions of $[k]\setminus[u,v]$. We can apply the inductive hypothesis to find that these noncrossing partitions are equal, so $\rho=\rho'$.       
\end{proof}

\begin{proposition}\label{prop:Psibijection}
The map $\Psi\colon\Tam_n(m)\to\ShNC^{(m)}(n)$ constructed above is a bijection.     
\end{proposition}

\begin{proof}
We already know by \cite[Section~4.3]{Armstrong} and \cite{Dvoretzky} that $\lvert\Tam_n(m)\rvert=\lvert\ShNC^{(m)}(n)\rvert=\Cat^{(m)}_n$, so it suffices to prove that $\Psi$ is injective. Suppose $\mathsf b,\mathsf b'\in\Tam_n(m)$ are such that $\Psi(\mathsf b)=\Psi(\mathsf b')$. For each $k\in[n]$, let $B_k$ (respectively, $B_k'$) be the set of indices $i\in[(m+1)n]$ such that $\mathsf b_i=k$ (respectively, $\mathsf b_i'=k$). Let $\rho=\{B_1,\ldots,B_n\}$ and $\rho'=\{B_1',\ldots,B_n'\}$. It follows from \cref{rem:121} that $\rho$ and $\rho'$ are noncrossing partitions of $[(m+1)n]$. Moreover, $\max B_k=\max B_k'=(m+1)k$ for all $k\in[n]$, so $\{\max B_1,\ldots,\max B_n\}=\{\max B_1',\ldots,\max B_n'\}$. Observe that $1$ is in both $\{\min B_1,\ldots,\min B_n\}$ and $\{\min B_1',\ldots,\min B_n'\}$. Now suppose $t\in[2,(m+1)n]$. We saw in the proof of \cref{lem:mn+1blocks} that $t\in\{\min B_1,\ldots,\min B_n\}$ if and only if $t+1$ has a successor in $\Psi(\mathsf b)$ (equivalently, $\varphi_{\mathsf b}(t+1)\neq\emptyset$). Similarly, $t\in\{\min B_1',\ldots,\min B_n'\}$ if and only if $t+1$ has a successor in $\Psi(\mathsf b')$. Since $\Psi(\mathsf b)=\Psi(\mathsf b')$, this shows that $\{\min B_1,\ldots,\min B_n\}=\{\min B_1',\ldots,\min B_n'\}$. \cref{lem:uniqueNC} tells us that $\rho=\rho'$, from which it follows that $\mathsf b=\mathsf b'$. 
\end{proof}

In order to complete the proof of the main result of this section, we still need to show that ${\Psi\circ\row=\Rot\circ\Psi}$. Let $\ell=(m+1)n$, and recall the definition of the maps $\Theta_0,\ldots,\Theta_\ell$ from \cref{sec:describing}. In the following lemmas, we will find it convenient to write \[\mathsf b^{(\alpha)}=(\Theta_{\alpha}\circ\cdots\circ\Theta_0)(\mathsf b);\] observe that $\mathsf b^{(\ell)}=\row(\mathsf b)$ by \cref{thm:rowmotionbracket}. Also, each of the vectors $\mathsf b^{(\alpha)}$ is a $\nu$-bracket vector since each of the maps $\Theta_j$ sends $\nu$-bracket vectors to $\nu$-bracket vectors.

\begin{lemma}\label{lem:1modm+1}
Let $\mathsf b\in\Tam_n(m)$. Suppose $i,j\in[(m+1)n]$ are such that $i\equiv j\equiv 1\pmod{m+1}$ and $\varphi_{\mathsf b}(i)=j$. If $j=(m+1)u+1$, then $i$ is $u$-initial in $\row(\mathsf b)$. 
\end{lemma}

\begin{proof}
Let us write $\row(\mathsf b)=\widehat{\mathsf b}=(\widehat{\mathsf b}_0,\ldots,\widehat{\mathsf b}_\ell)$. Let $i=(m+1)s+1$. Since $\varphi_{\mathsf b}(i)=j\neq\emptyset$, we must have $i\in\Delta(\mathsf b)$. Therefore, \cref{cor:delta} tells us that $\widehat{\mathsf b}_i=z<\mathsf b_i$, where $z=\eta_i(\mathsf b)$. In order to show that $i$ is $u$-initial in $\widehat{\mathsf b}$, we need to show that $z=u$ and that $z\not\in\{\widehat{\mathsf b}_0,\ldots,\widehat{\mathsf b}_{i-1}\}$. 

The definition of $\eta_i(\mathsf b)$ ensures that $\mathsf b_r\leq z$ for all $i\leq r\leq f_z$. Since $\mathsf b_j=\mathsf b_i>z$, this implies that $j>f_z$. But $j=f_{u}+1$, so $u\geq z$. Suppose by way of contradiction that $u>z$. Then $i<f_z+1<j$, and \cref{lem:etazthing} tells us that $\mathsf b_{f_z+1}=\mathsf b_i$. This contradicts the minimality in the definition of $\varphi_{\mathsf b}(i)$, so we conclude that $u=z$. 

Suppose $\widehat{\mathsf b}_r=z$ for some $r\in[0,i-1]$. It follows from \cref{thm:rowmotionbracket} that $\mathsf b_r^{(i-1)}=\widehat{\mathsf b}_r$, so $\mathsf b_r^{(i-1)}=z=u$ and $\mathsf b_i^{(i-1)}=\mathsf b_i=\mathsf b_j\geq b_j(\nu)=u+1$; however, this contradicts condition~\eqref{Item3} in \cref{DefMeet2} (applied to $\mathsf b^{(i-1)}$) because $r+1\leq i\leq f_{u}$. 
\end{proof}

\begin{lemma}\label{lem:Claim1}
Let $\mathsf b\in\Tam_n(m)$, and suppose $i,j\in[2,(m+1)n]$ are such that $\varphi_{\mathsf b}(i)=j$. Then $\varphi_{\row(\mathsf b)}(i-1)=j-1$. 
\end{lemma}

\begin{proof}
Let $\row(\mathsf b)=\widehat{\mathsf b}=(\widehat{\mathsf b}_0,\ldots,\widehat{\mathsf b}_\ell)$. It follows from the definition of $\varphi_{\mathsf b}$ that $i\equiv j\pmod{m+1}$ and $i<j$. If $i\equiv j\equiv 1\pmod{m+1}$, then \cref{lem:1modm+1} tells us that $i$ is $u$-initial in $\widehat{\mathsf b}$, where $u$ is such that $j=(m+1)u+1=f_u+1$. In this case, it follows from the definition of $\varphi_{\widehat{\mathsf b}}$ that $\varphi_{\widehat{\mathsf b}}(i-1)=j-1$, as desired. 

In the remainder of the proof, we may assume $i\equiv j\not\equiv 1\pmod{m+1}$. Keep in mind that \[\mathsf b^{(\alpha)}=(\widehat{\mathsf b}_0,\ldots,\widehat{\mathsf b}_\alpha,\mathsf b_{\alpha+1},\ldots,\mathsf b_\ell)\] for all $\alpha\in[0,\ell]$ by \cref{thm:rowmotionbracket} (and the fact that each map $\Theta_j$ can only change the entry in position $j$ of a $\nu$-bracket vector). We now consider two cases. 

\medskip

\noindent {\bf Case 1.} Assume $i\equiv j\not\equiv 0,1\pmod{m+1}$. Let $x=\mathsf b_i$ and $y=\widehat{\mathsf b}_{i-1}$. The definition of $\varphi_{\mathsf b}$ implies that $\mathsf b_j=x$. Invoking \cref{rem:121}, we find that 
\begin{equation}\label{Eq3}
\mathsf b_r\leq x\quad\text{for all}\quad r\in[i,j].
\end{equation}
We have $y=\mathsf b^{(i-1)}_{i-1}\geq \mathsf b^{(i-1)}_{i}=\mathsf b_i=x$ by \cref{rem:decreasing}. 

We are going to prove by induction on $\alpha$ that 
\begin{equation}\label{Eq2}
    \widehat{\mathsf b}_\alpha\not\in[x,y-1]\quad \text{for all}\quad \alpha\in[0,j-1].
\end{equation} If $\alpha\leq i-2$, then this follows from condition~\eqref{Item3} in \cref{DefMeet2} since, otherwise, we would have $\alpha+1\leq i-1\leq f_{\widehat{\mathsf b}_\alpha}$ and $\widehat{\mathsf b}_{i-1}=y>\widehat{\mathsf b}_\alpha$. We also know that \eqref{Eq2} is true when $\alpha=i-1$ since $\widehat{\mathsf b}_{i-1}=y$. Now suppose $i\leq \alpha\leq j-1$. If $\alpha=f_k=(m+1)k$ for some $k$, then $f_k<j<f_x$, so $\widehat{\mathsf b}_\alpha=k<x$. If $\alpha\in\Delta(\mathsf b^{(\alpha-1)})$, then we can use \cref{thm:rowmotionbracket}, the definition of $\Theta_\alpha$, and \eqref{Eq3} to see that \[\widehat{\mathsf b}_\alpha=(\Theta_\alpha(\mathsf b^{(\alpha-1)}))_\alpha=\eta_\alpha(\mathsf b^{(\alpha-1)})<\mathsf b_\alpha^{(\alpha-1)}=\mathsf b_\alpha\leq x.\] Finally, suppose $\alpha$ is not divisible by $m+1$ (i.e., is not of the form $f_k$) and $\alpha\not\in\Delta(\mathsf b^{(\alpha-1)})$. According to \cref{thm:rowmotionbracket} and the definition of $\Theta_\alpha$, we have \[\widehat{\mathsf b}_\alpha=(\Theta_\alpha(\mathsf b^{(\alpha-1)}))_\alpha=\zeta_\alpha(\mathsf b^{(\alpha-1)})=\min\{\mathsf b_r^{(\alpha-1)}:0\leq r\leq \alpha-1\text{ and }\mathsf b_r^{(\alpha-1)}\geq \mathsf b_\alpha^{(\alpha-1)}\}.\] Our induction hypothesis implies that $\zeta_\alpha(\mathsf b^{(\alpha-1)})\not\in[x,y-1]$. This proves~\eqref{Eq2}.  

We have $i\in\Delta(\mathsf b)$ because $\varphi_{\mathsf b}(i)\neq\emptyset$. It follows from \cref{cor:delta} that $\widehat{\mathsf b}_{i}=\eta_{i}(\mathsf b)<\mathsf b_i=x$. Now, $\widehat{\mathsf b}_{i-1}=y\geq x>\eta_i(\mathsf b)=\widehat{\mathsf b}_i$, so $i-1\in\Delta(\widehat{\mathsf b})$. This means that $\varphi_{\widehat{\mathsf b}}(i-1)$ is the smallest integer $p$ such that $p>i-1$, $p\equiv i-1\pmod{m+1}$, and $\widehat{\mathsf b}_p=y$; our goal is to show that $\varphi_{\widehat{\mathsf b}}(i-1)=j-1$. By combining \eqref{Eq3} with \cref{rem:decreasing} and the fact that $\mathsf b_j=x$, we find that $\mathsf b_{j-1}=x$. Therefore, $\mathsf b_{j-1}^{(j-2)}=\mathsf b_j^{(j-2)}=x$; it follows that $j-1\not\in\Delta(\mathsf b^{(j-2)})$.  \cref{thm:rowmotionbracket} tells us that \[\widehat{\mathsf b}_{j-1}=(\Theta_{j-1}(\mathsf b^{(j-2)}))_{j-1}=\zeta_{j-1}(\mathsf b^{(j-2)})=\min\{\mathsf b_r^{(j-2)}:0\leq r\leq j-2\text{ and }\mathsf b_r^{(j-2)}\geq x\}\] \[=\min\{\widehat{\mathsf b}_r:0\leq r\leq j-2\text{ and }\widehat{\mathsf b}_r\geq x\},\] so it follows from \eqref{Eq2} and the fact that $\widehat{\mathsf b}_{i-1}=y$ that $\widehat{\mathsf b}_{j-1}=y$. Consequently, we are left to show that 
\begin{equation}\label{Eq10}
\widehat{\mathsf b}_r\neq y\text{ for all }r\in[i,j-2]\text{ satisfying }r\equiv i-1\hspace{-.25cm}\pmod{m+1}.
\end{equation} 

Let us write $i=(m+1)s+t$ and $j=(m+1)s'+t$, where $0\leq s<s'\leq n-1$ and $2\leq t\leq m$. Let $z_1=\eta_i(\mathsf b)=\widehat{\mathsf b}_i$; note that $z_1<x$. Condition~\eqref{Item3} in \cref{DefMeet2} tells us that $\widehat{\mathsf b}_r\leq z_1<x$ for all $r\in[i+1,(m+1)z_1]$. In particular, 
\begin{equation}\label{Eq:z_1}
\widehat{\mathsf b}_{(m+1)v+t-1}<x\leq y\text{ for all } v\in[s+1,z_1-1].
\end{equation} 
If $z_1\geq s'$, then this proves \eqref{Eq10}. Now suppose $z_1<s'$. We can use \cref{lem:etazthing} to conclude that $\mathsf b_{(m+1)z_1+1}=x$. By the definition of $\varphi_{\mathsf b}(i)$, we have $\mathsf b_{(m+1)z_1+t}\neq x$, so (by \cref{rem:decreasing}) there must be some $1\leq u_1<t$ such that \[{\mathsf b_{(m+1)z_1+u_1}=x>\mathsf b_{(m+1)z_1+u_1+1}}.\] Then $(m+1)z_1+u_1\in\Delta(\mathsf b)$, so it follows from \cref{cor:delta} that $\widehat{\mathsf b}_{(m+1)z_1+u_1}=\eta_{(m+1)z_1+u_1}(\mathsf b)$. Let $z_2=\eta_{(m+1)z_1+u_1}(\mathsf b)=\widehat{\mathsf b}_{(m+1)z_1+u_1}$; note that $z_2<x$. Condition~\eqref{Item3} in \cref{DefMeet2} tells us that $\widehat{\mathsf b}_r\leq z_2<x$ for all $r\in[(m+1)z_1+u_1,(m+1)z_2]$. In particular, 
\begin{equation}\label{Eq:z_2}
\widehat{\mathsf b}_{(m+1)v+t-1}<x\leq y\text{ for all }v\in[z_1,z_2-1].
\end{equation}
If $z_2\geq s'$, then \eqref{Eq:z_1} and \eqref{Eq:z_2} together complete the proof of \eqref{Eq10}, so we may assume $z_2<s'$. An argument similar to the one just given shows that there exists $1\leq u_2<t$ such that \[\mathsf b_{(m+1)z_2+u_2}=x>\mathsf b_{(m+1)z_2+u_2+1}.\] Then $(m+1)z_2+u_2\in\Delta(\mathsf b)$, so it follows from \cref{cor:delta} that $\widehat{\mathsf b}_{(m+1)z_2+u_2}=\eta_{(m+1)z_2+u_2}(\mathsf b)$. We can then define $z_3=\eta_{(m+1)z_2+u_2}(\mathsf b)$ and note that $z_3<x$. Also, $\widehat{\mathsf b}_r\leq z_3<x$ for all $r\in[(m+1)z_2+u_2,(m+1)z_3]$, so $\widehat{\mathsf b}_{(m+1)v+t-1}<x\leq y$ for all $v\in[z_2,z_3-1]$. Continuing in this fashion, we construct numbers $z_1,z_2,\ldots$ until eventually obtaining some $z_h$ that is at least $s'$, at which point we obtain a proof of \eqref{Eq10}.

\medskip

\noindent {\bf Case 2.} Assume $i\equiv j\equiv 0\pmod{m+1}$. Let $i=(m+1)z_1=f_{z_1}$ and $j=(m+1)x=f_{x}$, and note that $z_1<x$. The definition of $\varphi_{\mathsf b}$ implies that $i+1$ is $x$-initial in $\mathsf b$; in particular, $\mathsf b_{i+1}=x$. Let $y=\widehat{\mathsf b}_{i-1}$. Condition~\eqref{Item3} in \cref{DefMeet2} tells us that 
\begin{equation}\label{Eq5}
\mathsf b_{r}\leq x\quad\text{for all}\quad r\in[i+1,j].
\end{equation}

In order to prove that $\varphi_{\widehat{\mathsf b}}(i-1)=j-1$, we first want to show that $i-1\in\Delta(\widehat{\mathsf b})$, meaning $y>\widehat{\mathsf b}_i=z_1$. We will actually show that $y\geq x$. Since $y\geq z_1$ (by \cref{rem:decreasing}), we just need to prove that $y\not\in[z_1,x-1]$. In fact, we will prove by induction on $\alpha$ the stronger assertion that \begin{equation}\label{Eq7}
\widehat{\mathsf b}_\alpha\not\in[z_1,x-1]\quad\text{for all}\quad \alpha\in[0,i-1].    
\end{equation} 
First, suppose $\mathsf b_\alpha\geq x$. Since $i+1$ is $x$-initial in $\mathsf b$, we must actually have $\mathsf b_\alpha\geq x+1$. Notice that $\mathsf b^{(\alpha-1)}_r=\mathsf b_r\leq x$ for all $r\in[i+2,f_x]$ by condition~\eqref{Item3} in \cref{DefMeet2}. Also, $\mathsf b_{i+1}^{(\alpha-1)}=\mathsf b_{i+1}=x$ and $\mathsf b^{(\alpha-1)}_i=\mathsf b_i=z_1<x$. This shows that $\mathsf b_r^{(\alpha-1)}\leq x$ for all $r\in[i,f_x]$. If, in addition, we have $\alpha\in\Delta(\mathsf b^{(\alpha-1)})$ and $\mathsf b^{(\alpha-1)}_r\leq x$ for all $r\in[\alpha+1,i-1]$, then it follows from \cref{thm:rowmotionbracket} and the definition of $\eta_\alpha(\mathsf b^{(\alpha-1)})$ that \[\widehat{\mathsf b}_\alpha=(\Theta_\alpha(\mathsf b^{(\alpha-1)}))_\alpha=\eta_\alpha(\mathsf b^{(\alpha-1)})\geq x.\] On the other hand, if there exists $r\in[\alpha+1,i-1]$ such that $\mathsf b_r^{(\alpha-1)}>x$, then \cref{rem:121} tells us that $\widehat{\mathsf b}_{\alpha}\not\in[z_1,x-1]$ since, otherwise, the entries in positions $\alpha,r,f_{\widehat{\mathsf b}_{\alpha}}$ would form a $121$-pattern in $\mathsf b^{(\alpha)}$ (since $\mathsf b_\alpha^{(\alpha)}=\widehat{\mathsf b}_\alpha$ and $\mathsf b^{(\alpha)}_r=\mathsf b_r^{(\alpha-1)}$). Finally, if $\alpha\not\in\Delta(\mathsf b^{(\alpha-1)})$, then \[\widehat{\mathsf b}_\alpha=(\Theta_\alpha(\mathsf b^{(\alpha-1)}))_\alpha=\zeta_\alpha(\mathsf b^{(\alpha-1)})=\min\{\mathsf b_r^{(\alpha-1)}:0\leq r\leq \alpha-1\text{ and }\mathsf b_r^{(\alpha-1)}\geq \mathsf b_\alpha^{(\alpha-1)}\}\] \[=\min\{\widehat{\mathsf b}_r:0\leq r\leq \alpha-1\text{ and }\widehat{\mathsf b}_r\geq \mathsf b_\alpha^{(\alpha-1)}\},\] so it follows by induction on $\alpha$ that $\widehat{\mathsf b}_\alpha\not\in[x,y-1]$ in this case as well. This proves \eqref{Eq7}, so we conclude that $y\geq x$ and $i-1\in\Delta(\widehat{\mathsf b})$.

We are now going to prove by induction on $\alpha$ that 
\begin{equation}\label{Eq4}
    \widehat{\mathsf b}_\alpha\not\in[x,y-1]\quad \text{for all}\quad \alpha\in[0,j-1].
\end{equation} 
This is certainly true if $0\leq \alpha\leq i-2$ since, otherwise, we would have $\alpha+1\leq i-1\leq f_{\widehat{\mathsf b}_\alpha}$ and $\widehat{\mathsf b}_{i-1}=y>\widehat{\mathsf b}_\alpha$ (violating condition~\eqref{Item3} in \cref{DefMeet2}). We also know that \eqref{Eq4} is true when $\alpha=i-1$ or $\alpha=i$ since $\widehat{\mathsf b}_{i-1}=y$ and $\widehat{\mathsf b}_i=\widehat{\mathsf b}_{f_{z_1}}=z_1<x$. Now suppose $i+1\leq \alpha\leq j-1$. If $\alpha=f_k=(m+1)k$ for some $k$, then $f_k<j=f_x$, so $\widehat{\mathsf b}_\alpha=k<x$. If $\alpha\in\Delta(\mathsf b^{(\alpha-1)})$, then we can use \cref{thm:rowmotionbracket}, the definition of $\Theta_\alpha$, and \eqref{Eq5} to see that \[\widehat{\mathsf b}_\alpha=(\Theta_\alpha(\mathsf b^{(\alpha-1)}))_\alpha=\eta_\alpha(\mathsf b^{(\alpha-1)})<\mathsf b_\alpha^{(\alpha-1)}=\mathsf b_\alpha\leq x.\] Finally, suppose $\alpha$ is not divisible by $m+1$ (i.e., is not of the form $f_k$) and $\alpha\not\in\Delta(\mathsf b^{(\alpha-1)})$. According to \cref{thm:rowmotionbracket} and the definition of $\Theta_\alpha$, we have \[\widehat{\mathsf b}_\alpha=(\Theta_\alpha(\mathsf b^{(\alpha-1)}))_\alpha=\zeta_\alpha(\mathsf b^{(\alpha-1)})=\min\{\mathsf b_r^{(\alpha-1)}:0\leq r\leq \alpha-1\text{ and }\mathsf b_r^{(\alpha-1)}\geq \mathsf b_\alpha^{(\alpha-1)}\}.\] Our induction hypothesis implies that $\zeta_\alpha(\mathsf b^{(\alpha-1)})\not\in[x,y-1]$. This proves~\eqref{Eq4}.  

We saw above that $i-1\in\Delta(\widehat{\mathsf b})$, so $\varphi_{\widehat{\mathsf b}}(i-1)$ is the smallest integer $p$ such that $p>i-1$, $p\equiv -1\pmod{m+1}$, and $\widehat{\mathsf b}_p=y$; our goal is to show that $\varphi_{\widehat{\mathsf b}}(i-1)=j-1$. By combining \eqref{Eq7} with \cref{rem:decreasing} and the fact that $\mathsf b_j=x$, we find that $\mathsf b_{j-1}=x$. Since $\mathsf b_{j-1}^{(j-2)}=\mathsf b_j^{(j-2)}=x$, we have $j-1\not\in\Delta(\mathsf b^{(j-2)})$.  \cref{thm:rowmotionbracket} tells us that \[\widehat{\mathsf b}_{j-1}=(\Theta_{j-1}(\mathsf b^{(j-2)}))_{j-1}=\zeta_{j-1}(\mathsf b^{(j-2)})=\min\{\mathsf b_r^{(j-2)}:0\leq r\leq j-2\text{ and }\mathsf b_r^{(j-2)}\geq x\}\] \[=\min\{\widehat{\mathsf b}_r:0\leq r\leq j-2\text{ and }\widehat{\mathsf b}_r\geq x\},\] so it follows from \eqref{Eq4} and the fact that $\widehat{\mathsf b}_{i-1}=y$ that $\widehat{\mathsf b}_{j-1}=y$. Therefore, we are left to show that $\widehat{\mathsf b}_{(m+1)v+m}\neq y$ for all $v\in[z_1,x-2]$. This is vacuously true if $z_1=x-1$, so we may assume $z_1<x-1$.   

Recall that $\mathsf b_{(m+1)z_1+1}=\mathsf b_{i+1}=x$. Let $u_1$ be the largest element of $[m]$ such that ${\mathsf b_{(m+1)z_1+u_1}=x}$. Then $(m+1)z_1+u_1\in\Delta(\mathsf b)$ (by \cref{rem:decreasing}), so it follows from \cref{cor:delta} that \[\widehat{\mathsf b}_{(m+1)z_1+u_1}=\eta_{(m+1)z_1+u_1}(\mathsf b).\] Let $z_2=\eta_{(m+1)z_1+u_1}(\mathsf b)=\widehat{\mathsf b}_{(m+1)z_1+u_1}$; note that $z_2<x$ since $\mathsf b_{(m+1)z_1+u_1}\leq x$ by \eqref{Eq5}. Condition~\eqref{Item3} in \cref{DefMeet2} tells us that $\widehat{\mathsf b}_r\leq z_2<x$ for all 
$r\in[(m+1)z_1+u_1,(m+1)z_2]$. In particular, $\widehat{\mathsf b}_{(m+1)v+m}<x$ for all $v\in[z_1,z_2-1]$. Let $u_2$ be the largest element of $[m]$ such that $\mathsf b_{(m+1)z_2+u_2}=x$. If $z_2<x-1$, then $(m+1)z_2+u_2\in\Delta(\mathsf b)$, so it follows from \cref{cor:delta} that $\widehat{\mathsf b}_{(m+1)z_2+u_2}=\eta_{(m+1)z_2+u_2}(\mathsf b)$. We can then define $z_3=\eta_{(m+1)z_2+u_2}(\mathsf b)$ and note that $z_3<x$ since $\mathsf b_{(m+1)z_2+u_2}\leq x$ by \eqref{Eq5}. Also, $\widehat{\mathsf b}_r\leq z_3<x$ for all $r\in[(m+1)z_2+u_2,(m+1)z_3]$, so $\widehat{\mathsf b}_{(m+1)v+m}<x$ for all $v\in[z_2,z_3-1]$. Continuing in this fashion, we construct numbers $z_1,z_2,\ldots$ until eventually obtaining some $z_h$ that is equal to $x-1$. This argument shows that $\widehat{\mathsf b}_{(m+1)v+m}<x\leq y$ for all $v\in[z_1,x-2]$, which proves the desired result in this case. 
\end{proof}

\begin{lemma}\label{lem:Claim2}
Let $\mathsf b\in\Tam_n(m)$. Let $N$ be the largest element of the block of $\Psi(\mathsf b)$ containing $1$, and assume $N>1$ (i.e., $1$ is not in a singleton block). Then  $\varphi_{\row(\mathsf b)}(N-1)=(m+1)n$. 
\end{lemma}

\begin{proof}
Let $\widehat{\mathsf b}=\row(\mathsf b)$. Let $i_1<\cdots<i_p$ be the elements of the block of $\Psi(\mathsf b)$ containing $1$. Then $i_1=1$ and $i_p=N$. The description of $\Psi(\mathsf b)$ given in \cref{lem:Psi1} tells us that $\varphi_{\mathsf b}(i_t)=i_{t+1}$ for all $t\in[p-1]$. It follows from the definition of $\varphi_{\mathsf b}$ that $1=i_1\equiv\cdots\equiv i_p\pmod{m+1}$. Let us write $i_t=(m+1)u_t+1$ for all $t\in[p]$. Since $N-1\equiv 0\pmod{m+1}$, we can refer to the definition of $\varphi_{\mathsf b}$ to see that our goal is to show that $N$ is $n$-initial in $\widehat{\mathsf b}$. 

Consider $t\in[p-1]$. It follows from \cref{lem:1modm+1} that $i_t$ is $u_{t+1}$-initial in $\row(\mathsf b)$. Because $f_{u_{t+1}}=(m+1)u_{t+1}=i_{t+1}-1$, it follows from condition~\eqref{Item3} in \cref{DefMeet2} that $\widehat{\mathsf b}_r\leq u_{t+1}$ for all $r\in[i_t,i_{t+1}-1]$. This shows that $\widehat{\mathsf b}_r\leq u_{p}$ for all $r\in[0,N-1]$. In particular, $\widehat{\mathsf b}_r\neq n$ for all $r\in[0,N-1]$. Therefore, we will be done if we can prove that $\widehat{\mathsf b}_N=n$.  

Because $N$ is the largest element of its block in $\Psi(\mathsf b)$, it is in $\varphi_{\mathsf b}^{-1}(\emptyset)$ by \cref{lem:Psi1}. This means that $N\not\in\Delta(\mathsf b)$. Since $\mathsf b^{(N-1)}_N=\mathsf b_N$ and $\mathsf b^{(N-1)}_{N+1}=\mathsf b_{N+1}$, we also have $N\not\in\Delta(\mathsf b^{(N-1)})$. By \cref{thm:rowmotionbracket} and the definition of $\Theta_N$, we have $\widehat{\mathsf b}_N=(\Theta_N(\mathsf b^{(N-1)}))_N=\zeta_N(\mathsf b^{(N-1)})$. Recall that our convention is that $\zeta_N(\mathsf b^{(N-1)})=n$ if the set $\{\mathsf b_r^{(N-1)}:0\leq r\leq N-1\text{ and }\mathsf b_r^{(N-1)}\geq \mathsf b_N^{(N-1)}\}$ is empty. We know that $\mathsf b^{(N-1)}_r=\widehat{\mathsf b}_r\leq u_p$ for all $r\in[0,N-1]$. On the other hand, we have $\mathsf b^{(N-1)}_N\geq b_N(\nu)=u_p+1$. This shows that $\{\mathsf b_r^{(N-1)}:0\leq r\leq N-1\text{ and }\mathsf b_r^{(N-1)}\geq \mathsf b_N^{(N-1)}\}$ is indeed empty, so $\widehat{\mathsf b}_N=\zeta_N(\mathsf b^{(N-1)})=n$. 
\end{proof}

We are now finally in a position to prove the main result of this section. 

\begin{theorem}\label{thm:nuTamariMain}
The map $\Psi\colon\Tam_n(m)\to\ShNC^{(m)}(n)$ is a bijection satisfying \[\Psi\circ\row=\Rot\circ\Psi.\] 
\end{theorem}

\begin{proof}
We already know by \cref{prop:Psibijection} that $\Psi$ is a bijection. Consider a $\nu$-bracket vector $\mathsf b\in\Tam_n(m)$; our goal is to show that $\Psi(\row(\mathsf b))=\Rot(\Psi(\mathsf b))$. We will make use of \cref{lem:mn+1blocks} (or rather, its proof), which tells us that 
\begin{equation}\label{Eq8}
\lvert\varphi_{\mathsf b}^{-1}(\emptyset)\rvert=\lvert\varphi_{\row(\mathsf b)}^{-1}(\emptyset)\rvert=mn+1.
\end{equation}
Let $N$ be the largest element of the block of $\Psi(\mathsf b)$ containing $1$. Let $\mathcal X$ be the set of pairs $(\alpha,\alpha')$ such that $\alpha,\alpha'\in[(m+1)n]$ and $\varphi_{\row(\mathsf b)}(\alpha)=\alpha'$. Note that $\lvert\mathcal X\rvert=(m+1)n-\lvert\varphi_{\row(\mathsf b)}^{-1}(\emptyset)\rvert=n-1$ by~\eqref{Eq8}. 

Suppose $N=1$. Let $\mathcal Y$ be the set of pairs of the form $(i-1,j-1)$ such that $i,j\in[2,(m+1)n]$ and $\varphi_{\mathsf b}(i)=j$. Since $N=1$, we have $\varphi_{\mathsf b}(1)=\emptyset$, so $|\mathcal Y|=(m+1)n-|\varphi_{\mathsf b}^{-1}(\emptyset)|=n-1$ by~\eqref{Eq8}.  \cref{lem:Claim1} tells us that $\mathcal Y\subseteq\mathcal X$, so we must have $\mathcal X=\mathcal Y$. This implies that $\Psi(\row(\mathsf b))=\Rot(\Psi(\mathsf b))$, as desired. 

Now suppose $N>1$. Let $\mathcal Z'$ be the set of pairs of the form $(i-1,j-1)$ such that $i,j\in[2,(m+1)n]$ and $\varphi_{\mathsf b}(i)=j$. Let $\mathcal Z=\mathcal Z'\cup\{(N,(m+1)n+1)\}$. Since $N>1$, we have $\varphi_{\mathsf b}(1)\neq\emptyset$, so $|\mathcal Z'|=|([(m+1)n]\setminus\varphi_{\mathsf b}^{-1}(\emptyset))\setminus\{1\}|=n-2$ by \eqref{Eq8}. This shows that $|\mathcal Z|=n-1=|\mathcal X|$. \cref{lem:Claim1,lem:Claim2} tell us that $\mathcal Z\subseteq\mathcal X$, so we must have $\mathcal X=\mathcal Z$. This implies that $\Psi(\row(\mathsf b))=\Rot(\Psi(\mathsf b))$, as desired. 
\end{proof}

The following result settles Thomas and Williams's \cref{conj:ThomasWilliams} when $a\equiv 1\pmod{b}$.

\begin{theorem}\label{cor:mainnuTamari}
The order of $\row\colon\Tam_n(m)\to\Tam_n(m)$ is $(m+1)n$. Moreover, the triple $(\Tam_n(m),\row,\Cat_n^{(m)}(q))$ exhibits the cyclic sieving phenomenon.  
\end{theorem}

\begin{proof}
The first sentence is immediate from \cref{thm:nuTamariMain} because the map \[\Rot\colon\ShNC^{(m)}(n)\to\ShNC^{(m)}(n)\] clearly has order $(m+1)n$. The second sentence is a consequence of \cref{thm:nuTamariMain} and \eqref{Eq9}. 
\end{proof}

Our final result of this section exhibits an instance of the homomesy phenomenon (see \cref{Sec:Homomesy}). Recall from the introduction that the \emph{down-degree statistic} on a poset $P$ is the function $\ddeg\colon P\to\mathbb R$ defined by $\ddeg(x)=\lvert\{y\in P:y\lessdot x\}\rvert$. 

\begin{theorem}\label{thm:homomesy}
The down-degree statistic on $\Tam_n(m)$ is homomesic for rowmotion with average \[\frac{m(n-1)}{m+1}.\] 
\end{theorem}

\begin{proof}
Suppose $\mathsf b\in\Tam_n(m)$. It follows from \cite[Proposition~4.4]{DefantMeeting} that $\ddeg(\mathsf b)=|\Delta(\mathsf b)|$. Referring to the description of $\Psi(\mathsf b)$ in \cref{lem:Psi1} and the definition of $\varphi_{\mathsf b}$, we find that $\Delta(\mathsf b)$ is the set of indices $i\in[(m+1)n]$ such that $i$ is not maximal in its block of $\Psi(\mathsf b)$ and $i\not\equiv 0\pmod{m+1}$. For $\rho\in\ShNC^{(m)}(n)$, let $Q(\rho)$ be the number of indices $i\in[(m+1)n]$ such that $i$ is maximal in its block in $\rho$ and $i\not\equiv 0\pmod{m+1}$. Then $\ddeg(\mathsf b)=|\Delta(\mathsf b)|=mn-Q(\Psi(\mathsf b))$. Invoking \cref{thm:nuTamariMain}, we find that the desired result will follow if we can prove that $Q$ is homomesic for $\Rot\colon\ShNC^{(m)}(n)\to\ShNC^{(m)}(n)$ with average $mn-m(n-1)/(m+1)$.      

For a set $X\subseteq[(m+1)n]$, let $f(X)=1$ if $\max X\not\equiv 0\pmod{m+1}$, and let $f(X)=0$ otherwise. Fix $\rho\in\ShNC^{(m)}(n)$. Because $\Rot\colon\ShNC^{(m)}(n)\to \ShNC^{(m)}(n)$ has order $(m+1)n$, the average of $Q$ along the orbit of $\Rot$ containing $\rho$ is \[\frac{1}{(m+1)n}\sum_{k=1}^{(m+1)n}Q(\Rot^k(\rho));\] our goal is to show that this quantity is $mn-m(n-1)/(m+1)$.
For each block $B\in\rho$ and each $k\geq 1$, let $B-k$ be the block of $\Rot^k(\rho)$ whose elements are obtained by subtracting $k$ from the elements of $B$ and reducing modulo $(m+1)n$. Since all of the elements of $B$ are congruent modulo $m+1$ (by the definition of $\ShNC^{(m)}(n)$), we have $\sum_{k=1}^{(m+1)n}f(B-k)=mn$. Thus, \[\frac{1}{(m+1)n}\sum_{k=1}^{(m+1)n}Q(\Rot^k(\rho))=\frac{1}{(m+1)n}\sum_{k=1}^{(m+1)n}\sum_{X\in\Rot^k(\rho)}f(X)=\frac{1}{(m+1)n}\sum_{k=1}^{(m+1)n}\sum_{B\in \rho}f(B-k)\] \[=\frac{1}{(m+1)n}\sum_{B\in \rho}\sum_{k=1}^{(m+1)n}f(B-k)=\frac{1}{(m+1)n}\sum_{B\in \rho}mn=\frac{1}{(m+1)n}mn(mn+1),\] where the last equality follows from the fact that $\rho$ has $mn+1$ blocks (by the definition of $\ShNC^{(m)}(n)$). Finally, note that $\frac{1}{(m+1)n}mn(mn+1)=mn-m(n-1)/(m+1)$.
\end{proof}

To end our discussion of $\nu$-Tamari lattices, we state two conjectures about rowmotion on rational Tamari lattices. The first conjecture strengthens \cref{conj:ThomasWilliams,cor:mainnuTamari}, while the second strengthens \cref{thm:homomesy}. Recall the definition of $\Tam(a,b)$ from \cref{subsec:LatticePaths}.

\begin{conjecture}\label{THING2}
Let $a$ and $b$ be relatively prime positive integers, and let \[\Cat^{(a,b)}(q)=\frac{1}{[a+b]_q}{a+b\brack b}_q.\] The triple $(\Tam(a,b),\row,\Cat^{(a,b)}(q))$ exhibits the cyclic sieving phenomenon. 
\end{conjecture}

\begin{conjecture}\label{THING}
Let $a$ and $b$ be relatively prime positive integers. The down-degree statistic on $\Tam(a,b)$ is homomesic for rowmotion with average \[\frac{(a-1)(b-1)}{a+b-1}.\]  
\end{conjecture}

\section{BiCambrian Lattices and Doubled Root Posets}\label{Sec:BiCambrianDoubled}

In the next several sections, we turn our attention to rowmotion on biCambrian lattices. We begin with some notation and terminology surrounding the main result that we seek to prove.

\subsection{Lattice Congruences}\label{Subsec:Congruences}
Let $L$ be a lattice. A \dfn{lattice congruence} of $L$ is an equivalence relation $\equiv$ on $L$ that respects meets and joins. More precisely, this means that if $x_1,x_2,y_1,y_2\in L$ satisfy $x_1\equiv x_2$ and $y_1\equiv y_2$, then $(x_1\wedge y_1)\equiv(x_2\wedge y_2)$ and $(x_1\vee y_1)\equiv(x_2\vee y_2)$. 

Each equivalence class $C$ of a lattice congruence $\equiv$ is an interval in $L$, so it has a unique minimal element $C_{\min}$ and a unique maximal element $C_{\max}$. We obtain projection operators $\pi_\downarrow,\pi^\uparrow:L\to L$ by letting $\pi_\downarrow(x)=C_{\min}$ and $\pi^\uparrow(x)=C_{\max}$ whenever $x$ is in the equivalence class $C$. Thus, $\pi_\downarrow(L)$ is the set of minimal elements of equivalence classes, while $\pi^\uparrow(L)$ is the set of maximal elements of equivalence classes. We consider $\pi_\downarrow(L)$ and $\pi^\uparrow(L)$ as subposets of $L$.

The lattice congruence $\equiv$ gives rise to a quotient lattice $L/\equiv$ whose elements are the equivalence classes of $\equiv$. The order relation on $L/\equiv$ is defined so that $C\leq C'$ if and only if there exist $x\in C$ and $x'\in C'$ such that $x\leq x'$. It is straightforward to check that $C\leq C'$ in $L/\equiv$ if and only if $C_{\min}\leq C_{\min}'$ in $L$ if and only if $C_{\max}\leq C_{\max}'$ in $L$. Thus, $\pi_\downarrow(L)$ and $\pi^\uparrow(L)$ are lattices that are naturally isomorphic to $L/\equiv$.

\subsection{Doubled Root Posets}

Let $\Phi$ be a finite root system with a corresponding Coxeter group $W$ (see \cite{BjornerBrenti,Humphreys} for relevant background). Let $\Phi^+$ be a system of positive roots for $\Phi$. If $\Phi$ is crystallographic, then there is a natural partial order on $\Phi^+$ obtained by saying $\alpha\leq\beta$ if $\beta-\alpha$ is a nonnegative linear combination of positive roots. Endowed with this partial order, $\Phi^+$ is called the \dfn{root poset} of $\Phi$. The minimal elements of $\Phi^+$ are the \dfn{simple roots}. 

Barnard and Reading \cite{BarnardReading} introduced \emph{Coxeter--biCatalan combinatorics} as a variant of Coxeter-Catalan combinatorics motivated by several fascinating enumerative phenomena. Given a  root poset $\Phi^+$, we obtain the associated \dfn{doubled root poset} by taking a copy of $\Phi^+$ and a copy of the dual of $\Phi^+$ and identifying corresponding simple roots. For example, the doubled root posets of types $A_5$ and $B_3$ are shown on the left in \cref{FigRowTam3}. In general, the doubled root poset of type $A_n$ is isomorphic to the product of two chains $[n]\times[n]$, and the doubled root poset of type $B_n$ is a shifted staircase with $n(2n-1)$ elements (this is the quotient of $[n]\times[n]$ modulo the action of the map $(i,j)\mapsto(j,i)$).  

Although the root systems of types $H_3$ and $I_2(m)$ are not (in general) crystallographic, Armstrong gave a definition of what their root posets should be \cite[Section~5.4.1]{Armstrong} (see also \cite{Cuntz}). Using this definition, we obtain the doubled root posets of types $H_3$ and $I_2(6)$ as shown on the right in \cref{FigRowTam3}. The doubled root poset of type $I_2(m)$ is the obvious analogue of the doubled root poset of type $I_2(6)$ with $2m-2$ elements. 

\begin{figure}[ht]
  \begin{center}{\includegraphics[height=5.9cm]{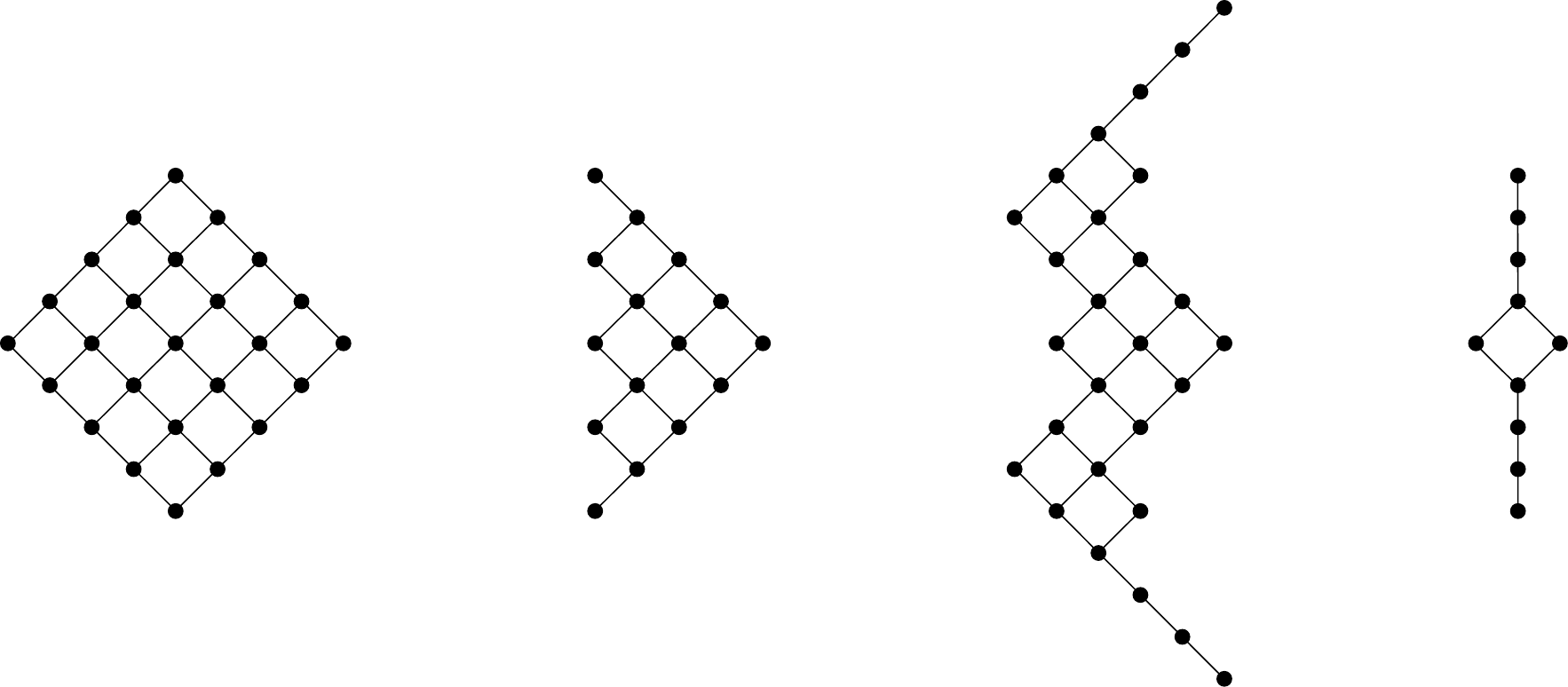}}
  \end{center}
  \caption{Doubled root posets of types $A_5$, $B_3$, $H_3$, and $I_2(6)$ (from left to right).}\label{FigRowTam3}
\end{figure}

The root systems of \dfn{coincidental type} are those of type $A$, $B$, $H_3$, or $I_2(m)$. It turns out that the doubled root poset of a root system of coincidental type is always a \emph{minuscule poset}; such posets arise naturally as the weight lattices of certain representations of simple complex Lie algebras. The orbit structures of rowmotion on the lattices of order ideals of minuscule posets are well understood. For type-$A$ minuscule posets (i.e., \emph{rectangle posets}), this follows from a theorem of Stanley \cite{StanleyPromotion} and Theorem~1.1(b) in the article \cite{CSPDefinition} by Reiner, Stanton, and White. Type-$B$ minuscule posets (i.e., \emph{shifted staircase posets}) were first worked out (with essentially the same argument used for minuscule posets of type $A$) by Striker and Williams \cite[Corollary~6.3]{StrikerWilliams}. Rush and Shi \cite{RushShi} gave a uniform treatment of all minuscule posets, proving the following theorem. Note that minuscule posets are ranked (see \cref{subsec:posets} for the definition of the rank function $\rk$).

\begin{theorem}[\cite{RushShi}]\label{thm:RushShi}
Let $P$ be a minuscule poset with rank function $\rk\colon P\to\mathbb Z_{\geq 0}$. The triple $(\mathcal J(P),\row,F_P(q))$ exhibits the cyclic sieving phenomenon, where \[F_P(q)=\prod_{x\in P}\frac{1-q^{\rk(x)+2}}{1-q^{\rk(x)+1}}.\]
\end{theorem}

\subsection{BiCambrian Lattices}
In what follows, whenever we refer to the \emph{weak order} on a finite Coxeter group, we mean the right weak order; it is well known that the weak order is a semidistributive lattice. Given a finite irreducible Coxeter group $W$ and a Coxeter element $c$ of $W$, there is an associated lattice congruence on the weak order of $W$ called the $c$-Cambrian congruence \cite{ReadingCambrian}. Since $c^{-1}$ is also a Coxeter element, one can also consider the $c^{-1}$-Cambrian congruence. The \dfn{$c$-biCambrian congruence} is the lattice congruence on the weak order of $W$ obtained by taking the common refinement of the $c$-Cambrian and $c^{-1}$-Cambrian congruences. The \dfn{$c$-biCambrian lattice} is the quotient of the weak order modulo the $c$-biCambrian congruence. Being a quotient of the weak order, the $c$-biCambrian lattice is semidistributive, so it comes equipped with a rowmotion operator.

Let $\Phi$ be a finite root system of coincidental type with a corresponding Coxeter group $W$ and corresponding doubled root poset $P$. Because the Coxeter diagram of $W$ is a tree, it has a bipartition $X\sqcup Y$. Let $c_+=\prod_{s\in X}s$ and $c_-=\prod_{s\in Y}s$. Then $c=c_+c_-$ is called a \dfn{bipartite Coxeter element}; its inverse is $c_-c_+$, which is also bipartite. Our goal in the next three sections is to prove the following theorem, which Thomas and Williams stated as a conjecture in \cite{ThomasWilliams}.

\begin{theorem}\label{Conj:biCamb}
Let $W$ be a Coxeter group of coincidental type, and let $c$ be a bipartite Coxeter element of $W$. Let $P$ be the doubled root poset of type $W$. The orbit structure of rowmotion on the $c$-biCambrian lattice is the same as the orbit structure of rowmotion on $\mathcal J(P)$.
\end{theorem} 

Let $W$, $c$ and $P$ be as in \cref{Conj:biCamb}, and let $F_P(q)$ be the polynomial defined in \cref{thm:RushShi}. \cref{thm:RushShi,Conj:biCamb} imply that rowmotion acting on the $c$-biCambrian lattice exhibits the cyclic sieving phenomenon with respect to $F_P(q)$.

\section{Type $A$ BiCambrian Lattices}\label{Sec:BiCambrianA}
Our goal in this section is to prove \cref{Conj:biCamb} in type $A$. The Coxeter group of type $A_n$ is isomorphic to the symmetric group $S_{n+1}$, whose elements are the permutations of the set $[n+1]$. We view permutations both as bijections from $[n+1]$ to itself and as words in one-line notation. Let $s_i$ denote the transposition that swaps $i$ and $i+1$. An index $i\in[n]$ is called a \dfn{descent} (respectively, \dfn{ascent}) of a permutation $w$ if $w(i)>w(i+1)$ (respectively, $w(i)<w(i+1)$); abusing notation, we will often say that $w(i)>w(i+1)$ is the descent (respectively, that $w(i)<w(i+1)$ is the ascent). For example, we would say that $6>3$ is a descent of $416352$, while $3<5$ is an ascent of $416352$. A permutation in $S_{n+1}$ is join-irreducible (respectively, meet-irreducible) in the weak order if and only if it has exactly $1$ descent (respectively, ascent). 

An \dfn{inversion} of a permutation $w\in S_{n+1}$ is a pair $(i,j)$ such that $1\leq i<j\leq n+1$ and $w^{-1}(i)>w^{-1}(j)$. For $w,w'\in S_{n+1}$, we have $w\leq w'$ if and only if every inversion of $w$ is also an inversion of $w'$. 

Let $c=c_+c_-$, where $c_+=\prod_{i\in[n]\text{ even}}s_i$ and $c_-=\prod_{i\in[n]\text{ odd}}s_i$. Throughout this section, we let $\equiv$ denote the $c$-biCambrian congruence on $S_{n+1}$. We will work with the following explicit combinatorial description of $\equiv$. Suppose $w\in S_{n+1}$ has a descent $a>b$. If there exist integers $k,m$ of different parities such that $b<k<m<a$ and such that $k$ and $m$ appear on the same side of $a$ and $b$ in the one-line notation of $w$, then we can swap $a$ and $b$ to produce a new permutation $w'$ such that $w'\lessdot w$ in the weak order and such that $w\equiv w'$. We call this operation that swaps $a$ and $b$ a \dfn{valid swap}. We also use the term \emph{valid swap} to refer to the inverse operation that swaps $a$ and $b$ to go from $w'$ to $w$ (in this case, $k$ and $m$ still have opposite parities and lie on the same side of $a$ and $b$ in $w'$). It is straightforward to check that a valid swap can be performed on a permutation $v$ by swapping two numbers $a$ and $b$ with $a<b$ if and only if the following hold: 
\begin{itemize} 
\item The entries $a$ and $b$ form a descent or an ascent in $v$.
\item There exists an integer $k$ such that $a<k<k+1<b$ and such that $k$ and $k+1$ appear on the same side of $a$ and $b$ in $v$.
\end{itemize} 
The \dfn{$c$-biCambrian congruence} $\equiv$ is defined by saying $v\equiv v'$ if and only if there is a sequence of valid swaps that transforms $v$ into $v'$; see \cref{FigRowTam5}. This equivalence relation is in fact a lattice congruence, so we can consider the associated projection operators $\pi_\downarrow$ and $\pi^\uparrow$ defined in \cref{Subsec:Congruences}. The \dfn{$c$-biCambrian lattice} is the lattice quotient of the weak order on $S_{n+1}$ modulo $\equiv$. As discussed in \cref{Subsec:Congruences}, the $c$-biCambrian lattice is isomorphic to both $\pi_\downarrow(S_{n+1})$ and $\pi^\uparrow(S_{n+1})$. Note that a permutation $w$ is in $\pi_\downarrow(S_{n+1})$ if and only if there does not exist a descent $b>a$ of $w$ for which we can perform a valid swap. Similarly, $w$ is in $\pi^\uparrow(S_{n+1})$ if and only if there does not exist an ascent $a<b$ of $w$ for which we can perform a valid swap.

\begin{figure}[ht]
  \begin{center}\includegraphics[height=6.5cm]{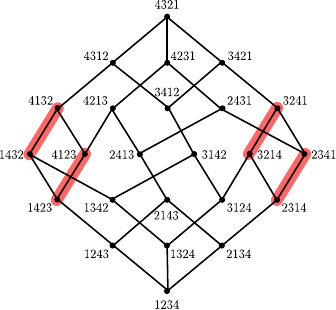}
  \end{center}
  \caption{There are $20$ $c$-biCambrian congruence classes in $S_4$; 16 of them are singletons while the other 4 (shown in red) have cardinality $2$.}\label{FigRowTam5}
\end{figure}

Consider a permutation $w$. Suppose $a$ and $b$ are integers that appear consecutively in the one-line notation of $w$ and satisfy $a<b$. Then either $b>a$ is a descent of $w$ or $a<b$ is an ascent of $w$. We say this descent or ascent is \dfn{right-even} (respectively, \dfn{left-even}) if $a$ and $b$ appear to the right (respectively, left) of all even numbers in $[a+1,b-1]$ and to the left (respectively, right) of all odd numbers in $[a+1,b-1]$ in $w$. Note that $b-a=1$ if and only if this descent or ascent is both right-even and left-even. It follows from the combinatorial description of the $c$-biCambrian congruence that $w\in\pi_\downarrow(S_{n+1})$ if and only if every descent of $w$ is right-even or left-even (or both).

Suppose $w\in\pi_\downarrow(S_{n+1})$ has $k$ descents $b_1>a_1,\ldots,b_k>a_k$; we will define two $k$-element subsets $S_\downarrow(w),T_\downarrow(w)\subseteq[n]$. If the descent $b_i>a_i$ is right-even, then we put $a_i$ into the set $S_\downarrow(w)$ and $b_i-1$ into the set $T_\downarrow(w)$; if the descent $b_i>a_i$ is left-even, we put $b_i-1$ into $S_\downarrow(w)$ and $a_i$ into $T_\downarrow(w)$. For example, if $w=896753421\in\pi_\downarrow(S_9)$, then $S_\downarrow(w)=\{1,3,4,5,6\}$ and $T_\downarrow(w)=\{1,2,3,6,8\}$. It is straightforward to check that no number will be added twice into one of these two sets from two different descents---thus, $(S_\downarrow(w),T_\downarrow(w))$ is indeed a pair of $k$-element subsets of $[n]$. Let $\Omega_n$ denote the set of all pairs $(S,T)$ such that $S,T\subseteq[n]$ and $|S|=|T|$. Barnard and Reading proved the following result.

\begin{theorem}[{\cite[Theorem~3.8]{BarnardReading}}]\label{thm:STbijectionpidown}
The map $\xi_\downarrow\colon\pi_\downarrow(S_{n+1})\to\Omega_n$ given by $\xi_\downarrow(w)=(S_\downarrow(w),T_\downarrow(w))$ is a bijection. 
\end{theorem}

Our general strategy for proving \cref{Conj:biCamb} for $A_n$ will be to use the bijection $\xi_\downarrow$ to transfer the operator $\row\colon\pi_\downarrow(S_{n+1})\to\pi_{\downarrow}(S_{n+1})$ to an operator $\chi\colon\Omega_n\to\Omega_n$ that is much easier to analyze (see \cref{thm:chi}). However, the main challenge is to prove that the map $\xi_\downarrow$ has the properties we need; this will occupy most of our focus in this section.  

Before we proceed, let us record an analogue of \cref{thm:STbijectionpidown} that concerns $\pi^\uparrow(S_{n+1})$. Given $w\in \pi^\uparrow(S_{n+1})$ with $k$ ascents $a_1<b_1, \ldots , a_k<b_k$, we define two $k$-element subsets $S^\uparrow(w),T^\uparrow(w)\subseteq[n]$. If the ascent $a_i<b_i$ is right-even, we put $a_i$ into the set $S^\uparrow(w)$ and $b_i-1$ into the set $T^\uparrow(w)$; if the ascent $a_i<b_i$ is left-even, we put $b_i-1$ into $S^\uparrow(w)$ and $a_i$ into $T^\uparrow(w)$. The next theorem follows from \cref{thm:STbijectionpidown} by considering duals, so we will attribute it to \cite{BarnardReading} even though it is not stated in this exact form in that article. 

\begin{theorem}[{\cite{BarnardReading}}]
The map $\xi^\uparrow\colon\pi^\uparrow(S_{n+1})\to\Omega_n$ given by $\xi^\uparrow(w)=(S^\uparrow(w),T^\uparrow(w))$ is a bijection. 
\end{theorem}

The next proposition tells us that an element $w\in\pi_\downarrow(S_{n+1})$ is join-irreducible if and only if $|S_\downarrow(w)|=|T_\downarrow(w)|=1$. The first statement is an immediate consequence of Proposition~4.11 (and the paragraph immediately after) in \cite{BarnardReading}, while the second statement follows from the dual version of that result. 

\begin{proposition}[{\cite[Proposition~4.11]{BarnardReading}}]\label{prop:1descent}
A permutation in $\pi_\downarrow(S_{n+1})$ is join-irreducible in $\pi_\downarrow(S_{n+1})$ if and only if it has exactly $1$ descent (i.e., it is join-irreducible in the weak order on $S_{n+1}$). A permutation in $\pi^\uparrow(S_{n+1})$ is meet-irreducible in $\pi^\uparrow(S_{n+1})$ if and only if it has exactly $1$ ascent (i.e., it is meet-irreducible in the weak order on $S_{n+1}$).
\end{proposition}

Suppose $w\in S_{n+1}$, and let $b>a$ be a descent of $w$. Let $\lambda(w,b>a)$ be the unique permutation in $S_{n+1}$ whose only descent is $b>a$ and such that the set of numbers in $[a+1,b-1]$ appearing to the left of $b$ in $\lambda(w,b>a)$ is the same as the set of numbers in $[a+1,b-1]$ appearing to the left of $b$ in $w$. For example, $\lambda(31746285,6>2)=13462578$. Note that if $b>a$ is right-even (respectively, left-even) in $w$, then $b>a$ is also right-even (respectively, left-even) in $\lambda(w,b>a)$. It is well known that if $b_1>a_1,\ldots,b_k>a_k$ are the descents of $w$, then $\{\lambda(w,b_1>a_1),\ldots,\lambda(w,b_k>a_k)\}$ is the canonical join representation of $w$ in the weak order (see \cref{subsec:distributive}). Proposition~4.11 (and the paragraph immediately after it) in \cite{BarnardReading} tells us that if $w\in\pi_\downarrow(S_{n+1})$, then the canonical join representation of $w$ in $\pi_\downarrow(S_{n+1})$ is the same as the canonical join representation of $w$ in $S_{n+1}$. We record this in the following proposition. 

\begin{proposition}\label{prop:canonicaljoinpidown}
Let $w\in\pi_\downarrow(S_{n+1})$. Let $b_1>a_1,\ldots,b_k>a_k$ be the descents of $w$. Then $\{\lambda(w,b_1>a_1),\ldots,\lambda(w,b_k>a_k)\}$ is the canonical join representation of $w$ in $\pi_\downarrow(S_{n+1})$.
\end{proposition}

Note that if $w\in \pi_\downarrow(S_{n+1})$ has descents $b_1>a_1,\ldots , b_k>a_k$, then 
\begin{equation}\label{Eq:union}
S_\downarrow(w)=\bigcup_{i\in[k]}S_\downarrow(\lambda(w,b_i>a_i))\quad\text{and}\quad T_\downarrow(w)=\bigcup_{i\in[k]}T_\downarrow(\lambda(w, b_i>a_i)).
\end{equation}

If $w\in S_{n+1}$ has an ascent $a<b$, then we let $\tau(w,a<b)$ be the unique permutation in $S_{n+1}$ whose only ascent is $a<b$ and such that the set of numbers in $[a+1,b-1]$ appearing to the left of $a$ in $\tau(w,a<b)$ is the same as the set of numbers in $[a+1,b-1]$ appearing to the left of $a$ in $w$. We have the following analogue of \cref{prop:canonicaljoinpidown}.

\begin{proposition}\label{prop:canonicalmeetpiup}
Let $w\in\pi^\uparrow(S_{n+1})$. Let $a_1<b_1, \ldots , a_k<b_k$ be the ascents of $w$. Then \linebreak $\{\tau(w,a_1<b_1),\ldots,\tau(w,a_k<b_k)\}$ is the canonical meet representation of $w$ in $\pi^\uparrow(S_{n+1})$.
\end{proposition}
 
Moreover, if $w\in \pi_\uparrow(S_{n+1})$ has ascents $a_1<b_1,\ldots , a_k<b_k$, then 
\begin{equation}\label{Eq:union2}
S^\uparrow(w)=\bigcup_{i\in[k]}S^\uparrow(\tau(w,a_i<b_i))\quad\text{and}\quad T^\uparrow(w)=\bigcup_{i\in[k]}T^\uparrow(\tau(w, a_i<b_i)).
\end{equation}

\begin{proposition}\label{cor:rowascentisdescent}
    Let $j$ be the join-irreducible element of $\pi_\downarrow(S_{n+1})$ with right-even (respectively, left-even) descent $b>a$. Let $m$ be the meet-irreducible element of $\pi^\uparrow(S_{n+1})$ with right-even (respectively, left-even) ascent $a<b$. Then $m=\pi^\uparrow(\row(j))$, where rowmotion is computed in $\pi_\downarrow(S_{n+1})$.
\end{proposition}

\begin{proof}
For $x\in S_{n+1}$, let $[x]_\equiv$ be the $c$-biCambrian congruence class containing $x$; thus, the equivalence classes of the form $[x]_\equiv$ are the elements of the $c$-biCambrian lattice $S_{n+1}/\equiv$. Let $j_*$ (respectively, $m^*$) be the unique element of $S_{n+1}$ that is covered by $j$ (respectively, covers $m$) in the weak order. One can check that $j\wedge m=j_*$ and $j\vee m=m^*$, where the join and meet are taken in the weak order. This implies that $[j]_\equiv\wedge[m]_\equiv=[j_*]_\equiv$ and $[j]_\equiv\vee[m]_\equiv=[m^*]_\equiv$, where the join and meet are taken in the $c$-biCambrian lattice. It follows from \cref{lem:propertiesofkappa} that $\kappa([j]_\equiv)=[m]_\equiv$. This proves that $\row([j]_\equiv)=[m]_\equiv$ in $S_{n+1}/\equiv$ (recall from \cref{subsec:distributive} that $\row$ is just $\kappa$ when restricted to join-irreducible elements). It follows that in the lattice $\pi_\downarrow(S_{n+1})$, rowmotion sends $j$ to the minimum element of $[m]_\equiv$. This yields the desired result. 
\end{proof}

We still need to establish a few more technical lemmas before we can transfer the action of rowmotion from $\pi_\downarrow(S_{n+1})$ to $\Omega_n$ via the map $\xi_\downarrow$. 

\begin{lemma}\label{lem:cannotswapatmosttwo}
    Suppose that $w,w'\in S_{n+1}$ are in the same $c$-biCambrian congruence class and that ${d_1,d_2\in [n+1]}$ satisfy $1\leq |d_1-d_2|\le 2$. Then $d_1$ appears to the left of $d_2$ in $w$ if and only if $d_1$ appears to the left of $d_2$ in $w'$.
\end{lemma}

\begin{proof}
    This is an immediate consequence of the combinatorial description of the $c$-biCambrian congruence since there is no valid swap of $d_1$ and $d_2$.
\end{proof}

In the following lemma, let us write $\Inv(w)\vert_{[a,b]}$ for the set of inversions $(x,y)$ of a permutation $w$ that satisfy $a=x<y\leq b$ or $a\leq x<y=b$.

\begin{lemma}\label{lem:relativeorderdoesnotchange}
    Let $w\in \pi_\downarrow(S_{n+1})$, and let $w'$ be a permutation in the same $c$-biCambrian congruence class as $w$. If $w$ has a descent $b>a$, then $\Inv(w)\vert_{[a,b]}=\Inv(w')\vert_{[a,b]}$. Analogously, if $\pi^\uparrow(w)$ has an ascent $d<e$, then $\Inv(\pi^\uparrow(w))\vert_{[d,e]}=\Inv(w')\vert_{[d,e]}$. 
\end{lemma}

\begin{proof}
    We will prove only the first statement since the proof of the second statement is similar. We will also assume that $b>a$ is a right-even descent of $w$; the case in which it is left-even is similar. Since $w\leq w'$, we know that $\Inv(w)\vert_{[a,b]}\subseteq\Inv(w')\vert_{[a,b]}$. Assume by way of contradiction that $\Inv(w)\vert_{[a,b]}\neq\Inv(w')\vert_{[a,b]}$. 
    
    Referring to the combinatorial description of the $c$-biCambrian congruence, we see that there is a chain $w=v_0\lessdot v_1\lessdot\cdots\lessdot v_m=w'$ in the weak order whose elements are all in the same $c$-biCambrian congruence class as $w$ and such that each $v_i$ is obtained from $v_{i-1}$ by applying a valid swap. Let $j$ be the smallest index such that $\Inv(v_{j-1})\vert_{[a,b]}\neq \Inv(v_{j})\vert_{[a,b]}$. Then $\Inv(v_{j-1})\vert_{[a,b]}= \Inv(w)\vert_{[a,b]}$. There is some $z\in[a+1,b-1]$ such that the valid swap used to move from $v_{j-1}$ to $v_j$ either swaps $a$ and $z$ or swaps $z$ and $b$; we will assume that it swaps $a$ and $z$ since the other case is similar. Because this swap is valid, there exists some $k\in[a+1,z-1]$ such that $k$ and $k+1$ appear on the same side of $a$ in $v_{j-1}$. However, since $\Inv(v_{j-1})\vert_{[a,b]}=\Inv(w)\vert_{[a,b]}$, this implies that $k$ and $k+1$ appear on the same side of $a$ in $w$. This contradicts the fact that $b>a$ is a right-even descent of~$w$. 
\end{proof}

\begin{lemma}\label{lem:cannotgofromdescenttoascent}
    Let $b$ be an even (respectively, odd) integer. Let $w\in \pi_\downarrow(S_{n+1})$ have a right-even (respectively, left-even) descent $b>a$. Then $\pi^\uparrow(w)$ cannot have a right-even (respectively, left-even) ascent of the form $b<d$.
\end{lemma}

\begin{proof}
Assume by way of contradiction that $b$ is even, $w$ has a right-even descent $b>a$, and $\pi^\uparrow(w)$ has a right-even ascent $b<d$; the proof of the other case is similar. Let $C$ be the $c$-biCambrian congruence class containing $w$. Because $(b,d)$ is not an inversion of $\pi^\uparrow(w)$, it is not an inversion of any element of $C$. There must be some permutations $v,v'\in C$ such that $v'$ covers $v$ in the weak order, $a<d$ is an ascent of $v$, and $v'$ is obtained from $v$ by performing a valid swap that swaps $a$ and $d$. However, it follows from \cref{lem:relativeorderdoesnotchange} that $a$ and $d$ are to the right of all even elements of $[a+1,d-1]$ (including $b$) and to the left of all odd elements of $[a+1,d-1]$ in $v$. This means that swapping $a$ and $d$ to reach $v'$ is actually not valid, which is a contradiction. 
\end{proof}

\begin{lemma}\label{lem:b+2condition}
    Let $b\le n-1$ be an integer. Suppose that either $w\in \pi_\downarrow(S_{n+1})$ has an ascent $a<b$ or $b$ is the first integer in $w$. Suppose also that $w$ has a descent $b+1>d$ such that $b+1$ appears to the right of $b$. Then either $\pi^\uparrow(w)$ has an ascent of the form $e<b+2$, or $b+1$ is to the left of $b+2$ in $\pi^\uparrow(w)$ and all integers appearing between $b+1$ and $b+2$ in $\pi^\uparrow(w)$ are larger than $b+2$.
\end{lemma}

\begin{proof}
    First we prove that all integers to the left of $b$ in $w$ are smaller than $b$. Assume for contradiction that there is an integer to the left of $b$ that is larger than $b$, and let $x$ be the rightmost such integer. There is a descent $x>y$, where $y$ is less than $b$ ($a$ and $y$ may coincide). Both $b$ and $b+1$ are to the right of $x$ and $y$, so the descent $x>y$ is not right-even or left-even; this contradicts the fact that $w\in\pi_\downarrow(S_{n+1})$.
    
    Now, \cref{lem:cannotswapatmosttwo} tells us that $b$ appears to the left of both $b+1$ and $b+2$ in $\pi^\uparrow(w)$. Because $(d,b+1)$ is an inversion of $w$, it is also an inversion of $\pi^\uparrow(w)$ (since $w\leq\pi^\uparrow(w)$); that is, $b+1$ appears to the left of $d$ in $\pi^\uparrow(w)$. If the lemma we are trying to prove is false, then we must have one of the following cases:
    \begin{itemize}
        \item \textbf{Case 1: } $b+2$ lies to the right of $b$ and to the left of $b+1$ in $\pi^\uparrow(w)$, and $\pi^\uparrow(w)$ has a descent of the form $d>b+2$.
        \item \textbf{Case 2: } $b+2$ lies to the right of $b+1$ in $\pi^\uparrow(w)$, and there is an integer $e$ between $b+1$ and $b+2$ in $\pi^\uparrow(w)$ that is smaller than $b$.
    \end{itemize}
    
    In Case 1, let $x_1$ be the rightmost integer to the left of $b+2$ in $\pi^\uparrow(w)$ that is less than $b+1$. Then there is an ascent $x_1<y_1$ in $\pi^\uparrow(w)$ with $y_1$ larger than $b+2$. But then, since $b+1$ and $b+2$ are both to the right of $y_1$, we can perform a valid swap by swapping $x_1$ and $y_1$; this contradicts the maximality of $\pi^\uparrow(w)$.
    
    In Case 2, let $x_2$ be the rightmost integer to the left of $b+2$ in $\pi^\uparrow(w)$ that is less than $b$. Then $x_2$ is to the right of $b+1$, and there is an ascent $x_2<y_2$ in $\pi^\uparrow(w)$. But then, since $b$ and $b+1$ are both to the left of $x_2$, we can perform a valid swap by swapping $x_2$ and $y_2$; this contradicts the maximality of $\pi^\uparrow(w)$.
\end{proof}

\begin{lemma}\label{lem:ascentusuallyleadstoascent}
    Let $w\in \pi_\downarrow(S_{n+1})$ have an ascent $a<b$, and suppose there exists some integer smaller than $a$ appearing to the right of $b$ in $w$. Then $\pi^\uparrow(w)$ must have an ascent of the form $a<d$. 
\end{lemma}

\begin{proof}
    Take $x$ to be the leftmost integer to the right of $b$ in $w$ that is smaller than $a$. Then there must be a descent $y>x$ in $w$, where $y$ is larger than $a$ ($b$ and $y$ may coincide). By \cref{lem:relativeorderdoesnotchange}, $a$ must appear to the left of $y$ in $\pi^\uparrow(w)$. We will reach $\pi^\uparrow(w)$ from $w$ by performing a particular sequence of valid swaps, each of which swaps the two integers in an ascent and results in a permutation in the same $c$-biCambrian congruence class. Throughout the sequence of valid swaps, let $M$ denote the consecutive subsequence beginning at $a$ and ending at $y$.
    
     We claim that whenever we swap an ascent $e<a$, we can choose to swap $e$ with every integer in $M$ from left to right until it is to the right of $y$ before moving on to perform other swaps. This is because we maintain throughout the swapping process that all the integers in $M$ are at least $a$ (except during the process of moving $e$ through the block $M$ from left to right). Indeed, we only add to the set of integers in $M$ when we swap $y$ with an integer that is larger than $y$ and to its right (which adds an integer larger than $a$ to $M$) or when we swap $a$ with an integer that is smaller than $a$ and to its left (which we immediately move through the subsequence $M$ until it exits to the right of $y$). When we move an integer $e$ smaller than $a$ through $M$, each valid swap results in a permutation in the same $c$-biCambrian congruence class. This is because the first swap of the ascent $e<a$ is valid, so there are integers $k,m$ of opposite parities on the same side of the ascent satisfying $e<k<m<a$. Since $M$ consists of integers at least $a$, this means that $k$ and $m$ are not in $M$---thus, we can choose these two integers to see that any swap of $e$ with an integer in $M$ is valid. It follows that $\pi^\uparrow(w)$ must have only integers larger than $a$ to the right of $a$ and to the left of $y$, which implies that $\pi^\uparrow(w)$ has an ascent of the form $a<d$.
\end{proof}

We are now in a position to state and prove our critical theorem that will allow us to describe rowmotion on type-$A$ biCambrian lattices. Recall that $s_i$ denotes the simple transposition in $S_{n+1}$ that swaps $i$ and $i+1$; when $i\in[n-1]$, we can also view $s_i$ as a transposition in $S_n$. Recall also that we defined a bipartite Coxeter element $c=c_+c_-$ in $S_{n+1}$. In the following theorem, we will need to use the analogues of $c_+$ and $c_-$ in $S_n$ instead of $S_{n+1}$. To avoid confusion, we denote these elements by $c_+'=\prod_{i\in[n-1]\text{ even}}s_i$ and $c_-'=\prod_{i\in[n-1]\text{ odd}}s_i$. For $\sigma\in S_n$ and $X\subseteq[n]$, we use the notation $\sigma X=\{\sigma(x):x\in X\}$. 

\begin{theorem}\label{thm:STupintermsofSTdown}
    If $w\in \pi_\downarrow(S_{n+1})$, then
    \[T^\uparrow(\pi^\uparrow(w))=[n]\setminus c_+'S_\downarrow(w),\quad S^\uparrow(\pi^\uparrow(w))=[n]\setminus c_-'T_\downarrow(w),\] \[T_\downarrow(w)=[n]\setminus c_-'S^\uparrow(\pi^\uparrow(w)),\quad\text{and}\quad
S_\downarrow(w)=[n]\setminus c_+'T^\uparrow(\pi^\uparrow(w)).\]
    \end{theorem}

\begin{proof}
    The latter two equalities follow immediately from the first two. We will only prove the first equality since the proof of the second is analogous. It suffices to prove each of the following statements:
    \begin{itemize}
        \item Exactly one of the following holds: $1\in S_\downarrow(w)$ or $1\in T^\uparrow(\pi^\uparrow(w))$.
        \item For even $k\in[2,n-1]$, exactly one of the following holds: $k\in S_\downarrow(w)$ or $k+1\in T^\uparrow(\pi^\uparrow(w))$.
        \item For even $k\in[2,n-1]$, exactly one of the following holds: $k+1\in S_\downarrow(w)$ or $k\in T^\uparrow(\pi^\uparrow(w))$.
        \item If $n$ is even, then exactly one of the following holds: $n\in S_\downarrow(w)$ or $n\in T^\uparrow(\pi^\uparrow(w))$.
    \end{itemize}
    All of these statements have similar proofs, so we will only prove the second statement. 
    
    First, let us assume by way of contradiction that $k\in S_\downarrow(\pi_\downarrow(w))$ and $k+1\in T^\uparrow(\pi^\uparrow(w))$. The condition $k\in S_\downarrow(w)$ is equivalent to one of the following mutually exclusive cases being true for $w$:
    
    \begin{itemize}
        \item \textbf{Case 1A: } There is a right-even descent $b_1>k$ for some $b_1\geq k+2$.
        \item \textbf{Case 1B: } There is a left-even descent $k+1>a_1$ for some $a_1$.
    \end{itemize}
    The condition $k+1\in T^\uparrow(\pi^\uparrow(w))$ is equivalent to one of the following mutually exclusive cases being true for $\pi^\uparrow(w)$:
    \begin{itemize}
        \item \textbf{Case 2A: } There is a right-even ascent $a_2<k+2$ for some $a_2\le k$.
        \item \textbf{Case 2B: } There is a left-even ascent $k+1<b_2$ for some $b_2$.
    \end{itemize}
    Recall that we have assumed $k$ is even. In Case~1A, the integers $k+2, k, k+1$ lie in that order in $w$, and in Case~1B, the integers $k+1, k$ lie in that order in $w$. In Case~2A, the integers $k, k+2, k+1$ lie in that order in $\pi^\uparrow(w)$, and in Case~2B, the integers $k+1, k+2$ lie in that order in $\pi^\uparrow(w)$. If Case~1A holds, then it follows from \cref{lem:cannotswapatmosttwo} that neither Case~2A nor Case~2B can hold. If Case~1B holds, then \cref{lem:cannotswapatmosttwo} tells us that Case~2A cannot hold, and  \cref{lem:cannotgofromdescenttoascent} tells us that Case~2B cannot hold. This exhausts all possibilities and yields our desired contradiction.
    
    Now we prove that we cannot have both $k\not\in S_\downarrow(w)$ and $k+1\not\in T^\uparrow(\pi^\uparrow(w))$ for even $k\in [2,n-1]$, which will complete the proof. Assume the contrary. The condition $k\not\in S_\downarrow(w)$ is equivalent to one of the following two mutually exclusive cases being true for $w$:
    \begin{itemize}
        \item \textbf{Case 3A: } There is a left-even descent $b_3>k$ for some $b_3\ge k+2$.
        \item \textbf{Case 3B: } Either there is an ascent $a_3<k$ for some $a_3$, or the first integer is $k$.
    \end{itemize}
    \emph{and} one of the following two mutually exclusive cases being true in $w$:
    \begin{itemize}
        \item \textbf{Case 3C: } There is a right-even descent $k+1>a_4$ for some $a_4\le k-1$.
        \item \textbf{Case 3D: } Either there is an ascent $k+1<b_4$ for some $b_4$, or the last integer is $k+1$.
    \end{itemize}
    Also, the condition $k+1\not\in T^\uparrow(\pi^\uparrow(w))$ is equivalent to one of the following two mutually exclusive cases being true for $\pi^\uparrow(w)$:
    \begin{itemize}
        \item \textbf{Case 4A: } There is a right-even ascent $k+1<b_5$ for some $b_5\ge k+3$.
        \item \textbf{Case 4B: } Either there is a descent $k+1>a_5$ for some $a_5$, or the last integer is $k+1$.
    \end{itemize}
    \emph{and} one of the following two mutually exclusive cases being true for $\pi^\uparrow(w)$:
    \begin{itemize}
        \item \textbf{Case 4C: } There is a left-even ascent $a_6<k+2$ for some $a_6\le k$.
        \item \textbf{Case 4D: } Either there is a descent $b_6>k+2$ for some $b_6$, or the first integer is $k+2$.
    \end{itemize}
    
    In Case 3A, the integer $k+1$ is to the left of $k$ and $k+2$ in $w$, and in Case 3C, the integers $k,k+1,k-1$ lie in that order in $w$. In Case 4A, the integers $k+2, k+1, k+3$ lie in that order in $\pi^\uparrow(w)$, and in Case 4C, the integer $k+1$ is to the left of $k$ and $k+2$ in $\pi^\uparrow(w)$. So it is not possible for either Cases 3A or 4C to hold at the same time as either Cases 3C or 4A.
    
    If Cases 3D and 4B both hold, then by \cref{lem:ascentusuallyleadstoascent}, we must have that all integers to the right of $k+1$ in $w$ are larger than $k+1$. But then Case 3A cannot also hold. An analogous argument implies that all integers to the right of $k+1$ in $\pi^\uparrow(w)$ are smaller than $k+1$, so Case 4C also cannot hold. To summarize, if Cases 3D and 4B both hold, then neither Case 3A nor Case 4C can hold.
    
    Now, assume Cases 3B and 3C both hold. By \cref{lem:b+2condition}, we know that either $\pi^\uparrow(w)$ has an ascent of the form $e<k+2$, or $k+1$ is to the left of $k+2$ and all integers between $k+1$ and $k+2$ are larger than $k+2$. If we additionally assume that Case 4D holds, the former possibility is ruled out. We also know that either Case 4A or Case 4B must hold, but both are contradicted by the latter possibility. That is, it is not possible for Cases 3B, 3C, and 4D to all hold. Analogously, we can prove that Cases 3B, 4A, and 4D cannot all hold simultaneously.
    
    Finally, consider when Cases 3B, 3D, 4B, and 4D all hold. By considering Cases 3D and 4B, from \cref{lem:ascentusuallyleadstoascent} we have that $k$ is to the left of $k+1$ in $w$. Analogously, we have that $k+2$ is to the left of $k+1$ in $\pi^\uparrow(w)$. Because a valid swap can only swap numbers that differ by at least $3$, we know that the numbers in the set $\{k,k+1,k+2\}$ appear in the same order in $w$ as in $\pi^\uparrow(w)$; this order is either $k, k+2, k+1$ or $k+2, k, k+1$. Without loss of generality, assume the order is $k, k+2, k+1$; we will show that Case 4D provides a contradiction (when the order is $k+2, k, k+1$, Case 3B provides an analogous contradiction). Then in $\pi^\uparrow(w)$, there is a descent $b_6>k+2$ to the right of $k$ and to the left of $k+1$. Let $x$ be the rightmost integer to the left of $k+2$ that is less than $k+1$. Then $\pi^\uparrow(w)$ has an ascent $x<y$ with $y$ larger than $k+2$. But then $k+1$ and $k+2$ both lie to the right of $x$ and $y$, so the ascent $x<y$ is not right-even or left even. This contradicts the fact that $\pi^\uparrow(w)\in\pi^\uparrow(S_{n+1})$. Hence, it is not possible for Cases 3B, 3D, 4B, and 4D to all hold at once.
    
    We have exhausted all $2^4=16$ possible scenarios, as shown in the table below. For an example of how to read this table, consider the entries in the first three columns in the first row. The entries in the first two columns assert that Cases 3A, 3C, 4A, and 4C cannot all hold simultaneously, and the entry in the third column asserts that the reason for this is that (as we have already established) Cases 3A and 3C cannot hold simultaneously. 

    \begin{table}[H]
        \resizebox{\textwidth}{!}{\begin{tabular}{|l|l|l||l|l|l||l|l|l||l|l|l|}
        \hline
        Case 3 & Case 4 & Reason & Case 3 & Case 4 & Reason     & Case 3 & Case 4 & Reason     & Case 3 & Case 4 & Reason         \\ \hline
        AC     & AC     & 3A, 3C & AD     & AC     & 3A, 4A     & BC     & AC     & 4A, 4C     & BD     & AC     & 4A, 4C         \\ \hline
        AC     & AD     & 3A, 3C & AD     & AD     & 3A, 4A     & BC     & AD     & 3B, 3C, 4D & BD     & AD     & 3B, 4A, 4D     \\ \hline
        AC     & BC     & 3A, 3C & AD     & BC     & 3A, 3D, 4B & BC     & BC     & 3C, 4C     & BD     & BC     & 3D, 4B, 4C     \\ \hline
        AC     & BD     & 3A, 3C & AD     & BD     & 3A, 3D, 4B & BC     & BD     & 3B, 3C, 4D & BD     & BD     & 3B, 3D, 4B, 4D \\ \hline
        \end{tabular}}
    \end{table}

    Thus, we conclude that exactly one of $k\in S_\downarrow(\pi_\downarrow(w))$ and $k+1\in T^\uparrow(\pi^\uparrow(w))$ holds.
    \end{proof}

Recall that we write $\Omega_n$ for the set of all pairs $(S,T)$ such that $S,T\subseteq[n]$ and $|S|=|T|$. Recall also the bijection $\xi_\downarrow\colon\pi_\downarrow(S_{n+1})\to\Omega_n$ from \cref{thm:STbijectionpidown}. Define the operator $\chi:\Omega_n\to\Omega_n$ by $\chi(S,T)=([n]\setminus c_+'T,[n]\setminus c_-'S)$. 

\begin{theorem}\label{thm:chi}
For $w\in\pi_\downarrow(S_{n+1})$, we have $\xi_\downarrow(\row(w))=\chi(\xi_\downarrow(w))$. 
\end{theorem}

\begin{proof}
    Let $b_1>a_1, \ldots , b_k>a_k$ be the descents of $w$. By \cref{prop:canonicaljoinpidown}, the canonical join representation of $w$ in $\pi_\downarrow(S_{n+1})$ is $\mathcal{D}(w)=\{\lambda(w,b_i>a_i):1\leq i\leq k\}$. As discussed in \cref{subsec:distributive}, the canonical meet representation of $\row(w)$ in $\pi_\downarrow(S_{n+1})$ is \[\kappa(\mathcal{U}(\row(w)))=\kappa(\mathcal{D}(w))=\row(\mathcal{D}(w))=\{\row(\lambda(w,b_i>a_i)):1\leq i\leq k\}.\] Consequently, the canonical meet representation of $\pi^\uparrow(\row(w))$ in $\pi^\uparrow(S_{n+1})$ is \[\{\pi^\uparrow(\row(\lambda(w,b_i>a_i))):1\leq i\leq k\}.\] \cref{thm:STupintermsofSTdown} and \eqref{Eq:union2} tell us that
\[
        S_\downarrow(\row(w))= [n]\setminus c_+'T^\uparrow(\pi^\uparrow(\row(w)))= [n]\setminus \bigcup_{i\in [k]} c_+'T^\uparrow(\pi^\uparrow(\row(\lambda(w, b_i>a_i)))).\]  For each $i\in[k]$, it follows from \cref{cor:rowascentisdescent} and the definitions of $T_\downarrow$ and $T^\uparrow$ that \[T^\uparrow(\pi^\uparrow(\row(\lambda(w,b_i>a_i))))=T_\downarrow(\lambda(w, b_i>a_i)).\] Hence, \[S_\downarrow(\row(w))=[n]\setminus\bigcup_{i\in[k]}c_+'T_\downarrow(\lambda(w,b_i>a_i))= [n]\setminus c_+'T_\downarrow(w),\]
where we have used \eqref{Eq:union}. A completely analogous argument shows that $T_\downarrow(\row(w))=[n]\setminus c_-'S_\downarrow(w)$ as well. Therefore, \[\xi_\downarrow(\row(w))=(S_\downarrow(\row(w)),T_\downarrow(\row(w)))=([n]\setminus c_+'T_\downarrow(w),[n]\setminus c_-'S_\downarrow(w))=\chi(\xi_\downarrow(w)). \qedhere\]
\end{proof}

The previous theorem shows that the operator $\chi\colon\Omega_n\to\Omega_n$ has the same orbit structure as rowmotion on the $c$-biCambrian lattice. To complete the proof of \cref{Conj:biCamb} in type $A$, we need to relate $\chi$ to rowmotion on $\mathcal J([n]\times[n])$, where we recall that $[n]\times[n]$ is the doubled root poset of type $A_n$. 

The elements of the rectangle poset $[n]\times [n]$ are ordered pairs $(x,y)$ with $x,y\in[n]$. Let $\{0,1\}^{2n}_n$ denote the set of words in $\{0,1\}^{2n}$ that have exactly $n$ occurrences of $0$ and exactly $n$ occurrences of $1$. Given an order ideal $I$ of $[n]\times [n]$, let $\max(I)$ denote the set of maximal elements of $I$. If $\max(I)=\{(x_1,y_1),\ldots,(x_k,y_k)\}$, then the \dfn{Stanley--Thomas word} of $I$ is the word $\ST(I)=u_1\cdots u_n v_{1}\cdots v_n\in\{0,1\}_n^{2n}$ defined by 
\[u_j=\begin{cases} 1, & \mbox{if $j\in\{x_1,\ldots,x_k\}$} \\ 0, & \mbox{if $j\not\in\{x_1,\ldots,x_k\}$}\end{cases}\quad\text{and}\quad v_j=\begin{cases} 0, & \mbox{if $j\in\{y_1,\ldots,y_k\}$} \\ 1, & \mbox{if $j\not\in\{y_1,\ldots,y_k\}$.}\end{cases}\] For example, if $I$ is the order ideal of $[5]\times [5]$ such that $\max(I)=\{(2,3),(4,2),(5,1)\}$, then $\ST(I)=0101100011$.

Let $\cyc\colon\{0,1\}_n^{2n}\to\{0,1\}_n^{2n}$ denote the cyclic shift operator defined by \[\cyc(z_1z_2\cdots z_{2n})=z_2\cdots z_{2n}z_1.\] The following fundamental result, which was first described by Stanley in \cite{StanleyPromotion} and later discussed further by Propp and Roby in \cite{ProppRoby}, relates rowmotion on $\mathcal J([n]\times [n])$ to this cyclic shift operator. 

\begin{theorem}[{\cite[Proposition~26]{ProppRoby}}]\label{thm:Stanley-Thomas}
The map $\ST\colon\mathcal J([n]\times [n])\to\{0,1\}_n^{2n}$ is a bijection such that $\ST({\row}(I))=\cyc(\ST(I))$ for all $I\in\mathcal J([n]\times [n])$.     
\end{theorem}

We can now prove \cref{Conj:biCamb} in type $A$.

\begin{proposition}\label{thm:typeAconjecture}
    The orbit structure of the operator $\row\colon\pi_\downarrow(S_{n+1})\to \pi_\downarrow(S_{n+1})$ is the same as that of ${\row}\colon\mathcal J([n]\times[n])\to \mathcal J([n]\times[n])$. 
\end{proposition}

\begin{proof}
\cref{thm:chi} tells us that $\row\colon\pi_\downarrow(S_{n+1})\to \pi_\downarrow(S_{n+1})$ has the same orbit structure as $\chi\colon\Omega_n\to\Omega_n$, and \cref{thm:Stanley-Thomas} tells us that ${\row}\colon\mathcal J([n]\times[n])\to \mathcal J([n]\times[n])$ has the same orbit structure as $\cyc\colon\{0,1\}_n^{2n}\to\{0,1\}_n^{2n}$. Therefore, it suffices to find a bijection $G\colon\Omega_n\to\{0,1\}_n^{2n}$ such that $\cyc\circ G=G\circ\chi$. 
    
Given $X\subseteq[n]$ and $i\in[n]$, let $X\langle i\rangle=0$ if $i\not\in X$, and let $X\langle i\rangle=1$ if $i\in X$. Note that if $\sigma\in S_n$, then $(\sigma X)\langle i\rangle=X\langle\sigma^{-1}(i)\rangle$. Let $(p_1,\ldots,p_n)$ be the $n$-tuple obtained by listing the odd elements of $[n]$ in increasing order and then listing the even elements of $[n]$ in decreasing order. For example, if $n=7$, then $(p_1,\ldots,p_7)=(1,3,5,7,6,4,2)$. This $n$-tuple is chosen so that $c_+'c_-'p_{j}=p_{j+1}$ for all $j\in[n]$ (with indices taken modulo $n$). Given $(S,T)\in\Omega_n$, let \[G(S,T)=T^*\langle p_n\rangle S\langle p_n\rangle T^*\langle p_{n-1}\rangle S\langle p_{n-1}\rangle\cdots T^*\langle p_1\rangle S\langle p_1\rangle,\] where $T^*=[n]\setminus c_-'T$. We claim that this defines our desired bijection $G\colon\Omega_n\to\{0,1\}_n^{2n}$. 

We first need to check that $G(S,T)$ is actually in $\{0,1\}_n^{2n}$, meaning that it has exactly $n$ occurrences of $0$ and exactly $n$ occurrences of $1$. This is immediate from the definition of $G$ since $|T^*|=n-|T|=n-|S|$ (by the definition of $\Omega_n$). We can easily recover $T$ from $T^*$ (since $T=[n]\setminus c_-'T^*$), so $G$ is evidently injective. We have $|\Omega_n|=|\{0,1\}_n^{2n}|=\binom{2n}{n}$, so $G$ must be a bijection.   

We are left to prove that $\cyc\circ G=G\circ\chi$. Consider $(S,T)\in\Omega_n$, and let $i\in[n]$. 
For $x\in\{0,1\}_n^{2n}$ and $k\in[2n]$, let $x_k$ denote the $k$-th letter in $x$, where indices are taken modulo $2n$ (so $x_{2n+1}=x_1$ and $x_0=x_{2n}$). We have \[\cyc(G(S,T))_{2i}=G(S,T)_{2i+1}=T^*\langle p_{n-i}\rangle=([n]\setminus c_-'T)\langle p_{n-i}\rangle=1-(c_-'T)\langle p_{n-i}\rangle.\] Since $c_-'$ and $c_+'$ are involutions and $c_-'(p_{n-i})=c_+'(c_+'c_-'(p_{n-i}))=c_+'(p_{n-i+1})$, we have \[1-(c_-'T)\langle p_{n-i}\rangle = 1-T\langle c_-'(p_{n-i})\rangle=1-T\langle c_+'(p_{n-i+1})\rangle=1-(c_+'T)\langle p_{n-i+1}\rangle = ([n]\backslash c_+'T)\langle p_{n-i+1}\rangle\]\[=G(\chi(S,T))_{2i}.\] Thus, $\cyc(G(S,T))$ and $G(\chi(S,T))$ have the same $(2i)$-th letter. We also have \[\cyc(G(S,T))_{2i-1}=G(S,T)_{2i}=S\langle p_{n-i+1}\rangle=S^{**}\langle p_{n-i+1}\rangle = G([n]\backslash c_+'T,S^*)_{2i-1}=G(\chi(S,T))_{2n-i},\] so $\cyc(G(S,T))$ and $G(\chi(S,T))$ have the same $(2i+1)$-th letter as well. We conclude that $\cyc(G(S,T))=G(\chi(S,T))$.
\end{proof}

\begin{example}
Suppose $n=7$. Let $S=\{2,3,6\}$ and $T=\{1,2,4\}$. Then $(S,T)\in\Omega_7$. We have $T^*=[7]\setminus c_-'T=\{4,5,6,7\}$, so \[G(S,T)=T^*\langle 2\rangle S\langle 2\rangle T^*\langle 4\rangle S\langle 4\rangle T^*\langle 6\rangle S\langle 6\rangle T^*\langle 7\rangle S\langle 7\rangle T^*\langle 5\rangle S\langle 5\rangle T^*\langle 3\rangle S\langle 3\rangle T^*\langle 1\rangle S\langle 1\rangle\] \[=01101110100100.\] Thus, $\cyc(G(S,T))=11011101001000$. Now, $\chi(S,T)=([7]\setminus c_+'T,[7]\setminus c_-'S)=(\widetilde S,\widetilde T)$, where $\widetilde S=\{2,4,6,7\}$ and $\widetilde T=\{2,3,6,7\}$. Then $\widetilde T^*=[7]\setminus c_-'\widetilde T=\{2,3,6\}=S$, so \[G(\chi(S,T))=\widetilde T^*\langle 2\rangle \widetilde S\langle 2\rangle \widetilde T^*\langle 4\rangle \widetilde S\langle 4\rangle \widetilde T^*\langle 6\rangle \widetilde S\langle 6\rangle \widetilde T^*\langle 7\rangle \widetilde S\langle 7\rangle \widetilde T^*\langle 5\rangle \widetilde S\langle 5\rangle \widetilde T^*\langle 3\rangle \widetilde S\langle 3\rangle \widetilde T^*\langle 1\rangle \widetilde S\langle 1\rangle\] \[=11011101001000=\cyc(G(S,T)).\qedhere\]
\end{example}

\section{Type $B$ BiCambrian Lattices}\label{Sec:BiCambrianB}
In this section, we prove \cref{Conj:biCamb} in type $B$. The general strategy is to deduce the result for type $B$ from the results derived in the preceding section for type $A$. 

Let $w_0=(2n)(2n-1)\cdots 1$ denote the longest element of the Coxeter group $S_{2n}$. There is a group automorphism $\beta\colon S_{2n}\to S_{2n}$ given by $\beta(x)=w_0xw_0$. Explicitly, $\beta(x)(i)=2n+1-x(2n+1-i)$ for all $1\leq i \leq 2n$. The Coxeter group $B_n$ is the $n$-th hyperoctahedral group, which can be realized as the subgroup of $S_{2n}$ consisting of all $x\in S_{2n}$ such that $\beta(x)=x$. 

Let $s_i$ denote the simple transposition in $S_{2n}$ that swaps $i$ and $i+1$. The simple generators of $B_n$ are $s_1^B,\ldots,s_n^B$, where $s_i^B=s_is_{2n-i}$ for $1\leq i\leq n-1$ and $s_n^B=s_n$. Let $c=c_+c_-$ be the bipartite Coxeter element of $S_{2n}$, where $c_+=\prod_{i\in[2n-1]\text{ even}}s_i$ and $c_-=\prod_{i\in[2n-1]\text{ odd}}s_i$. We can also write $c_+=\prod_{i\in[n]\text{ even}}s_i^B$ and $c_-=\prod_{i\in[n]\text{ odd}}s_i^B$, showing that $c$ can be viewed as a bipartite Coxeter element of $B_n$. Our goal is to understand rowmotion on the associated type-$B$ $c$-biCambrian lattice. 

As in the preceding section, we write $\pi_\downarrow\colon S_{2n}\to S_{2n}$ for the operator that sends a permutation to the minimal element of its type-$A$ $c$-biCambrian congruence class. We write $\pi_\downarrow^B\colon B_n\to B_n$ for the operator that sends an element of $B_n$ to the minimal element of its type-$B$ $c$-biCambrian congruence class.

It is well known that $\beta$ is a lattice automorphism of the weak order on $S_{2n}$. It follows that the weak order on $B_n$ is a sublattice of the weak order on $S_{2n}$. Recall that the $c$-biCambrian lattice of type $B$ is isomorphic to the set $\pi_\downarrow^B(B_n)$ under the weak order. It is straightforward to check that $\beta$ restricts to a bijection from $\pi_\downarrow(S_{2n})$ to itself. Since $\beta$ is an automorphism of the weak order on $S_{2n}$, it follows that $\beta\colon \pi_{\downarrow}(S_{2n})\to \pi_{\downarrow}(S_{2n})$ is a lattice automorphism. According to \cite[Proposition~3.15]{BarnardReading}, we have \begin{equation}\label{Eq:pidownB}
    \pi_\downarrow^B(B_n)=B_n\cap \pi_\downarrow(S_{2n}).
\end{equation} Therefore, the set of fixed points of $\beta\colon \pi_{\downarrow}(S_{2n})\to \pi_{\downarrow}(S_{2n})$ is $\pi_\downarrow^B(B_n)$; this readily implies that $\pi_\downarrow^B(B_n)$ is a sublattice of $\pi_\downarrow(S_{2n})$. Moreover, in the proof of the type-$B$ case of \cite[Theorem~2.12]{BarnardReading}, Barnard and Reading explained that the type-$B$ $c$-biCambrian congruence is simply the restriction of the type-$A$ $c$-biCambrian congruence to $B_n$. This means that if $y\in B_n$, then $y$ and $\pi_\downarrow^B(y)$ are in the same type-$A$ $c$-biCambrian congruence class, so $\pi_\downarrow(\pi_\downarrow^B(y))=\pi_\downarrow(y)$. But \eqref{Eq:pidownB} tells us that $\pi_\downarrow^B(y)\in\pi_\downarrow(S_{2n})$, so $\pi_\downarrow(\pi_\downarrow^B(y))=\pi_\downarrow^B(y)$. This proves that \begin{equation}\label{Eq:pidownB2}
\pi_\downarrow(y)=\pi_\downarrow^B(y)\text{ for all }y\in B_n.
\end{equation}

The Coxeter element $c$ is fixed by the automorphism $\beta$. As discussed in the proof of the type-$B$ case of \cite[Theorem~2.12]{BarnardReading}, this implies that $\beta$ preserves the type-$A$ $c$-biCambrian congruence on $S_{2n}$. In other words, 
\begin{equation}\label{Eq:pidownB3}
\pi_\downarrow(\beta(x))=\beta(\pi_\downarrow(x))\text{ for all }x\in S_{2n}.
\end{equation}

In what follows, let us write $\wedge$, $\wedge_{A}$, and $\wedge_{B}$ for the meet operations in the weak order on $S_{2n}$, in $\pi_\downarrow(S_{2n})$, and in $\pi_\downarrow^B(B_n)$, respectively. 

\begin{lemma}\label{lem:ybeta(y)}
If $y\in S_{2n}$, then $y\wedge\beta(y)\in B_n$, and $\pi_\downarrow^B(y\wedge\beta(y))=\pi_\downarrow(y)\wedge_{A}\pi_{\downarrow}(\beta(y))$. 
\end{lemma}

\begin{proof}
Since $\beta$ is an involutive lattice automorphism of the weak order on $S_{2n}$, we have \[\beta(y\wedge\beta(y))=\beta(y)\wedge\beta^2(y)=y\wedge \beta(y).\] This proves that $y\wedge\beta(y)\in B_n$. 

We have seen that $\beta$ is a lattice automorphism of $\pi_\downarrow(S_{2n})$. Therefore, it follows from \eqref{Eq:pidownB3} that \[\beta(\pi_\downarrow(y)\wedge_{A}\pi_\downarrow(\beta(y)))=\beta(\pi_\downarrow(y))\wedge_{A}\beta(\pi_\downarrow(\beta(y)))=\pi_\downarrow(\beta(y))\wedge_{A}\pi_\downarrow(\beta^2(y))=\pi_\downarrow(y)\wedge_{A}\pi_\downarrow(\beta(y)).\] This proves that $\pi_\downarrow(y)\wedge_{A}\pi_\downarrow(\beta(y))\in B_n$. Since $\pi_\downarrow(y)\wedge_{A}\pi_\downarrow(\beta(y))\in\pi_\downarrow(S_{2n})$, it follows from \eqref{Eq:pidownB} that $\pi_\downarrow(y)\wedge_{A}\pi_\downarrow(\beta(y))\in\pi_\downarrow^B(B_n)$. Because $\pi_\downarrow(y)\leq y$ and $\pi_\downarrow (\beta(y))\leq\beta(y)$, we have \[\pi_\downarrow(y)\wedge_{A}\pi_\downarrow(\beta(y))\leq y\wedge \beta(y).\] The map $\pi_\downarrow^B$ is order-preserving, so \[\pi_\downarrow(y)\wedge_{A}\pi_\downarrow(\beta(y))=\pi_\downarrow^B(\pi_\downarrow(y)\wedge_{A}\pi_\downarrow(\beta(y)))\leq \pi_\downarrow^B(y\wedge \beta(y)).\] To prove the reverse inequality, observe that \[\pi_\downarrow^B(y\wedge\beta(y))=\pi_\downarrow(y\wedge\beta(y))\leq\pi_\downarrow(y)\quad\text{and}\quad\pi_\downarrow^B(y\wedge\beta(y))=\pi_\downarrow(y\wedge\beta(y))\leq\pi_\downarrow(\beta(y));\] this implies that $\pi_\downarrow^B(y\wedge \beta(y))\leq \pi_\downarrow(y)\wedge_{A}\pi_\downarrow(\beta(y))$.  
\end{proof}

We are going to make use of the pop-stack sorting operator defined in \cref{subsec:distributive}. To avoid confusion, let us write $\Pop_A$ for the pop-stack sorting operator on $\pi_\downarrow(S_{2n})$ and $\Pop_B$ for the 
pop-stack sorting operator on $\pi_\downarrow^B(B_n)$. 

\begin{proposition}\label{prop:pop}
    If $x\in \pi_{\downarrow}^B(B_n)$, then $\Pop_{A}(x)=\Pop_{B}(x)$.
\end{proposition}

\begin{proof}
If $x$ is the identity permutation, then $\Pop_{A}(x)=\Pop_{B}(x)=x$. Now assume $x$ is not the identity permutation. Let $y_1,\ldots,y_k$ be the elements covered by $x$ in the weak order on $S_{2n}$. The elements of $B_n$ covered by $x$ are the elements of the form $y_i\wedge\beta(y_i)$. The elements covered by $x$ in $\pi_\downarrow(S_{2n})$ are $\pi_\downarrow(y_1),\ldots,\pi_\downarrow(y_k)$, and the elements covered by $x$ in $\pi_\downarrow^B(B_n)$ are those of the form $\pi_\downarrow^B(y_i\wedge\beta(y_i))$. Using \cref{lem:ybeta(y)} and the fact that $\pi_\downarrow^B(B_n)$ is a sublattice of $\pi_\downarrow(S_{2n})$, we find that 
\[\Pop_{B}(x)=\bigwedge\nolimits_{B}\{\pi_\downarrow^B(y_i\wedge\beta(y_i)):1\leq i \leq k\}=\bigwedge\nolimits_{A}\{\pi_\downarrow(y_i)\wedge_{A}\pi_\downarrow(\beta(y_i)):1\leq i \leq k\}.\] Because $x\in B_n$, we have $\beta(y_i)\in\{y_1,\ldots,y_k\}$ for each $i$. Thus, \[\bigwedge\nolimits_{A}\{\pi_\downarrow(y_i)\wedge_{A}\pi_\downarrow(\beta(y_i)):1\leq i \leq k\}=\bigwedge\nolimits_{A}\{\pi_\downarrow(y_i):1\leq i \leq k\}=\Pop_{A}(x). \qedhere\]
\end{proof}

In what follows, we write $\row_B$ for rowmotion on the lattice $\pi_\downarrow^B(B_n)$, and we write $\row_A$ for rowmotion on $\pi_\downarrow(S_{2n})$. 

\begin{proposition}\label{prop:rowArowB}
    If $x\in\pi_\downarrow^B(B_n)$, then $\row_A(x)=\row_B(x)$. 
\end{proposition}

\begin{proof}
Since $\beta$ is a lattice automorphism of $\pi_\downarrow(S_{2n})$, it commutes with $\row_A$. Consequently,  ${\row_A(x)=\row_A(\beta(x))=\beta(\row_A(x))}$. This shows that $\row_A(x)\in B_n$. We also know that $\row_A(x)\in\pi_\downarrow(S_{2n})$, so \eqref{Eq:pidownB} tells us that $\row_A(x)\in\pi_\downarrow^B(B_n)$. Since $\pi_\downarrow^B(B_n)$ is a sublattice of $\pi_\downarrow(S_{2n})$, we can use \cref{prop:pop1,prop:pop} to see that \[\row_A(x)\wedge_{B}x=\row_A(x)\wedge_{A}x=\Pop_{A}(x)=\Pop_{B}(x).\] \cref{prop:pop1} tells us that $\row_B(x)$ is the unique maximal element of the set \[\{z\in\pi_\downarrow^B(B_n):z\wedge_{B} x=\Pop_{B}(x)\}.\] We have shown that $\row_A(x)$ is in this set, so $\row_A(x)\leq\row_B(x)$. 

On the other hand, we have $\row_B(x)\in\pi_\downarrow(S_{2n})$, and we can use \cref{prop:pop1,prop:pop} once again to see that $\row_B(x)\wedge_{A}x=\row_B(x)\wedge_{B}x=\Pop_{B}(x)=\Pop_{A}(x)$. This shows that $\row_B(x)$ is an element of the set $\{z\in\pi_\downarrow(S_{2n}):z\wedge_{A}x=\Pop_{A}(x)\}$. The unique maximal element of this set is $\row_A(x)$, so $\row_B(x)\leq\row_A(x)$. 
\end{proof}

Recall that the doubled root posets of types $A_{2n-1}$ and $B_n$ are the rectangle poset $[2n-1]\times[2n-1]$ and the shifted staircase with $n(2n-1)$ elements, respectively (see \cref{FigRowTam3}). Let $\SSS_n$ denote the shifted staircase with $n(2n-1)$ elements. We can view $\SSS_n$ as the subposet of $[2n-1]\times[2n-1]$ consisting of all elements $(x,y)$ with $x\leq y$. There is a natural embedding \[\iota\colon\mathcal J(\SSS_n)\to\mathcal J([2n-1]\times [2n-1])\] given by $\iota(I)=\{(x,y),(y,x):(x,y)\in I\}$. It is straightforward to check that $\iota$ commutes with rowmotion. In other words, $\row(\iota(I))=\iota(\row(I))$ for all $I\in\mathcal J(\SSS_n)$, where the first rowmotion is on the lattice $\mathcal J([2n-1]\times[2n-1])$ and the second is on $\mathcal J(\SSS_n)$. Thus, the orbit structure of rowmotion on $\SSS_n$ is the same as the orbit structure of rowmotion on the sublattice $\iota(\mathcal J(\SSS_n))$ of $\mathcal J([2n-1]\times [2n-1])$; note that this sublattice consists of all order ideals $I\in\mathcal J([2n-1]\times [2n-1])$ that are symmetric in the sense that $(x,y)\in I$ if and only if $(y,x)\in I$.   

We can now prove \cref{Conj:biCamb} in type $B$. 

\begin{proposition}\label{thm:conjecturetypeB}
    The orbit structure of the operator $\row_B\colon\pi_\downarrow^B(B_n)\to \pi_\downarrow^B(B_n)$ is the same as that of ${\row}\colon\mathcal J(\SSS_n)\to \mathcal J(\SSS_n)$. 
\end{proposition}

\begin{proof}
As discussed above, rowmotion on $\mathcal J(\SSS_n)$ has the same orbit structure as rowmotion on the sublattice $\iota(\mathcal J(\SSS_n))$ of $\mathcal J([2n-1]\times[2n-1])$. Recall from \cref{thm:Stanley-Thomas} that we have a map $\ST$ that sends each order ideal $I\in\mathcal J([2n-1]\times[2n-1])$ to its Stanley--Thomas word $\ST(I)\in\{0,1\}_{2n-1}^{2(2n-1)}$ and satisfies $\ST(\row(I))=\cyc(\ST(I))$. Let $\Xi_{2n-1}$ be the subset of $\{0,1\}_{2n-1}^{2(2n-1)}$ consisting of all words $u_1\cdots u_{2(2n-1)}$ such that $u_{i}=1-u_{i+2n-1}$ for all $1\leq i\leq 2(2n-1)$, where indices are taken modulo $2(2n-1)$. It is straightforward to check that $\Xi_{2n-1}$ is the image of the set $\iota(\mathcal J(\SSS_n))$ under $\ST$. Therefore, it suffices to show that the orbit structure of $\row_B\colon\pi_\downarrow^B(B_n)\to \pi_\downarrow^B(B_n)$ is the same as that of $\cyc\colon \Xi_{2n-1}\to\Xi_{2n-1}$.

\cref{thm:chi} and the proof of \cref{thm:typeAconjecture} provide bijections $\xi_\downarrow\colon\pi_\downarrow(S_{2n})\to\Omega_{2n-1}$ and $G\colon\Omega_{2n-1}\to\{0,1\}_{2n-1}^{2(2n-1)}$ such that $G(\xi_\downarrow(\row_A(x)))=\cyc(G(\xi_\downarrow(x)))$ for all $x\in\pi_\downarrow(S_{2n})$. We will show that the image of $\pi_\downarrow^B(B_n)$ under $G\circ\xi_\downarrow$ is $\Xi_{2n-1}$. \cref{prop:rowArowB} tells us that the map $\row_B\colon\pi_\downarrow^B(B_n)\to \pi_\downarrow^B(B_n)$ is just the restriction of $\row_A$ to $\pi_\downarrow^B(B_n)$, so this will complete the proof. 

Given a set $S\subseteq[2n-1]$, let $2n-S=\{2n-s:s\in S\}$. The proof of \cite[Theorem~3.18]{BarnardReading} asserts (using different notation) that the image of $\pi_\downarrow^B(B_n)$ under $\xi_\downarrow$ is the set of all pairs $(S,2n-S)$ such that $S\subseteq[2n-1]$ (this is also straightforward to verify directly). Therefore, we must check that $G$ maps the set of such pairs bijectively onto $\Xi_{2n-1}$. The number of pairs $(S,2n-S)$ with $S\subseteq [2n-1]$ is $2^{2n-1}$, which is also the size of $\Xi_{2n-1}$; since $G$ is injective, we just need to show that $G(S,2n-S)\in\Xi_{2n-1}$ for all $S\subseteq[2n-1]$. 

As in the proof of \cref{thm:typeAconjecture}, we let $(p_1,\ldots,p_{2n-1})$ be the tuple obtained by listing the odd elements of $[2n-1]$ in increasing order and then listing the even elements of $[2n-1]$ in decreasing order. Let $c_-'=\prod_{i\in[2n-2]\text{ odd}}s_i$. Given $T\subseteq [2n-1]$, let $T^*=[2n-1]\setminus c_-'T$. Recalling the definition of the map $G$, we find that \[G(S,2n-S)=(2n-S)^*\langle p_{2n-1}\rangle S\langle p_{2n-1}\rangle(2n-S)^*\langle p_{2n-2}\rangle S\langle p_{2n-2}\rangle\cdots (2n-S)^*\langle p_1\rangle S\langle p_1\rangle.\] To show that this word is in $\Xi_{2n-1}$, it suffices to check that $(2n-S)^*\langle p_i\rangle=1-S\langle p_{i-n+1}\rangle$ for all $1\leq i \leq 2n-1$ (with indices taken modulo $2n-1$). This is straightforward to verify directly. \end{proof}

\section{Completing the Proof of \cref{Conj:biCamb}}\label{Sec:BiCambrianOther}
In \cref{thm:typeAconjecture,thm:conjecturetypeB}, we resolved \cref{Conj:biCamb} for biCambrian lattices of types $A$ and $B$. The purpose of this brief section is to handle the remaining cases: types $H_3$ and $I_2(m)$. We have checked by computer that the conjecture holds in type $H_3$, so we only need to prove that it holds for $I_2(m)$. 

The Coxeter group $I_2(m)$ is the dihedral group of order $2m$. Let $s_1$ and $s_2$ denote the simple generators of $I_2(m)$. Let $[s_1,s_2]_k$ denote the alternating product $s_1s_2s_1\cdots$ of length $k$ that starts with $s_1$, and let $[s_2,s_1]_k$ denote the alternating product $s_2s_1s_2\cdots$ of length $k$ that starts with $s_2$. The weak order on $I_2(m)$ consists of the two disjoint $(m-1)$-chains \[[s_1,s_2]_1\lessdot[s_1,s_2]_2\lessdot\cdots\lessdot[s_1,s_2]_{m-1}\quad\text{and}\quad[s_2,s_1]_1\lessdot[s_2,s_1]_2\lessdot\cdots\lessdot[s_2,s_1]_{m-1}\] together with the identity element $e$ below both chains and the longest element $w_0$ above both chains.

Consider the bipartite Coxeter element $c=s_1s_2$. The $c$-biCambrian lattice is actually equal to the weak order on $I_2(m)$. (For instance, see \cite[Theorem~1.4]{BarnardReading}, which asserts that the size of the $c$-biCambrian lattice is $2m=|I_2(m)|$.) For $2\leq k\leq m-1$, we have $\row([s_1,s_2]_k)=[s_1,s_2]_{k-1}$ and $\row([s_2,s_1]_k)=[s_2,s_1]_{k-1}$. Furthermore, $\row(s_1)=[s_2,s_1]_{m-1}$, $\row(s_2)=[s_1,s_2]_{m-1}$, $\row(e)=w_0$, and $\row(w_0)=e$. This shows that rowmotion acting on the $c$-biCambrian lattice has one orbit of size $2$ and one orbit of size $2m-2$. See \cref{FigRowTam4}. 

The doubled root poset of type $I_2(m)$ consists of two chains $x_1\lessdot\cdots\lessdot x_{m-2}$ and $y_1\lessdot\cdots\lessdot y_{m-2}$ together with two incomparable elements $z_1$ and $z_2$ satisfying $x_{m-2}\lessdot z_1\lessdot y_1$ and $x_{m-2}\lessdot z_2\lessdot y_1$. It is straightforward to check that rowmotion acting on the lattice of order ideals of this doubled root poset has one orbit of size $2$ and one orbit of size $2m-2$. We illustrate this for $m=5$ in \cref{FigRowTam4}. This proves that \cref{Conj:biCamb} holds in type $I_2(m)$, which completes the proof of the full theorem.

\begin{figure}[ht]
  \begin{center}\raisebox{0.25cm}{\includegraphics[height=4cm]{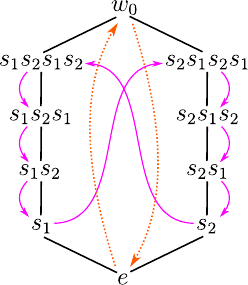}}\qquad\qquad\qquad\includegraphics[height=4.5cm]{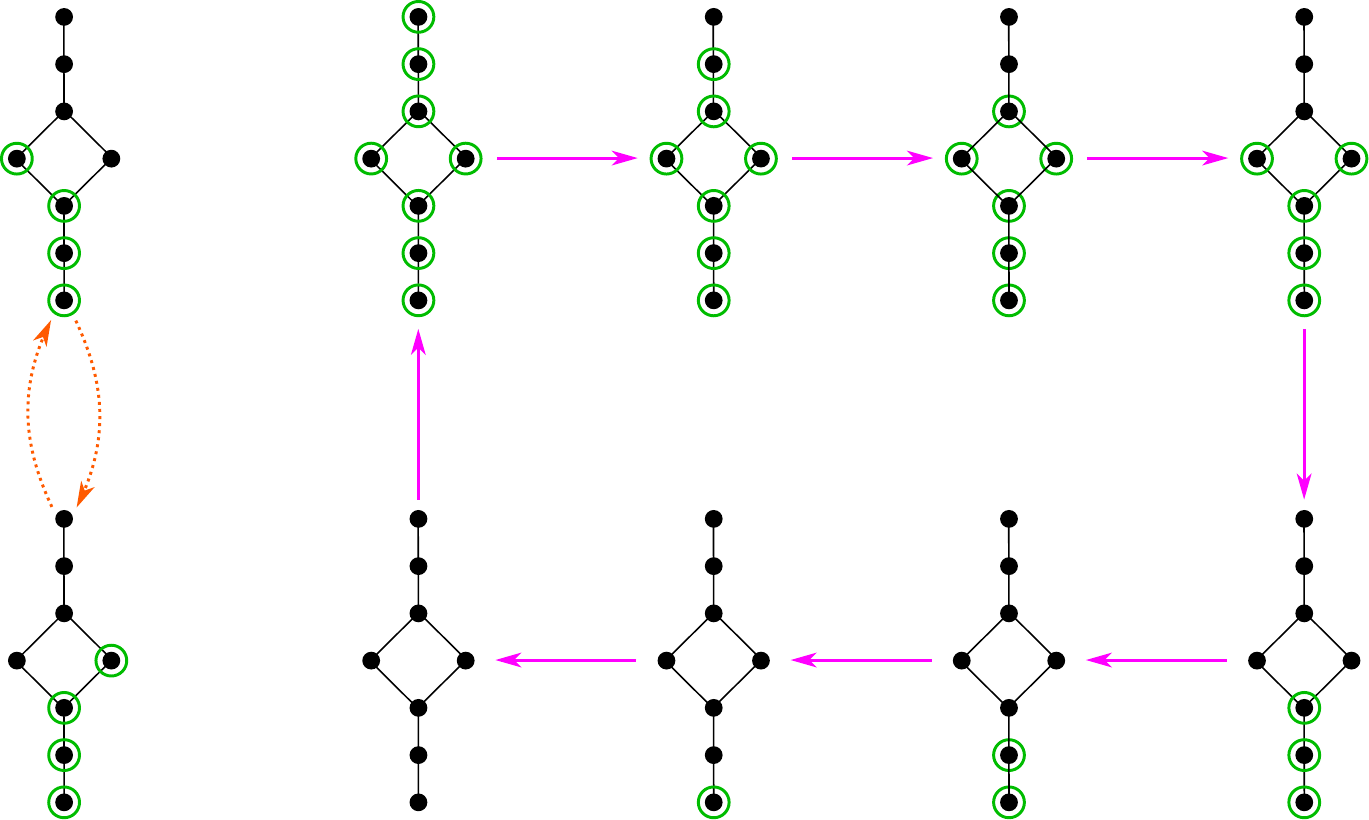}
  \end{center}
  \caption{On the left is the action of rowmotion on the type-$I_2(5)$ $c$-biCambrian lattice; this lattice is the same as the weak order on $I_2(5)$. On the right is the action of rowmotion on $\mathcal J(P)$, where $P$ is the doubled root poset of type $I_2(5)$. Each order ideal in $\mathcal J(P)$ is represented by elements that are circled in green. }\label{FigRowTam4}
\end{figure}

\section{Future Directions}\label{Sec:Conclusion}

It would be interesting to extend our results about rowmotion on $m$-Tamari lattices to the setting of the rational Tamari lattices $\Tam(a,b)$. In particular, recall \Cref{THING2,THING}, which state that $(\Tam(a,b),\row,\Cat^{(a,b)}(q))$ exhibits the cyclic sieving phenomenon and that the down-degree statistic on $\Tam(a,b)$ is homomesic for rowmotion with average $(a-1)(b-1)/(a+b-1)$.

In general, the down-degree statistic on a $\nu$-Tamari lattice need not be homomesic for rowmotion. However, data suggests that it is ``almost'' homomesic in the sense that the averages of $\ddeg$ on the orbits are all very close to each other. For example, if $\nu=\text{NE}^3\text{NE}^2\text{NE}^6\text{N}$, then rowmotion has orbits of sizes $7,7,12,18,18,18,80$; the average values of $\ddeg$ along these orbits are (approximately) $2.286, 2.286, 2.25, 2.278, 2.278, 2.278, 2.263$, respectively. It would be very interesting to have some explanation of why these values are so close. The following conjecture can be seen as predicting a new generalization of the homomesy phenomenon that we call \emph{asymptotic homomesy}. At the moment, this notion is fairly vague, and we do not wish to give a formal definition of it.

\begin{conjecture}
Let $\nu_0$ be a lattice path that has $a$ north steps and $b$ east steps. For each integer $t\geq 1$, let $\nu_0^t$ be the lattice path obtained by concatenating $\nu_0$ (viewed as a word over $\{\N,\E\}$) with itself $t$ times. Let $\Orb_{\row}(\Tam(\nu_0^t))$ be the set of orbits of $\row\colon\Tam(\nu_0^t)\to\Tam(\nu_0^t)$. We have \[\lim_{t\to\infty}\max_{\mathcal O\in \Orb_{\row}(\Tam(\nu_0^t))}\left\lvert\frac{1}{t}\cdot\frac{1}{|\mathcal O|}\sum_{x\in\mathcal O}\ddeg(x)-\frac{ab}{a+b}\right\rvert=0.\]
\end{conjecture}

Recall from \cref{Sec:Homomesy} that a statistic on a set $X$ is called \emph{homometric} for a function $f\colon X\to X$ if it always has the same sum over $f$-orbits of the same cardinality. 

\begin{conjecture}
Let $\nu$ be a lattice path. The down-degree statistic $\ddeg\colon\Tam(\nu)\to\mathbb R$ is homometric for the rowmotion operator $\row\colon\Tam(\nu)\to\Tam(\nu)$.
\end{conjecture}

Finally, we remark that it would be nice to have a more conceptual, type-independent proof of \cref{Conj:biCamb}. 

\section*{Acknowledgments}
This research was conducted at the University of Minnesota Duluth Mathematics REU and was supported in part by NSF-DMS Grant 1949884 and NSA Grant H98230-20-1-0009. Colin Defant was supported by the National Science Foundation under Award
No. DGE–1656466 and Award No. 2201907, by a Fannie and John Hertz Foundation Fellowship,
and by a Benjamin Peirce Fellowship at Harvard University. James Lin was supported by the MIT CYAN Mathematics Undergraduate Activities Fund. Any opinions, findings, and conclusions or recommendations expressed in this material are those of the authors and do not necessarily reflect the views of the National Science Foundation. We thank Joe Gallian, Bruce Sagan, and Nathan Williams for helpful comments. We thank the anonymous referees for reading carefully through our paper and making several useful comments.

\end{document}